\documentclass[english]{smfart}
\usepackage[marginratio=1:1,height=600pt,width=430pt]{geometry}
\usepackage{url}
\usepackage[english]{babel}
\usepackage{mathrsfs}
\usepackage{yfonts}
\usepackage{amssymb}
\usepackage{latexsym}
\usepackage{mathtools}
\usepackage{graphicx}
\usepackage{color}
\usepackage[all]{xy}
\usepackage{pstricks}
\usepackage{pst-plot}
\usepackage{epsfig}
\usepackage[hidelinks]{hyperref}

\theoremstyle{plain}
\newtheorem{thm}{Theorem}[section]
\newtheorem{conjecture}[thm]{Conjecture}
\newtheorem{propo}[thm]{Proposition}
\newtheorem{lem}[thm]{Lemma}
\newtheorem{cor}[thm]{Corollary}

\theoremstyle{definition}
\newtheorem*{def-unnumbered}{Definition}
\newtheorem{def-numbered}[thm]{Definition}
\theoremstyle{remark}
\newtheorem{rem}[thm]{Remark}

\definecolor{grey30}{rgb}{0.3,0.3,0.3}
\definecolor{grey67}{rgb}{0.67,0.67,0.67}
\definecolor{newcyan}{RGB}{105,164,216}
\definecolor{neworange}{RGB}{215,125,57}

\newcommand{\tif}{\text{if }}

\newcommand{\tand}{\text{ and }}

\newcommand{\tfor}{\text{for }}

\newcommand{\ga}{\alpha}
\newcommand{\gb}{\beta}

\newcommand{\gd}{\delta}
\newcommand{\gep}{\varepsilon}

\newcommand{\gj}{\theta}

\newcommand{\gl}{\lambda}
\newcommand{\gm}{\mu}

\newcommand{\gs}{\sigma}
\newcommand{\gr}{\rho}
\newcommand{\gt}{\tau}

\newcommand{\gx}{\xi}
\newcommand{\gy}{\psi}
\newcommand{\gz}{\zeta}

\newcommand{\gD}{\Delta}
\newcommand{\gF}{\Phi}

\newcommand{\gL}{\Lambda}
\newcommand{\gO}{\Omega}
\newcommand{\gP}{\Pi}

\newcommand{\gU}{\Upsilon}

\newcommand{\gY}{\Psi}
 
\newcommand{\C}[1]{{\mathcal{#1}}} 
\newcommand{\D}[1]{{\mathbb{#1}}} 
\newcommand{\E}[1]{{\mathscr{#1}}} 

\newcommand{\refS}[1]{Section~\ref{#1}} 

\newcommand{\refT}[1]{Theorem~\ref{#1}}
\newcommand{\refL}[1]{Lemma~\ref{#1}}

\newcommand{\refP}[1]{Proposition~\ref{#1}}
\newcommand{\refC}[1]{Corollary~\ref{#1}}
\newcommand{\refE}[1]{\eqref{#1}}
\newcommand{\refF}[1]{Figure~\ref{#1}}
\newcommand{\refR}[1]{Remark~\ref{#1}}

{\par \samepage}%
{\par}

\newcommand{\ol}{\overline}

\newcommand{\q}{\quad}

\newcommand{\rs}{\hat{\mathbb{C}}}
\newcommand{\ra}{\rightarrow}
\renewcommand{\Im}{\operatorname{Im}}
\renewcommand{\Re}{\operatorname{Re}}
   
\DeclareMathOperator{\Dom}{Dom}   
\newcommand{\cp}{\textup{cp}}
\newcommand{\cv}{\textup{cv}}
\newcommand{\IS}{\mathcal{I}\hspace{-1pt}\mathcal{S}}

\newcommand{\rr}{\operatorname{\mathcal{R}}}
\newcommand{\irr}{\mathrm{HT}_{N}}

\newcommand{\interior}{\operatorname{int}\,}
\newcommand{\ex}{\operatorname{\mathbb{E}xp}}

\newcommand{\eps}{\varepsilon}

\newcommand{\vfi}{\varphi}
\newcommand{\ea}{e^{2\pi \ga i}}

\newcommand{\co}[1]{^{\circ {#1}}}
\newcommand{\QIS}{\mathcal{Q}\hspace*{-.5pt}\mathcal{I}\hspace{-.5pt}\mathcal{S}}

\begin{document}
\title[Topology of irrationally indifferent attractors]
{Topology of irrationally indifferent attractors}
\alttitle{Topologie des attracteurs irrationnellement indifférents}
\author[Davoud Cheraghi]{Davoud Cheraghi}
\email[Davoud Cheraghi]{d.cheraghi@imperial.ac.uk}
\address[Davoud Cheraghi]{Department of Mathematics, Imperial College London, London SW7 2AZ, UK}
\date{\today}

\begin{abstract}
We study the post-critical set of a class of holomorphic maps with an irrationally indifferent fixed point. 
We prove a trichotomy for the topology of the post-critical set based on the arithmetic of the rotation number 
at the fixed point.
The only possibilities are Jordan curve, one-sided hairy Jordan curve, and Cantor bouquet. 
This explains the degeneration of the closed invariant curves inside the Siegel disks, as one varies the 
rotation number. 
\end{abstract}

\begin{altabstract}
Nous étudions l'ensemble post-critique d'une classe d'applications holomorphes avec un point fixe indifférent irrationnel.
Nous prouvons une trichotomie pour la topologie de l'ensemble post-critique basée sur l'arithmétique du 
nombre de rotation au point fixe.
Les seules options sont une courbe de Jordan, une courbe de Jordan velue unilat\'erale et un bouquet de Cantor. 
Cela explique la dégénérescence des courbes invariantes fermées à l'intérieur des disques de Siegel, lorsque l'on 
fait varier le nombre de rotation.
\end{altabstract}

\maketitle

\numberwithin{equation}{section}

\renewcommand{\thethm}{\Alph{thm}}
\section{Introduction}\label{S:intro}
\subsection{Irrationally indifferent attractors}
Let
\begin{equation}\label{E:form-maps}
f(z)=e^{2\pi i \ga} z+ O(z^2)
\end{equation}
be a germ of a holomorphic map defined near $0 \in \mathbb{C}$, 
with $\ga \in \D{R} \setminus \D{Q}$. 
The fixed point at $0$ is called \textbf{irrationally indifferent}. 
It is known that the local dynamics of $f$ near $0$ depends on the arithmetic nature of $\ga$ in a delicate 
fashion. 
By classical results of Siegel \cite{Sie42} and Brjuno \cite{Brj71}, if $\ga$ satisfies an arithmetic condition, 
now called \textbf{Brjuno type}, $f$ is conformally conjugate to the rotation by $2\pi \alpha$ near $0$.
The maximal domain of linearisation (conjugacy) is called the \textbf{Siegel disk} of $f$ at $0$, and is denoted 
by $\gD(f)$ here. 
Within $\gD(f)$, the local dynamics is trivial; any orbit in $\gD(f)$ is dense in an invariant 
analytic Jordan curve. 
On the other hand, in a remarkable development \cite{Yoc95}, Yoccoz showed that if $\ga$ is not a Brjuno 
number, the quadratic polynomial 
\[P_\alpha(z)= e^{2\pi i\ga} z+z^2\] 
is not linearisable at $0$. 
Despite that, Perez-Marco \cite{PM97} showed that there remains a non-trivial local invariant set at $0$. 
However, the topology of the local invariant set, and the local dynamics near $0$ remained mysterious, even for 
$P_\ga$. 
In this paper, for the first time, we explain the delicate topological structure of the (local) attractor, and the 
dynamics of the map on it. 

When $f$ is a polynomial or a rational function, the irrationally indifferent fixed point at $0$ influences the 
global dynamics of $f$. 
By the classical results \cite{Fat19,Man87}, there is at least a recurrent critical point of $f$ which ``interacts'' 
with the fixed point at $0$. For any such critical point $c_f$, we define 
\[\Lambda(c_f)= \overline{\textstyle{\bigcup_{i\geq 1}} f^{\circ i}(c_f)}.\] 
When $f$ is linearisable at $0$, the boundary of $\gD(f)$ is contained in $\Lambda(c_f)$, 
and when $f$ is not linearisable at $0$, then $0 \in \Lambda(c_f)$. 
The set $\Lambda(c_f)$ is part of the post-critical set of $f$, which is the closure of the orbits of all 
critical values of $f$. 
By a general result in holomorphic dynamics \cite{Ly83b}, unless the Julia set is equal to the whole 
Riemann sphere, for Lebesgue almost every $z$ in the Julia set of $f$, the spherical distance between $f\co{k}(z)$
and the post-critical set of $f$ tends to $0$ as $k \to \infty$. 

For ``badly approximable'' $\alpha$, $\Lambda(c_f)$ is well understood over the last four decades.
The main method is an ingenious surgery procedure, which is introduced by Douady \cite{Do86} for quadratic 
polynomials, Zakeri \cite{Za99} for cubic polynomials, Shishikura (unpublished work) for all polynomials, 
and Zhang \cite{Zha11} for all rational functions. 
Through the surgery, the problem is linked to the dynamics of analytic circle maps, where 
the works of Herman, Yoccoz and Swianek \cite{Her79,Yoc84b, Swi98} play a key role. 
The culmination of those works shows that when $\alpha$ is bounded type and $f$ is a rational function, 
$c_f \in \partial \gD(f)$ and $\Lambda(c_f)=\partial \gD(f)$ is a quasi-circle 
(a Jordan curve with controlled geometry). 
In \cite{McM98}, McMullen developed a renormalisation method to 
show that, among other features, when $\alpha$ is an algebraic number, $\Lambda(c_f)$ enjoys 
rescaling self-similarity at $c_p$. 
In a far reaching generalisation in the quadratic case, Petersen and Zakeri \cite{PZ04} 
employed trans quasi-conformal surgery to show that for almost every $\ga$, 
$c_f \in \partial \gD(f)$ and $\Lambda(c_{P_\alpha})=\partial \gD(P_\ga)$ is a David circle (a generalisation 
of quasi-circle).
 
For ``well-approximable'' $\ga$, the structure of $\Lambda(c_f)$ remained mostly mysterious, 
despite few sporadic surprising results such as \cite{Perez-Marco97,ABC04}.
Computer simulations suggest that large entries in the continued fraction of $\ga$ result in oscillations of 
the invariant curves in $\gD(f)$. 
The size of an entry and its location in the continued fraction of $\ga$, as well as non-linearity of large iterates of $f$,
result in an intricate oscillation of the invariant curves in $\gD(f)$.  
See \refF{F:Oscillations}. 
For $\ga$ with infinitely many extremely large entries, the consecutive oscillations may degenerate 
the closed invariant curves.  
For a class of maps, we explain the degeneration of invariant curves under perturbations of $\ga$. 

\begin{figure}[ht]
\begin{center}
\includegraphics[scale=.5]{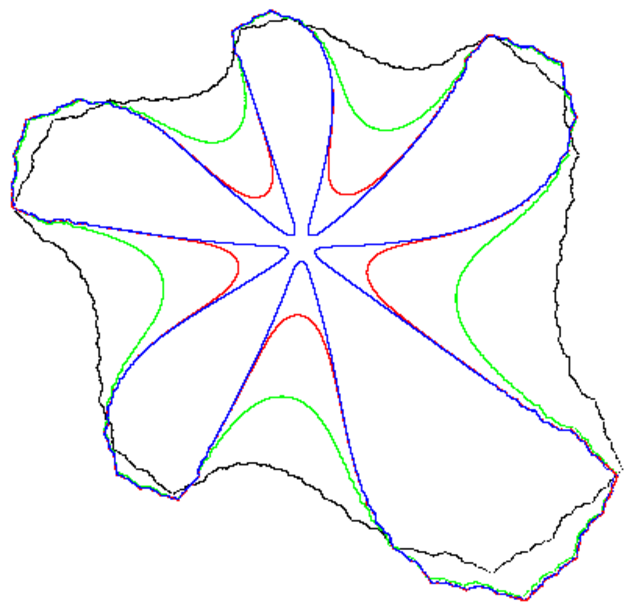} \hspace*{3em}
\includegraphics[scale=.5]{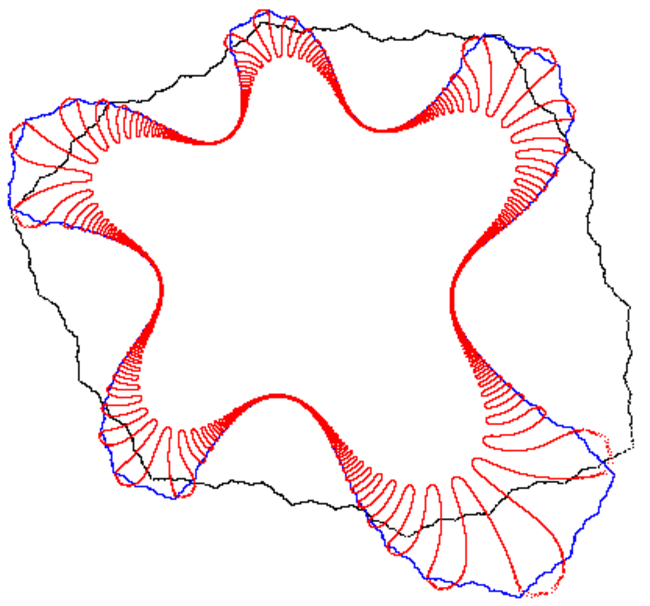}
\caption{Left image: computer simulations of the orbit of $c_{P_\alpha}$ for rotation numbers 
$\alpha=[2,2,\overline{2}]$, $[2,2,10^2, \overline{2}]$, $[2,2,10^4, \overline{2}]$, and $[2,2,10^8, \overline{2}]$. 
Right image: computer simulations of the orbit of $c_{P_\alpha}$ for 
$\alpha=[2,2,\overline{2}]$, $[2,2,10^2, \overline{2}]$, and $[2,2,10^2, 10^8, \overline{2}]$.}
\label{F:Oscillations}
\end{center}
\end{figure}

\subsection{Statements of the results}
Inou and Shishikura in \cite{IS06} introduced a sophisticated renormalisation scheme $(\C{F}, \C{R})$, 
where $\C{F}$ is an infinite dimensional class of maps as in \eqref{E:form-maps}, 
and $\C{R}: \C{F} \to \C{F}$ is a renormalisation operator.  
Every $f \in \C{F}$ has a certain covering structure, and a (preferred) critical point $c_f$. 
The set $\C{F}$ contains (the restriction to a neighbourhood of $0$ of) some polynomials and 
rational functions of arbitrarily large degrees. 
The scheme requires $\ga$ to be of sufficiently \textbf{high type}, that is, $\alpha$ belongs to the set 
\[\irr= \{\gep_0/(a_0+ \gep_1/(a_1+\gep_2/(a_2+\dots))) \mid  \forall n \geq 0, a_n\geq N, \gep_n=\pm 1\},\]
for a suitable $N$. 
In $\irr$, there are $\ga$ of bounded type, as well as $\ga$ with arbitrarily large entries. 

The scheme $(\C{F}, \C{R})$ was successfully employed by Inou and Shishikura to trap the orbit of 
$c_f$ in a dynamically defined neighbourhood of $0$. 
Moreover, they showed that the orbit of $c_f$ is infinite, there are no periodic points in 
$\Lambda(c_f)$, and in particular, $\Lambda(c_{P_\ga})$ is not equal to the Julia set of $P_\ga$. 

In \cite{Che13,Che19} we carried out a detailed quantitative analysis of the renormalisation scheme 
$(\C{F}, \C{R})$, and obtained fine estimates on the changes of coordinates which appear in the renormalisation. 
In \cite{Che23}, we built a toy model for the renormalisation of maps with an irrationally indifferent fixed point. 
We employ those methods to explain the delicate structure of $\Lambda(c_f)$. 

\begin{thm}[trichotomy of irrationally indifferent attractors]\label{T:trichotomy-main-thm}
There is $N\geq 2$ such that for every $\ga \in \irr$ and every $f(z)=e^{2\pi i \ga} z+ O(z^2)$ in the Inou-Shishikura 
class $\C{F}$, one of the following holds: 
\begin{itemize}
\item[(i)] $\ga$ is Herman type, and $\gL(c_f)$ is a Jordan curve enclosing $0$,
\item[(ii)] $\ga$ is Brjuno but not Herman type, and $\gL(c_f)$ is a one-sided hairy Jordan curve enclosing $0$,
\item[(iii)] $\ga$ is not Brjuno type, and $\gL(c_f)$ is a Cantor bouquet at $0$. 
\end{itemize}
The trichotomy also holds for the quadratic polynomials $P_\alpha$, when $\alpha \in \irr$. 
\end{thm}

The set of \textbf{Herman numbers} was discovered by Herman and Yoccoz \cite{Her79,Yoc95-ICM} 
in their landmark studies of the dynamics of analytic circle diffeomorphisms. 
In this paper we do not make any connections to circle maps --
the Brjuno and Herman types naturally come up. 
The set of Herman numbers is complicated to characterise in terms of the arithmetic of $\ga$; 
see \refS{S:arithmetic}. 
But we note that the set of Herman numbers is contained in the set of Brjuno numbers, and both 
sets have full Lebesgue measure in $\mathbb{R}$. 
However, the set of non-Brjuno numbers, and the set of Brjuno but not Herman numbers, 
are both uncountable and dense in $\D{R}$. 
Similarly, the set of $\ga$ corresponding to each of the cases in \refT{T:trichotomy-main-thm} is 
uncountable and dense in $\irr$. 

Cantor bouquets and hairy Jordan curves are universal topological objects like the Cantor sets; 
they are characterised by some topological axioms \cite{AaOv93}. 
Roughly speaking, a Cantor bouquet is a collection of arcs landing at a single point, such that 
every arc is accumulated from both sides by arcs in the collection. 
A (one-sided) hairy Jordan curve is a collection of arcs landing on a dense subset of a Jordan curve 
such that every arc in the collection is accumulated from both sides by arcs in the collection. 
Both sets have empty interior, and necessarily have complicated topologies; have uncountably 
many hairs and are not locally connected.
See \refS{SS:HCS-CB} for the definitions. 

In \refT{T:trichotomy-main-thm}, in cases (ii) and (iii), $c_f$ is an end point of a hair of $\Lambda(c_f)$. 
We also show that in case (iii), the arcs in $\Lambda(c_f)$ land at $0$ at well-defined (distinct) angles. 

\refT{T:trichotomy-main-thm} explains the degeneration of the boundaries of Siegel disks as one 
varies $\ga \in \irr$. 
Either the oscillations diminish and no degeneration occurs, or 
the oscillations build up and reach $0$ in the limit, collapsing onto uncountably many arcs landing 
at $0$ (case (iii)); or oscillations remain short of $0$, but collapse onto uncountably many arcs landing 
on a closed invariant curve (case (ii)). 
These phenomena also happen to the invariant curves within the Siegel disks, 
which give rise to a one-parameter family of closed invariant sets in $\gL(c_f)$, all with the same topology. 

\begin{thm}[degeneration of closed invariant curves]\label{T:degenerations-invariant-curves}
For every $\ga \in \irr$ there is $r_\alpha \geq 0$ such that for every $f(z)=e^{2\pi i \ga} z+ O(z^2)$ 
in the class $\mathcal{F}$, there is a map 
\[\phi_f : [0, r_\alpha] \to \{X \subseteq \Lambda(c_f) \mid X \text{ is non-empty, closed and invariant}\},\]
which is a homeomorphism with respect to the Hausdorff metric on the range. 
Moreover, 
\begin{itemize}
\item[(i)] $\phi_f$ is strictly increasing on $[0,r_\alpha]$, with respect to 
the inclusion in the range;
\item[(ii)] if $\ga$ is not a Brjuno number, $\phi_f(t)$ is a Cantor bouquet for every $t \in (0, r_\alpha]$, 
$\phi_f(0)=\{0\}$, and $\phi_f(r_\ga)= \gL(c_f)$; 
\item[(iii)] if $\ga$ is a Brjuno but not a Herman number, $\phi_f(t)$ is a hairy Jordan curve for all $t \in (0,r_\alpha]$, 
$\phi_f(0)$ is a Jordan curve, and $\phi_f(r_\ga)= \gL(c_f)$. 
\end{itemize}
\end{thm}

The above theorem, and the next results, all apply to the quadratic polynomials $P_\ga$, provided $\ga \in \irr$. 
Evidently, in \refT{T:degenerations-invariant-curves}, $r_\alpha=0$ when $\alpha$ is a Herman number, 
and otherwise $r_\alpha>0$. 
We also characterise the non-empty closed invariant subsets of $\Lambda(c_f)$. 

\begin{thm}[invariant sets in irrationally indifferent attractors]\label{T:dynamics-on-pc}
For every $\ga \in \irr$ and every $f(z)=e^{2\pi i \ga} z+ O(z^2)$ in $\mathcal{F}$, 
$f: \gL(c_f) \to \gL(c_f)$ is a topologically recurrent homeomorphism. 
Moreover, every non-empty closed invariant set in $\Lambda(c_f)$ is equal to the closure of the orbit of some $z \in \gL(c_f)$. 
\end{thm}

A partial result in the direction of \refT{T:trichotomy-main-thm} is obtained by Shishikura 
and Yang \cite{ShY18} around the same time. 
They prove that if $\alpha$ is a Brjuno number of high type, $\partial \gD(f)$ is a Jordan curve, 
and $c_f \in \partial \gD(f)$ iff $\alpha$ is a Herman number. 
These results also follow immediately from \refT{T:trichotomy-main-thm}.
In both cases (i) and (ii), the region inside the unique Jordan curve in $\Lambda(c_f)$ is invariant by $f$, 
and hence it must be $\gD(f)$. 
In case (i), since $c_f$ is recurrent, $c_f \in \gL(c_f)= \partial \gD(f)$, 
and in case (ii), $c_f \notin \partial \gD(f)$ (otherwise the orbit of $c_f$ remains in $\partial \gD(f)$).

\begin{cor}\label{C:jordan-curve}
For any Brjuno $\ga \in \irr$ and any $f(z)=e^{2\pi i \ga} z+ O(z^2)$ in $\mathcal{F}$, 
the boundary of the Siegel disk of $f$ at $0$ is a Jordan curve. 
\end{cor}

\begin{cor}\label{C:optimal-herman}
For any Brjuno $\ga \in \irr$ and any $f(z)=e^{2\pi i \ga} z+ O(z^2)$ in $\mathcal{F}$, 
the boundary of the Siegel disk of $f$ at $0$ contains a critical point of $f$ if and only if $\ga$ is a Herman number. 
\end{cor}

The above corollaries partially confirm conjectures of Herman and Douady on the Siegel 
disks of rational functions, \cite{Her85,Do86}. 
In \cite{Her85}, Herman employs a conformal welding argument of Ghys \cite{Ghy84} 
to show that if $\alpha$ is a Herman number, 
$c_{P_\ga} \in \partial \gD(P_\ga)$. 
On the other hand, Ghys and Herman \cite{Ghy84,Her86b} gave the first examples of 
polynomials having a Siegel disk with no critical point on the boundary.  
Based on these results, Herman conjectured in 1985 that \refC{C:optimal-herman} holds for all rational 
functions $f$ of degree $\geq 2$ and all irrational numbers $\alpha$. 
Using an elegant Schwarzian derivative argument, Graczyk and Swiatek in \cite{GrSw03} proved 
a general result, which implies in particular that if $f$ is a rational function or an entire function 
with degree $\geq 2$ and $\alpha$ is bounded type, there must be a critical point on the boundary of $\gD(f)$.
The result of Herman is extended to cubic polynomials in \cite{ChRo16}. 
We note that Shishikura and Yang in \cite{ShY18} have a fundamentally different approach to the proofs of 
Corollaries \ref{C:jordan-curve} and \ref{C:optimal-herman}. 
They work directly in the renormalisation tower of $f$. 
Those corollaries were not the main purpose of this paper, but a bi-product of studies towards explaining the 
global dynamics of non-linearisable maps.

In \cite{Che23} we conjectured that the trichotomy in \refT{T:trichotomy-main-thm}, as well as the dynamical 
features in Theorems \ref{T:degenerations-invariant-curves} and \ref{T:dynamics-on-pc} 
hold for all irrational numbers $\ga$ and all rational functions $f$. 
In particular, the conjectures of Douady and Herman follow from the conjecture on the trichotomy of the 
irrationally indifferent attractors.
See \refS{SS:methodology} for some justification of our conjectures. 

By a general result of Perez-Marco  \cite{PM97}, the invariant sets within Siegel disks do not disappear 
under perturbations of $\ga$.  
That is, for every $f$ as in \eqref{E:form-maps}, there are non-trivial, compact, connected, invariant sets 
containing $0$, called Siegel compacta or hedgehog. 
They built examples of non-linearisable $f$ with interesting pathological behaviour, such as examples with no 
small cycles \cite{PM93}, and examples with uncountably many conformal symmetries \cite{PerezMarco95}. 
In \cite{Che11}, Cheritat uses similar methods to build an example with $\partial \gD(f)$ non-locally connected. 
In \cite{Bis16,Bi2008}, Biswas builds hedgehogs with empty interior but positive area, and also 
hedgehogs with Hausdorff dimension 1. 
Such behaviours are not expected for rational functions. 
Indeed, by \cite{Man87}, when $f$ is a polynomial or a rational function, every Hedgehog of $f$ is 
contained in the post-critical set of $f$. This also holds for maps in $\C{F}$, \cite{AC18}. 
By explaining the topology of the post-critical set, we have shown that those pathological behaviours do not 
occur for many classes of rational functions. 

The renormalisation scheme $(\C{F}, \C{R})$ is a sophisticated but powerful tool. 
It has helped with making substantial progress in our understanding of the dynamics of maps with 
an irrationally indifferent fixed point. 
In \cite{BC12}, this was employed to prove the upper simi-continuity of $\Lambda(c_f)$ at every $f \in \C{F}$ 
with bounded type rotation number $\ga$. 
That was a main ingredient in the remarkable work of Buff and Cheritat on the existence of $P_\alpha$ with 
positive area Julia sets. 
In \cite{Che13,Che19}, we proved that $\Lambda(c_f)$ has zero area, and depends upper simi-continuously at every 
$f \in \C{F}$. 
Combining those results with the current paper, we now fully understand the topological behaviour of typical orbits of 
$P_\ga$, for all $\ga \in \irr$. That is, for almost every $z$ in the Julia set of $P_\ga$, 
the set of accumulation points of the orbit of $z$ is equal to $\gL(c_{P_\ga})$. 
In particular, the basin of attraction of any closed invariant set strictly contained in $\gL(c_{P_\ga})$ 
has zero area.
In \cite{AC18}, the statistical behaviour of the orbits is explained by showing that 
$f: \Lambda(c_f) \to \Lambda(c_f)$ is uniquely ergodic. 
On the other hand, in \cite{CC15}, the Marmi-Mousa-Yoccoz conjecture, which provides a fine estimate on the sizes 
of the Siegel disks in terms of the arithmetic of $\ga$, is confirmed for the class $\C{F}$.
The flexibility of the scheme $(\C{F}, \C{R})$ allows one to make perturbations of $\ga$ into 
complex numbers.  
In \cite{ChSh14}, the author and Shishikura prove the hyperbolicity of the renormalisation operator for satellite types, 
and conclude the local connectivity of the Mandelbrot set at some infinitely renormalisable parameters of satellite type. 
In \cite{AvLy22}, Avila and Lyubich combine the upper semi-continuity of $\gL(c_f)$ in \cite{Che19} with 
a random walk argument to prove the existence of Feigenbaum quadratic polynomials with positive area Julia sets. 

\subsection{Methodology}\label{SS:methodology}
We implement an alternative point of view on the use of renormalisation methods for the study of the dynamics.
By employing a toy model for the renormalisation scheme $(\C{F}, \C{R})$, we avoid cumbersome analytic arguments 
and error estimates via geometric constructions.

The toy model consists of a one-parameter family of maps 
$\{T_\ga: \hat{A}_\ga \to \hat{A}_\ga\}_{\ga \in \mathbb{R} \setminus \mathbb{Q}}$, 
where each $\hat{A}_\ga$ is star-like about $0$, 
and $T_\ga$ is a homeomorphism of the form $r e^{i\theta} \mapsto g(r,\theta) e^{i (\theta+2\pi \ga)}$ 
in the polar coordinate.
Moreover, $T_{\ga+1}= T_\ga$ and $\hat{A}_{\ga+1}= \hat{A}_\ga$, for all $\ga \in \mathbb{R}\setminus \mathbb{Q}$. 
The renormalisation is the action of the modular group, which sends $T_\ga: \hat{A}_\ga \to \hat{A}_\ga$ to 
$T_{-1/\ga}: \hat{A}_{-1/\ga} \to \hat{A}_{-1/\ga}$.

To present the alternative point of view, it is necessary to briefly describe how the toy model is built. 
It is convenient to carry out the construction in the $\log$ coordinate 
(where $T_\ga$ is unwrapped).
For each $\ga \in \mathbb{R} \setminus \mathbb{Q}$, we start with a sequence of change of coordinates 
$Y_n: \mathbb{H}_{-1} \to \mathbb{H}_{-1}$, for $n\geq 0$, where $\mathbb{H}_{-1}$ is the upper 
half-plane $\Im w > -1$, with each $Y_n$ not necessarily surjective. 
The maps $Y_n$ must satisfy two functional relations in order to be induced by successive 
applications of a renormalisation operator. Otherwise, there is considerable flexibility in choosing them.
The collection $(Y_n)_{n\geq 0}$ forms a non-autonomous dynamical system, as in \refF{F:renormalisation-diagrams}. 
The set of points in $\mathbb{H}_{-1}$ where the infinite composition 
$\dots \circ Y_2^{-1} \circ Y_1^{-1} \circ Y_0^{-1}$ is defined corresponds to $\hat{A}_\ga$. 
More precisely, it projects via $w \mapsto e^{2\pi i w}$ to define $\hat{A}_\ga$. 
Loosely speaking, $T_\ga$ corresponds to $Y_0 (Y_0^{-1}+1)$ via the same projection. 

The maps $Y_n$ are chosen to uniformly contract the Euclidean metric, 
and map half-infinite vertical lines in $\mathbb{H}_{-1}$ to half-infinite vertical lines; see \refF{F:model-Y_r}. 
Then, the non-autonomous system $(Y_n)_{n\geq 0}$ exhibits Markov behaviour, which is conveniently used 
to study the topology of $\hat{A}_\ga$ and the dynamics of $T_\ga$ on $\hat{A}_{\ga}$.  
The Brjuno and Herman types naturally come up in this setting. 

\begin{figure}[h]
\[\xymatrix{
& \mathbb{H}_{-1} \ar@{-->}[dl]_{e^{2\pi i w}} & \mathbb{H}_{-1} \ar[l]_{Y_0} & \mathbb{H}_{-1} \ar[l]_{Y_1} & 
\quad \cdots \quad \ar[l]_{Y_2} & \mathbb{H}_{-1} \ar[l]_{Y_{n-1}} & \mathbb{H}_{-1} \ar[l]_{Y_n} & 
\quad \cdots \ar[l]_{Y_{n+1}} \\
\hat{A}_{\ga} & \C{M}_{0} \ar@{-->}[dl]_{e^{2\pi i w}} \ar[u]_{U_0} & \C{M}_1 \ar[l]_{\gU_1} \ar[u]_{U_1} & \C{M}_2 \ar[l]_{\gU_1} \ar[u]_{U_2} & 
\quad \cdots \quad \ar[l]_{\gU_2} &  \C{M}_{n-1} \ar[l]_{\gU_{n-1}} \ar[u]_{U_{n-1}} & 
\C{M}_n \ar[l]_{\gU_n} \ar[u]_{U_n} & \quad \cdots \ar[l]_{\gU_{n+1}} \\
\gL(c_f) \cup \gD(f)\ar[u]^U }\]
\caption{Conjugacy of the non-autonomous dynamics of the changes of coordinates in two renormalisation schemes. 
The map $U_0$ projects to $U:\gL(c_f) \cup \gD(f) \to \hat{A}_\ga$ which conjugates $f$ on $\gL(c_f)$
to $T_\ga$ on $\partial \hat{A}_\ga$.} 
\label{F:renormalisation-diagrams}
\end{figure}

The operator $\C{R}$ is based on a geometric construction which involves a coordinate change 
(perturbed Fatou coordinate). 
Successive applications of $\C{R}$ at some $f\in \C{F}$ with $f'(0)=e^{2\pi i \ga}$ 
produces a sequence of changes of coordinate $\gU_n: \C{M}_n \to \C{M}_{n-1}$, for $n\geq 0$. 
When viewed in the $\log$ coordinate (where the dynamics is unwrapped), each $\C{M}_n$ is uniformly 
close to $\mathbb{H}_{-1}$, and each $\gU_n$ is either holomorphic or anti-holomorphic, 
with highly non-linear behaviour. 
Despite that, the maps $\gU_n$ are uniformly contracting with respect to suitable hyperbolic metrics on the regions 
$\C{M}_n$. 

The heart of the argument is to show that the non-autonomous dynamics of $(Y_n)_{n\geq 0}$ 
is conjugate to the non-autonomous dynamics of $(\gU_n)_{n\geq 0}$. 
That is, a collection of homeomorphisms $U_n$, as in \refF{F:renormalisation-diagrams}, so that the 
diagram is fully commutative along solid arrows. 
However, in order to induce a conjugacy between $f$ and $T_\ga$, $(U_n)_{n\geq 0}$ 
must satisfy some circular functional relations due to unwrapping the recurrent dynamics of the underlying maps, 
and also due to the presence of critical points in $(\C{F}, \C{R})$ but not in the toy model. 

To build such a conjugacy $(U_n)_{n\geq 0}$, we need to choose the maps $Y_n$ uniformly close to $\gU_n$. 
That requires some understanding of the coordinate changes $\gU_n$, obtained in \cite{Che19}.
It is worth noting that the detailed information about the locations and geometries of relevant pieces near $0$ 
obtained in \cite{IS06} were not utilised here.
More precisely, given a renormalisation scheme with the qualitative characteristics of $(\C{F}, \C{R})$, 
one only needs to verify \refP{P:gc-vs-Y-value}, \refP{P:asymptotics-similar}, and \refL{L:well-inside-lifts} 
about the changes of coordinates in that renormalisation. 
Hence, the trichotomy conjecture.

The Markov system $(Y_n)_{n\geq 0}$ provides partitions of the phase spaces $\mathbb{H}_{-1}$, 
shrinking to half-infinite vertical lines. 
We build corresponding partitions for the system $(\gU_n)_{n\geq 0}$, which are in the spirit of Yoccoz puzzle 
pieces with boundary markings. 
The first parametrised arc for the partitions is obtained from the limit of certain hyperbolic 
geodesics (suitably re-parametrised by the maps $Y_n$). 
This gives rise to a collection of curves, one in each $\C{M}_n$, which collectively enjoy equivariant 
properties with respect to the dynamics of $(\gU_n)_{n\geq 0}$ and $(Y_n)_{n\geq 0}$. 
By employing the circular functional relations imposed on the conjugacy $(U_n)_{n\geq 0}$, further 
parametrised arcs are built. 
Those arcs divide the sets $\C{M}_n$ into partitions (not necessarily shrinking to points). 

To complete the argument, we build a sequence of partial conjugacies by mapping the corresponding 
partition pieces one to another. 
Contrary to $\gU_n$, the maps $Y_n$ are not conformal, and their long compositions may degenerate the 
complex structure. 
This leads to the degeneration of the sequence of partial conjugacies. 
However, the degenerations only occur transverse to the hairs in $\gL(c_f)$, and the partial conjugacies 
form Cauchy sequences along the hairs. 

There are a number of advantages in explaining the dynamics of $f$ through a toy model. 
It allows one to study the delicate role of the arithmetic properties of $\ga$ in the simpler setting of the model, 
while only dealing with the highly distorting coordinate changes in another setting. 
The construction of the conjugacy does not require detailed understanding of the dynamics of the underlying maps, 
and simultaneously works rotation numbers of different type. 
Moreover, the toy model allows us to build puzzle partitions enjoying equivariant properties with respect to the 
renormalisation. 
This paves the way for further progress on the topic, see for instance, \cite{CDY20}. 

\setcounter{tocdepth}{2}
\tableofcontents

\renewcommand{\thethm}{\thesection.\arabic{thm}}

\section{Arithmetic classes of Brjuno and Herman}\label{S:arithmetic}
In this section we define the arithmetic classes of Brjuno and Herman. 
The definition requires the action of the modular group $\mathrm{PGL}(2, \D{Z})$ on $\D{R}$.
To study the action of this group, one may choose a fundamental interval for the action of $z \mapsto z+1$
and study the action of $z \mapsto 1/z$ on that interval. 
Due to a feature of the near-parabolic renormalisation, it is natural to work with the fundamental interval 
$(-1/2,1/2)$ for the translation. 
That is because, as we shall see in \refS{S:NP-renormalisation-scheme}, the scheme works for rotation numbers 
close to $0$.
This choice of the fundamental interval leads to a modified notion of continued fractions, which we briefly 
present below. 

\subsection{Modified continued fraction}
\label{SS:modified-fractions}
Let us fix an irrational number $\ga \in \D{R}$. 
For $x$ in $\D{R}$, define $d(x, \D{Z})= \min_{k\in \D{Z}} |x-k|$. 
Let $\ga_0=d(\ga, \D{Z})$, and for $n\geq 1$, inductively define the numbers 
\begin{equation}\label{E:rotations-rest}
\ga_{n+1}=d(1/\ga_n, \D{Z}). 
\end{equation} 
For our convenience in future formulae, we let $\ga_{-1}=+1$.
There are unique integers $a_n$, for $n\geq -1$, and $\gep_n \in \{+1, -1\}$, for $n\geq 0$, such that 
\begin{equation}\label{E:rotations-relations}
\ga= a_{-1}+ \gep_0 \ga_0 \quad \text{ and } \quad  1/\ga_n= a_n + \gep_{n+1} \ga_{n+1}. 
\end{equation}
Evidently, for all $n\geq 0$,
\begin{equation}
1/\ga_n \in (a_n-1/2, a_n+1/2), \quad a_n\geq 2, \quad \tand \quad 
\gep_{n+1}= 
\begin{cases}
+1 & \text{if } 1/\ga_n \in (a_n, a_n+1/2), \\
-1 & \text{if } 1/\ga_n \in (a_n-1/2, a_n).
\end{cases}
\end{equation}
The sequences $a_n$ and $\gep_n$ provide the infinite continued fraction 
\begin{equation*}\label{E:modified-expansion-alpha}
\ga=a_{-1}+\cfrac{\gep_0}{a_0+\cfrac{\gep_1}{a_1+\cfrac{\gep_2}{a_2+\dots}}}.
\end{equation*}
The best rational approximants of $\ga$, or the convergents of $\ga$, are defined as  
\[\frac{p_n}{q_n}=a_{-1}+\cfrac{\gep_0}{a_0+\cfrac{\gep_1}{\ddots + \cfrac{\gep_n}{a_n}}}, \;  \tfor n\geq -1.\] 
We assume that $p_n$ and $q_n$ are relatively prime, and $q_n >0$. 

\subsection{Brjuno numbers}\label{SS:Brjuno-numbers}
By a careful study of the Siegel's approach in \cite{Sie42}, Brjuno in \cite{Brj71} showed that the convergence 
of the infinite series $\sum_{n=-1}^{+\infty} q_n^{-1} \log q_{n+1}$ for $\alpha$ 
is sufficient for the analytic linearisation of the germs $f(z)=e^{2\pi i \ga} z+ O(z^2)$ near $0$.
Siegel-Brjuno approach is based on estimating the coefficients of the formal conjugacy, but do not involve any 
notion of renormalisation. 

Yoccoz in \cite{Yoc95} carried out a geometric approach to the linearisation problem based on a notion of 
renormalisation introduced by Douady-Ghys.  
Thanks to his work, a natural way to look into the Brjuno condition is through a function
which enjoys remarkable equivariant properties with respect to the action of $\mathrm{PGL}(2, \D{Z})$.
That is, the Brjuno function is defined (by Yoccoz) as 
\begin{equation}\label{E:Brjuno-condition}
\textstyle{ \C{B}(\ga)= \sum_{n=0}^\infty \gb_{n-1} \log \ga_n^{-1}}
\end{equation} 
where $\gb_n= \prod_{i=-1}^n \ga_i$, for $n \geq -1$. 
The function $\C{B}$ is defined on $\D{R} \setminus \D{Q}$, and takes values in $(0, +\infty]$.
It satisfies the remarkable relations 
\begin{equation}\label{E:Brjuno-functional-equations}
\begin{gathered}
\C{B}(\ga)= \C{B}(\ga+1)= \C{B}(-\ga), \; \tfor \ga \in \D{R}\setminus \D{Q}, \\
\C{B}(\ga)= \ga \C{B}(1/\ga)+ \log 1/\ga, \; \tfor \ga \in (0,1/2) \setminus \D{Q}.
\end{gathered}
\end{equation}
As shown by Yoccoz, $| \sum_{n=-1}^{+\infty} q_n^{-1} \log q_{n+1} - \C{B}(\ga) |$ is uniformly 
bounded from above independent of $\ga$. 
Thus, an irrational number $\ga$ is a \textbf{Brjuno number} iff $\C{B}(\ga)<+\infty$.

For a generic choice of $\ga \in \D{R}$, $\C{B}(\ga)=+\infty$.
The function $\C{B}$ is highly irregular. One may refer to \cite{MMY97,MMY01,JM18}, and the extensive 
list of references therein, for detailed analysis of the regularity properties of the Brjuno function. 
The irregularity features of the Brjuno function is somehow reflected in the sizes of the Siegel disks as a 
function of the rotation number, as suggested in \cite{MMY97} and studied in \cite{BC06b,CC15}.

\begin{rem}
Besides the quadratic polynomials, the optimality of the Brjuno condition has been 
(re)confirmed for several classes of maps in 
\cite{PM93,Pe2001,Ge01,BC04,Ok04,Ok2005,FMS2018,Che19}. 
But in its general form for polynomial and rational functions remains a significant challenge.  
For progress on the linearisation problem in higher dimensions one may refer to
\cite{Stolo2000,Ge2007,Rong2008,YoGr2008,Ra2010,BZ2013,GiLoSa2015},
for twist maps refer to \cite{BeGe2001,Po2010}, see also 
\cite{Carletti2000,Lindahl2004,MS2011} and the references therein.
\end{rem}

\subsection{Herman numbers}\label{SS:Herman-numbers}
The problem of analytic linearisation of orientation-preserving 
diffeomorphisms of the circle $\D{R}/\D{Z}$ was first successfully studied by Arnold  \cite{Ar61} 
for maps close to rigid rotations. 
In \cite{Her79}, Herman presented a systematic study of the problem, and presented a rather 
technical arithmetic condition which guaranteed the analytic linearisation.
Later, Yoccoz made improvements in the work of Herman and 
identified the optimal arithmetic condition for the linearisation, which he called Herman type, 
in honour of the work of Herman. 

In \cite{Yoc02}, Yoccoz gives several characterisations of the Herman numbers. 
We present the most relevant of those in our setting. 
To do that, we need to consider the functions $h_r: \D{R} \to (0, +\infty)$, for $r\in (0,1)$:
\[h_r(y)= 
\begin{cases}
r^{-1} (y- \log r^{-1} +1)  & \tif y \geq \log r^{-1}, \\
e^{y} & \tif y \leq \log r^{-1}.
\end{cases}
 \]
Each $h_r$ is a $C^1$ monotone map, which satisfies $y+1 \leq h_r(y) \leq e^y$ for all $y \in \D{R}$, and 
$h_r'(y)\geq 1$ for all $y\geq 0$. 
An irrational number $\ga$ is a \textbf{Herman number} if and only if for all $n\geq 0$ there is $m\geq n$ such that 
\[h_{\ga_{m-1}} \circ \dots \circ h_{\ga_n} (0)\geq \C{B}(\ga_{m}).\]
In the above definition, the composition $h_{\ga_{m-1}} \circ \dots \circ h_{\ga_n}$ is understood as the
identity map when $m=n$, and as $h_{\ga_n}$ when $m=n+1$. 

The arithmetic characterisation of the Herman numbers by Yoccoz in \cite{Yoc02} uses the standard 
continued fraction. 
That is, he works with the interval $(0,1)$ for the action of $z \mapsto z+1$. 
The above form of the Herman numbers in terms of the modified continued fractions is established  
in \cite{Che23}. 

One may see that every Herman number is a Brjuno number, the set of Herman numbers is invariant under the 
action of $\mathrm{PGL}(2,\D{Z})$, and every Diophantine number is of Herman type. 
In particular, the set of Herman numbers and the set of Brjuno numbers have full Lebesgue measure in $\D{R}$. 
On the other hand, there is an uncountable dense set of irrational numbers which are of Brjuno but not Herman type. 

Although Herman did not have the optimal characterisation for the linearisation of analytic circle diffeomorphisms, 
he used the linearisation property of circle maps to show that if $\ga$ satisfies that optimal condition, 
the critical point of $e^{2\pi i \ga} z+ z^2$ must lie on the boundary of the Siegel disk \cite{Her85}. 
His argument also applies to polynomials with a single critical point of higher orders. 
This result has been extended to cubic polynomials in \cite{ChRo16}.
On the other hand, Ghys and Herman built the first examples of polynomials with a Siegel disk whose boundary does 
not contain any critical points.  
Until this paper, and \cite{ShY18}, it was not known how the arithmetic condition of Herman is related to 
the presence of critical points on the boundary of Siegel disks.

\begin{rem}\label{R:high-type-conditions} 
The set of high type numbers $\irr$ in the modified continued fraction may be strictly larger than the set of 
high type numbers in the standard continued fraction. 
To be precise, let $\irr^s$ denote the set of irrational numbers $\ga$ whose entries 
$\tilde{a}_n$ in the standard continued fraction are at least $N$, for all $n\geq 0$. 
If $\tilde{a}_n\geq 2$ for all $n\geq 0$, then $a_n=\tilde{a}_n$ and $\gep_n=+1$, for all $n\geq 0$. 
This shows that for $N\geq 2$, $\irr^s \subseteq \irr$. 
But in general, $\irr$ is not contained in $\mathrm{HT}^s_{N-1}$, or in $\mathrm{HT}^s_{N-2}$, etc. 
Indeed, for any $N$, an element of $\irr$ may have infinitely many $+1$ entries in its standard continued fraction. 
In this sense, the theorems and corollaries stated in the introduction are slightly more general than the ones 
stated in \cite{ShY18}.
\end{rem} 

\begin{rem}
There are $\ga$ in $\irr$ which do not satisfy the Petersen-Zakeri condition in \cite{PZ04}. 
Thus, case (i) in \refT{T:trichotomy-main-thm} applies to some rotation numbers outside the Petersen-Zakeri class. 
\end{rem} 
\section{Topological model for the post-critical set}\label{S:arithmetic-model}
In this section we present the topological model for the post-critical set and the map on it.
This is a brief summary of the detailed construction in \cite{Che23}. 
We present the key features which will be used in this paper. 

A renormalisation scheme is often given as a class of maps, and a renormalisation operator which preserves that class of maps. 
The renormalisation operator involves a change of coordinates (rescaling). 
However, the approach taken to build the toy model for the renormalisation works in a different fashion. 
We start by defining a sequence of changes of coordinates first, and then build a map so that those changes of coordinates 
appear in the consecutive renormalisation of that map. 
We briefly present this below. 

\subsection{Model for the changes of coordinates}\label{SS:coordinate-model}
Consider the set  
\[\D{H}'= \{w\in \D{C} \mid \Im w > -1 \}.\]
For $r \in (0,1/2]$, define $Y_r : \ol{\D{H}'} \to \D{C}$ as 
\[Y_r(w)= r\Re w + 
\frac{ i }{2\pi} \log \Big |\frac{e^{-3\pi r}- e^{-\pi r  i }e^{-2\pi r i  w}}{e^{-3\pi r}- e^{\pi r  i }}\Big|.\]
We have $Y_r(0)=0$, $Y_r$ is continuous on $\D{H}'$, and real analytic in the variables $\Re w$ and $\Im w$. 
It maps vertical lines in $\D{H}'$ to vertical lines. 
But, it maps horizontal lines in $\D{H}'$ to non-straight curves which are $1/r$-periodic in $\Re w$. 
In particular, $Y_r$ is not conformal for any value of $r\in (0, 1/2]$ (The map degenerates the conformal structure as $r \to 0$). 
Despite that, it is proved fundamentally useful when compared to the conformal changes of coordinates in the 
near-parabolic renormalisation. 
\refF{F:model-Y_r} shows the behaviour of $Y_r$ on horizontal and vertical lines.

\begin{figure}[ht]
\begin{center}
\includegraphics[width=0.8\textwidth]{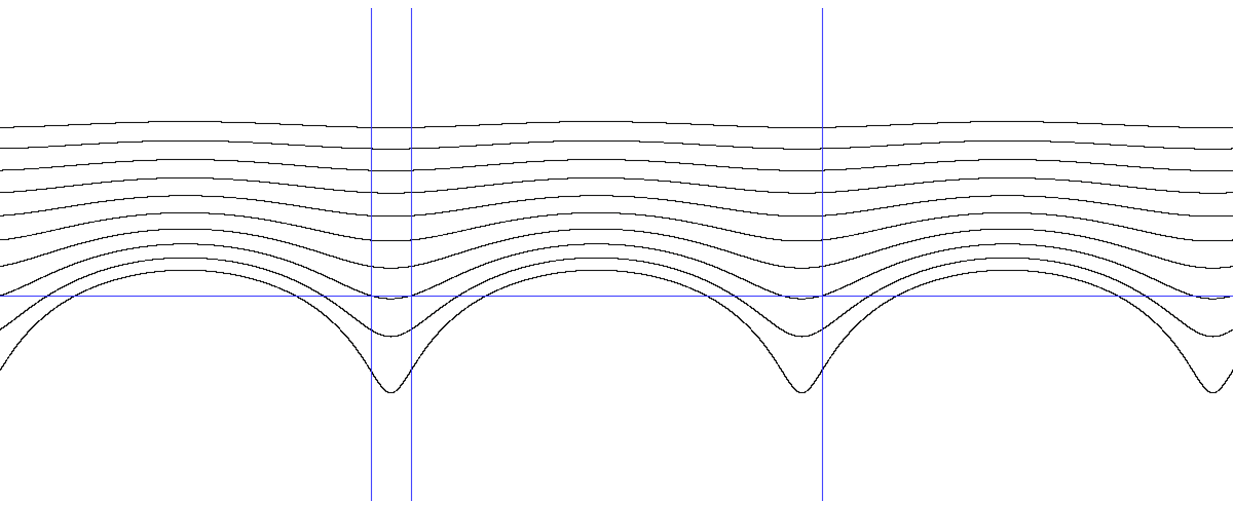}
\caption{The black curves are the images of some horizontal lines by $Y_r$. 
The vertical lines in blue, from left to right, are the images of the vertical lines $\Re w=-1$, $\Re w=0$, 
and $\Re w=1/\ga$, under $Y_r$. Here, $r=1/(10+1/(1+1/(1+\dots)))$.}
\label{F:model-Y_r}
\end{center}
\end{figure}

We summarise the key properties of $Y_r$ in the next proposition. 
Recall the map $h_r$ employed in the characterisation of the Herman numbers in \refS{S:arithmetic}.

\begin{propo}\label{P:coordinate-properties}
For every $r \in (0, 1/2]$ the following properties hold. 
\begin{itemize}
\item[(i)] The map $Y_r$ is injective on $\ol{\D{H}'}$, and $Y_r(\ol{\D{H}'}) \subset \D{H}'$.  
\item[(ii)] For every $w\in \ol{\D{H}'}$, 
\[Y_r(w+1/r) = Y_r(w)+1,\] 
\item[(iii)] For every $t \geq -1$, 
\[Y_r( i  t+ 1/r-1)= Y_r( i  t)+ 1- r,\]
\item[(iv)] For all $w_1, w_2$ in $\ol{\D{H}'}$,
\[|Y_r(w_1)- Y_r(w_2)| \leq 0.9 |w_1-w_2|,\]
\item[(v)] For all $y \geq 1$, 
\[|2 \pi \Im Y_r( i y/(2\pi)) - h_r^{-1}(y)| \leq \pi,\] 
\item[(vi)] For all $y \geq 0$, 
\[2\pi ry+ \log (1/r) - 4 \leq 2\pi \Im Y_r (1/(2r)+iy) \leq 2\pi ry + \log (1/r) + 2.\]
\end{itemize}
\end{propo}

By item (v) in the above proposition, $Y_r$ closely traces the behaviour of $h_r^{-1}$ on the imaginary axis.
By item (vi), $Y_r$ mimics the remarkable functional relation for the Brjuno function in 
\eqref{E:Brjuno-functional-equations}.

\subsection{Successive changes of coordinates}\label{SS:successive-coordinates}
Recall the sequence $\{\ga_n\}_{n=0}^\infty$ introduced in \refS{SS:modified-fractions}. 
Let $s(w)=\ol{w}$ denote the complex conjugation map. 
For $n\geq0$ we define $Y_n: \overline{\mathbb{H}'} \to \mathbb{H}'$ as 
\begin{equation}\label{E:Y_n}
Y_n(w) = 
\begin{cases}
Y_{\ga_n}(w) & \tif \gep_n=-1,\\
- s\circ Y_{\ga_n}(w)  & \tif \gep_n=+1. 
\end{cases}
\end{equation}
Each $Y_n$ is either orientation preserving or reversing, depending on $\gep_n$.

For a given set $X \subset \D{C}$, let us use the notation \[i X= \{ i  x \mid x\in X \}.\]
For $n\geq 0$, we have
\begin{equation} \label{E:invariant-imaginary-line}
Y_n( i  [-1, +\infty)) \subset   i  (-1, +\infty),\; Y_n(0)=0.
\end{equation}
By \refP{P:coordinate-properties}-(ii), for all $n\geq 0$ and all $w\in \ol{\D{H}'}$, 
\begin{equation} \label{E:Y_n-comm-1}
Y_n(w+1/\ga_n) = 
\begin{cases}
Y_n(w)+1 & \tif \gep_n=-1,\\
Y_n(w)-1 & \tif \gep_n=+1. 
\end{cases}
\end{equation}
Also, by the same proposition, for all $n\geq 0$ and all $t \geq -1$, 
\begin{equation}\label{E:Y_n-comm-2}
Y_n( i  t+ 1/\ga_n-1)= 
\begin{cases}
Y_n( i  t)+ (1-\ga_n) & \tif \gep_n=-1, \\
Y_n( i  t) +(\ga_n-1) & \tif \gep_n=+1,
\end{cases}
\end{equation}
and for all $n\geq 0$ and all $w_1, w_2$ in $\ol{\D{H}'}$, we have 
\begin{equation}\label{E:uniform-contraction-Y}
|Y_n(w_1)- Y_n(w_2)| \leq 0.9 |w_1-w_2|.
\end{equation}

\subsection{The straight topological model}\label{SS:straight-model}
For $n\geq 0$, let 
\begin{equation}\label{E:M_n-J_n-K_n}
\begin{gathered}
M_n^0 = \{w\in \ol{\D{H}'} \mid \Re w\in [0, 1/\ga_n]\}, \\
K_n^0 =  \{w\in M_n^0 \mid \Re w \in [0, 1/\ga_n-1] \}, \\
J_n^0 =  \{w\in M_n^0 \mid \Re w \in [1/\ga_n-1, 1/\ga_n]\}. 
\end{gathered}
\end{equation}
We inductively define the sets $M_n^j$, $J_n^j$, and $K_n^j$, for $j \geq 1$ and $n\geq 0$. 
Assume that $M_n^j$, $J_n^j$, and $K_n^j$ are defined for some $j \geq 0$ and all $n \geq 0$. 
We define these sets for $j+1$ and all $n\geq 0$ as follows. 
Fix an arbitrary $n \geq 0$. 
If $\gep_{n+1}=-1$, let 
\begin{equation}\label{E:M_n^j--1}
M_n^{j+1} 
= \textstyle{\bigcup_{l=0}^{a_n-2}} \big( Y_{n+1} ( M_{n+1}^j)+ l \big)  \bigcup \big( Y_{n+1}(K_{n+1}^j)+ a_n-1\big).
\end{equation}
If $\gep_{n+1}=+1$, let 
\begin{equation}\label{E:M_n^j-+1}
M_n^{j+1} 
= \textstyle{\bigcup_{l=1}^{a_n}} \big( Y_{n+1} ( M_{n+1}^j)+ l \big)  \bigcup \big(Y_{n+1}(J_{n+1}^j)+ a_n+1\big ).
\end{equation}
Regardless of the sign of $\gep_{n+1}$, define 
\[K_n^{j+1} =  \{w\in M_n^{j+1} \mid \Re w \in [0, 1/\ga_n-1] \},\]
\[J_n^{j+1} =  \{w\in M_n^{j+1} \mid \Re w \in [1/\ga_n-1, 1/\ga_n]\}.\]
\refF{F:topological-model} presents two generations of these domains.  

\begin{figure}[ht]
\begin{center}
\begin{pspicture}(14,11) 
\epsfxsize=6cm
\rput(3,8){\epsfbox{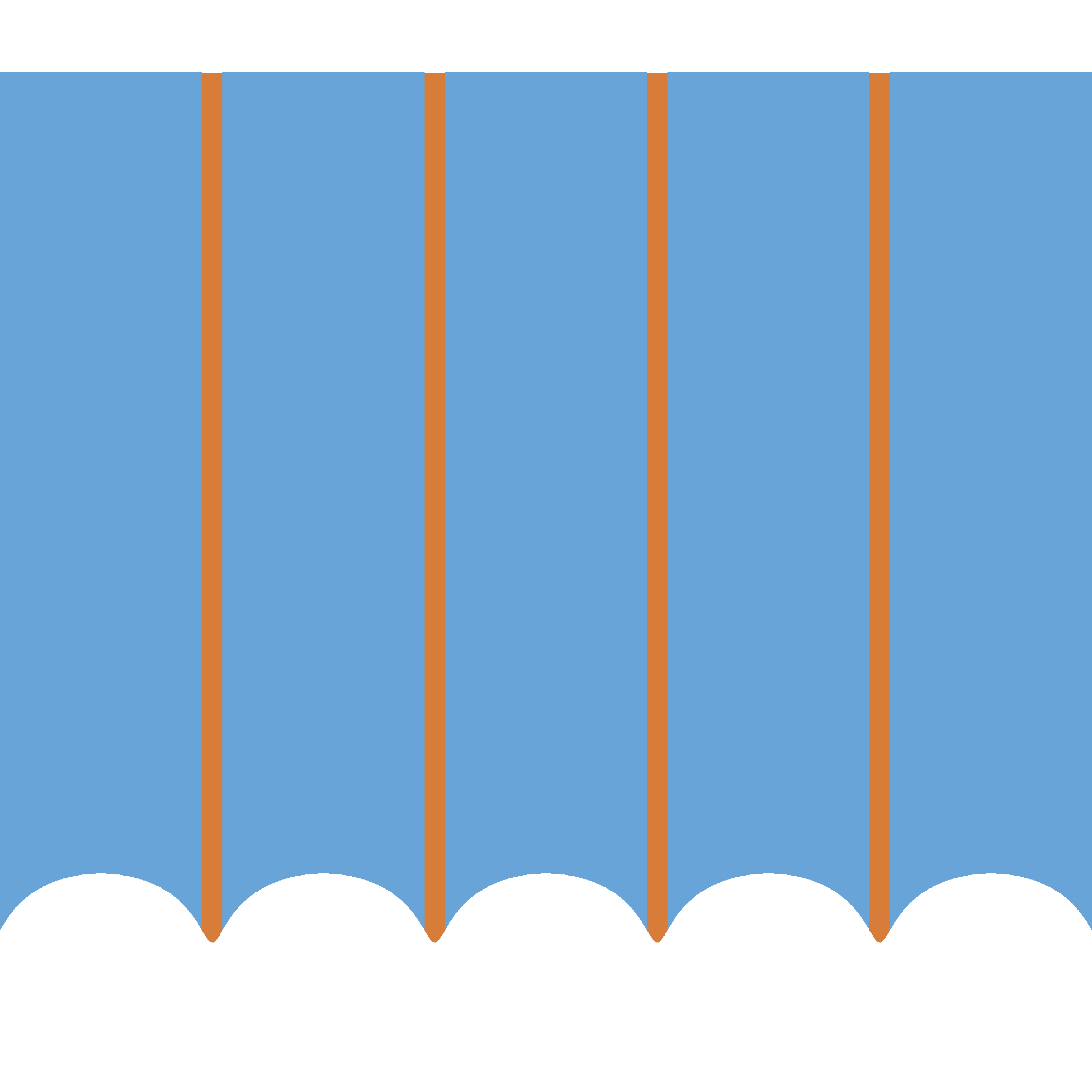}}
\rput(11,8){\epsfbox{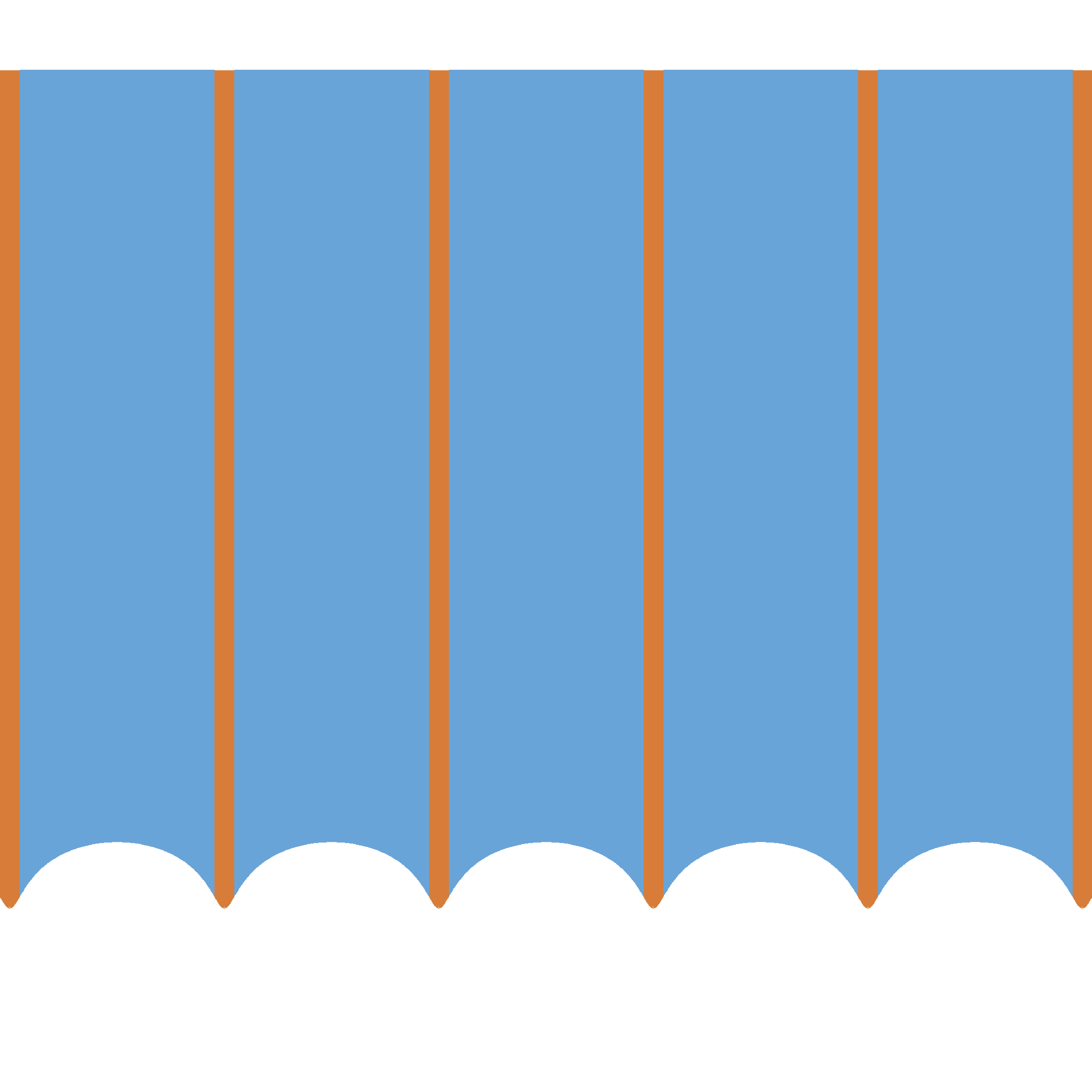}}

\pspolygon[linecolor=newcyan,fillstyle=solid,fillcolor=newcyan](0,0)(5.4545,0)(5.4545,3)(0,3)
\pspolygon[linecolor=neworange,fillstyle=solid,fillcolor=neworange](5.4545,0)(6,0)(6,3)(5.4545,3)
\rput(3,1.5){\small $K_n^0$}
\rput(5.7,1.5){\small $J_n^0$}

\pspolygon[origin={8,0},linecolor=newcyan,fillstyle=solid,fillcolor=newcyan](0,0)(5.4545,0)(5.4545,3)(0,3)
\pspolygon[origin={8,0},linecolor=neworange,fillstyle=solid,fillcolor=neworange](5.4545,0)(6,0)(6,3)(5.4545,3)
\rput(11,1.5){\small $K_n^0$}
\rput(13.7,1.5){\small $J_n^0$}

\psline[linewidth=.4pt]{->}(0,3.05)(0,5.8)
\psline[linewidth=.4pt]{->}(5.95,3.05)(1.25,5.85)

\psline[linewidth=.4pt]{->}(0.05,3.05)(4.9,5.8)
\psline[linewidth=.4pt]{->}(5.4545,3.05)(6,5.8)

\rput(1,5){\small $Y_n$}
\rput(4.8,5){\small $Y_n+a_{n-1}-1$}

\psline[origin={8,0},linewidth=.4pt]{->}(0,3.05)(1.15,5.95)
\psline[origin={8,0},linewidth=.4pt]{->}(5.95,3.05)(0,5.9)

\psline[origin={8,0},linewidth=.4pt]{->}(5.95,3.05)(5.9,5.9)
\psline[origin={8,0},linewidth=.4pt]{->}(5.4545,3.05)(6,5.9)

\rput(8.5,5){\small $Y_n$}
\rput(12.8,5.5){\small $Y_n+a_{n-1}+1$}
\end{pspicture}
\caption{In the left hand column $\gep_n=-1$ and in the right hand column $\gep_n=+1$. 
The sets $K_n^0$ and $J_n^0$ are on the lower row, and the set $M_{n-1}^1$ is on the upper row.}
\label{F:topological-model}
\end{center}
\end{figure}

For all $n \geq 0$ and $j\geq 0$, the sets $M_n^j$, $J_n^j$, and $K_n^j$ are closed and connected subsets of 
$\D{C}$, and are bounded by piece-wise analytic curves.  
Moreover, 
\[\{\Re w \mid w\in M_n^j\}= [0, 1/\ga_n].\]
The functional relations in \eqref{E:Y_n-comm-1} and \eqref{E:Y_n-comm-2} allow us to align the pieces together in the unions \eqref{E:M_n^j--1} 
and \eqref{E:M_n^j-+1}. 
More precisely, we have the following property of $M_n^j$. 

\begin{cor}\label{C:model-almost-periodic}
For every $n\geq 0$ and $j\geq 0$, the following hold:
\begin{itemize}
\item[(i)] for all $w\in \mathbb{C}$ satisfying $\Re w\in [0, 1/\ga_n-1]$, 
$w \in M_n^j$ if and only if $w+1 \in M_n^j$; 
\item[(ii)] for all $t \in \mathbb{R}$, $it \in M_n^j$ if and only if $it +1/\ga_n \in M_n^j$.    
\end{itemize}
\end{cor}

Recall that $\ga_{-1}=+1$. Let $M_{-1}^0=\{w\in \ol{\D{H}'} \mid \Re w \in [0, 1/\ga_{-1}]\}$, and for $j\geq 1$, 
consider
\[M_{-1}^j= Y_0(M_0^{j-1}) + (\gep_0+1)/2.\]
By \refP{P:coordinate-properties}-(i), $M_n^1 \subset M_n^0$, for $n\geq -1$. 
By an inductive argument, this implies that for all $n\geq -1$ and all $j\geq 0$, $M_n^{j+1} \subset M_n^j$.
For $n\geq -1$, we define 
\[M_{n}= \cap_{j\geq 0} M_{n}^j.\] 
Each $M_n$ consists of closed half-infinite vertical lines. The set $M_n$ may or may not 
be connected. 
By \refC{C:model-almost-periodic}, for real $t$, $it \in M_{-1}$ if and only if $(it +1) \in M_{-1}$. 
We may define
\begin{equation}\label{E:A_ga}
\hat{A}_\ga = \left\{s(e^{2\pi i w}) \mid w \in M_{-1} \right \} \cup \{0\}, \qquad \text{and} \qquad A_\ga = \partial  \hat{A}_\ga. 
\end{equation}
The set $A_\ga$ is the topological model for the post-critical set. 
It is defined for irrational values of $\ga$, and depends only on the arithmetic of $\ga$. 

\begin{rem}
An alternative approach for building a topological model for $\gL(f)$ was introduced by Buff and Ch\'eritat 
in 2009 \cite{BC09}.
In their model, they use rational approximation of $\ga$, and some conformal changes of coordinates, 
in order to build topological objects invariant for parabolic maps.
Then, the model for irrational values of $\ga$ is obtained from taking Hausdorff limits of those objects. 
The loss of control along the limit presents obstacles in further study of that model.
\end{rem}

\subsection{Hairy Cantor sets and Cantor bouquets}\label{SS:HCS-CB}
A \textbf{Cantor bouquet} is any subset of $\mathbb{C}$ which is ambiently homeomorphic to a set of the form 
\begin{equation*}\label{E:straight-cantor-bouquet}
\{re^{2\pi i \gj} \in \D{C} \mid 0 \leq \gj \leq 1, 0 \leq r \leq R(\gj) \},
\end{equation*}
where $R: \D{R}/\D{Z} \to [0, 1]$ satisfies the following:
\begin{itemize}
\item[(a)] $R=0$ on a dense subset of $\D{R}/\D{Z}$,  and $R > 0$ on a dense subset of $\D{R}/\D{Z}$, 
\item[(b)] for each $\gj_0\in \D{R}/\D{Z}$ we have 
\[\limsup_{\gj \to \gj_0^+ } R(\gj) = R(\gj_0) = \limsup_{\gj \to \gj_0^-} R(\gj).\]
\end{itemize}

A \textbf{one-sided hairy Jordan curve} is any subset of $\mathbb{C}$ which is ambiently homeomorphic to 
a set of the form 
\begin{equation*}\label{E:straight-hairy-circle}
\{re^{2\pi i \gj} \in \D{C} \mid 0 \leq \gj \leq 1, 1\leq r \leq 1+ R(\gj) \},
\end{equation*}
where $R: \D{R}/\D{Z} \to [0,1]$ satisfies properties (a) and (b) in the above definition. 

\begin{rem}
The Cantor bouquet and one-sided hairy Jordan curve enjoy similar topological features as the standard Cantor set. 
Under an additional mild condition (topological smoothness) they are uniquely characterised by some 
topological axioms, see \cite{AaOv93}.
\end{rem}

\subsection{Topology of the model}\label{SS:topology-model}

\begin{thm}[\cite{Che23}]\label{T:trichotomy-model}
For every irrational number $\ga$ the following hold.
\begin{itemize}
\addtolength{\itemsep}{.2em}
\item[(i)]If $\ga$ is a Herman number, $A_\ga$ is a closed Jordan curve around $0$. 
\item[(ii)]If $\ga$ is a Brjuno but not a Herman number, $A_\ga$ is a one-sided hairy Jordan curve around $0$. 
\item[(iii)]If $\ga$ is not a Brjuno number, $A_\ga$ is a Cantor Bouquet at $0$. 
\end{itemize}
\end{thm}

In the remaining of \refS{SS:topology-model}, we briefly sketch a proof of the above result, outlining the role of 
the properties of the changes of coordinates stated in \refP{P:coordinate-properties}.
One may skip the rest of this section, and move to \refS{SS:dynamics-model}, without detriment to the 
main purpose of this paper, which is to show that the post-critical set is homeomorphic to $A_\ga$. 

\begin{proof}[Brief sketch of the proof of \refT{T:trichotomy-model}] 
Recall that the sets $M_n^j$ and $M_n$ consist of closed half-infinite vertical lines. 
Each of these sets lies above the graph of a function, which may be conveniently used to study these sets. 
For $n\geq -1$, and $j \geq 0$, define $\mathfrak{b}_n^j:[0, 1/\ga_n] \to [-1, +\infty)$ according to 
\[M_n^j= \{w\in \D{C} \mid 0 \leq \Re w \leq 1/\ga_n, \Im w\geq \mathfrak{b}_n^j(\Re w)\}.\]
By Equations \eqref{E:Y_n-comm-1}--\eqref{E:Y_n-comm-2}, each $\mathfrak{b}_n^j$ is continuous, and since 
$M_n^{j+1} \subset M_n^j$, $\mathfrak{b}_n^{j+1} \geq \mathfrak{b}_n^j$ on $[0, 1/\ga_n]$. 
Thus, for $n\geq -1$, we may define $\mathfrak{b}_n:[0, 1/\ga_n] \to [-1,+\infty]$ as
\[\mathfrak{b}_n(x)= \lim_{j\to + \infty} \mathfrak{b}_n^j(x)= \sup_{j\geq 0} \mathfrak{b}_n^j(x).\]
Note that $\mathfrak{b}_n$ may attain $+\infty$. 
Evidently,
\begin{equation}\label{E:M_n-b_n}
M_n= \{w\in \D{C} \mid 0 \leq \Re w \leq 1/\ga_n, \Im w \geq \mathfrak{b}_n(\Re w)\}.
\end{equation}
By \refC{C:model-almost-periodic}, $\mathfrak{b}_n^j(0)=\mathfrak{b}_n^j(1/\ga_n)$, 
and $\mathfrak{b}_n^j(x+1)= \mathfrak{b}_n^j(x)$ for all $x\in [0, 1/\ga_n-1]$. 
Thus, for all $n\geq -1$, 
\begin{equation}\label{E:b_n^j-cont-periodic}
\mathfrak{b}_n(0)=\mathfrak{b}_n(1/\ga_n), \quad \text{and} \quad \mathfrak{b}_n(x+1)=\mathfrak{b}_n(x), \, \forall x \in [0, 1/\ga_n-1].   
\end{equation}
Using $Y_n(0)=0$ for all $n\geq 0$, and the uniform contraction of the $Y_n$, for all $n\geq -1$, 
\begin{equation}\label{E:b_n(0)=0}
\mathfrak{b}_n(0)=0.
\end{equation}

The collection of the functions $\mathfrak{b}_n^j$ and $\mathfrak{b}_n$ enjoy equivariant relations induced 
by the maps $Y_n$. 
That is, the graph of $\mathfrak{b}_n^j$ is obtained from the graph of $\mathfrak{b}_{n+1}^{j-1}$ by applying $Y_{n+1}$ 
and its translations. Similarly, $\mathfrak{b}_n$ is related to $\mathfrak{b}_{n+1}$ through $Y_{n+1}$. 
Each $Y_n$ exhibits two distinct behaviour. 
Above the line $\Im w=1/\ga_n$ it nearly acts as the linear map multiplication by $\ga_n$. 
Below that line, it has a logarithmic behaviour. 

Now assume that $\ga$ is a Brjuno number, and for $n\geq -1$ and $j\geq 0$ inductively define the functions 
\[\mathfrak{p}_n^j: [0,1/\ga_n] \to [-1, +\infty).\] 
For all $n\geq -1$, set $\mathfrak{p}_n^0\equiv (\C{B}(\ga_{n+1})+5\pi)/(2\pi)$. 
Assuming $\mathfrak{p}_n^j$ is defined for some $j\geq 0$ and all $n\geq -1$, define $\mathfrak{p}_n^{j+1}$ on $[0, 1/\ga_n]$ 
so that 
the graph of $\mathfrak{p}_n^{j+1}$ is obtained from applying $Y_{n+1}$ and its integer translations to 
the graph of $\mathfrak{p}_{n+1}^j: [0, 1/\ga_{n+1}] \to [-1, +\infty)$.
As for $\mathfrak{b}_n^j$, each $\mathfrak{p}_n^j$ is continuous, 1-periodic and $\mathfrak{p}_n^j(0)=\mathfrak{p}_n^j(1/\ga_n)$.
By explicit calculations, one may see that $\mathfrak{p}_n^1 \leq  \mathfrak{p}_n^0$ for all $n\geq -1$, and then by an inductive argument, 
one may show that for all $n\geq -1$ and $j \geq 0$, $\mathfrak{p}_n^{j+1} \leq \mathfrak{p}_n^j$. 
Therefore, we may define 
\[\mathfrak{p}_n(x)= \lim_{j \to +\infty}  \mathfrak{p}_n^j(x), \quad \forall x\in [0, 1/\ga_n].\]
A main difference with the functions $\mathfrak{b}_n^j$ here is that the convergence in the above equation is uniform on 
$[0, 1/\ga_n]$. This is because the maps $Y_n$ behave better near the top end of $M_n^0$, where they are close 
to multiplication by $\ga_n$.
Then, each $\mathfrak{p}_n$ is continuous. 
Moreover, 
\begin{equation}\label{E:p_n-peiodic}
\mathfrak{p}_n(0)=\mathfrak{p}_n(1/\ga_n), \quad \mathfrak{p}_n(x)=\mathfrak{p}_{n}(x+1), \forall x\in [0, 1/\ga_n-1].
\end{equation}
By definition, $\mathfrak{p}_n^0 \geq \mathfrak{b}_n^0$, for all $n\geq -1$. Then, by the equivariant properties of $\mathfrak{b}_n^j$ and $\mathfrak{p}_n^j$ 
one concludes that for all $n\geq -1$ and $j\geq 0$,  $\mathfrak{p}_n^j \geq \mathfrak{b}_n^j$. 
In particular, 
\begin{equation}\label{E:p_n>=b_n}
\mathfrak{p}_n (x) \geq \mathfrak{b}_n(x), \quad \forall x\in [0, 1/\ga_n].
\end{equation}
Using $\mathfrak{b}_n(0)=0$ for all $n$, the 1-periodicity of the functions $\mathfrak{b}_n$, and their equivariant property, 
one concludes that for every $n\geq -1$, $\mathfrak{b}_n(x) < +\infty$ holds for a dense set of $x$ in $[0,1/\ga_n]$.
In the same fashion, if some $\mathfrak{b}_m$ attains $+\infty$ at a single point, then every $\mathfrak{b}_n$ attains $+\infty$ on a dense 
subset of $[0, 1/\ga_n]$.
Whether those happen or not depend on the arithmetic of $\ga$. 
Indeed, by explicit calculations, $\max_{x\in [0, 1/\ga_n]} \mathfrak{b}_n^1(x)$ is uniformly close to $\log 1/\ga_n$. 
Then, because of the functional relations in Equation \eqref{E:Brjuno-functional-equations} and 
\refP{P:coordinate-properties}-(vi), it follows that for all $n\geq -1$, 
\begin{equation}\label{E:b_n-max-brjuno}
\Big|2 \pi  \sup_{x\in [0, 1/\ga_n]} \mathfrak{b}_n(x) - \C{B}(\ga_{n+1})\Big | \leq  5 \pi.
\end{equation}
In the above relation, $\infty - \infty$ is assumed to be $0$. 
It follows that if $\ga$ is a Brjuno number, every $\mathfrak{b}_n$ is bounded, and if $\ga$ is not a Brjuno number, 
every $\mathfrak{b}_n$ attains $+\infty$ at a single point, and hence on a dense set of points. 

Also, by Equation~\eqref{E:b_n-max-brjuno}, when $\ga$ is a Brjuno number, 
$\mathfrak{b}_n$ is uniformly close to $\mathfrak{p}_n^0$ at some points.
Then, by the uniform contraction of the maps $Y_n$, and the equivariant properties of the functions $\mathfrak{b}_n$ 
and $\mathfrak{p}_n$, we must have $\mathfrak{p}_n(x)=\mathfrak{b}_n(x)$ on a dense set of points $x$ in $[0, 1/\ga_n]$.
Indeed, one may see that these equalities occur near the vertical line $\Re w=1/(2\ga_n)$, due 
to the extreme contracting factor of each map $Y_n$ near that line.  
Among the vertical lines in the domain of $Y_n$, the least amount of contraction occurs near the 
vertical line $\Re w=0$. 
So, $0$ is the least likely place to have $\mathfrak{b}_n=\mathfrak{p}_n$. 
The answer to this question depends on the arithmetic of $\ga$ as we briefly explain below.

Because of the uniform contraction of the maps $h_{\ga_n}$, the criterion for the Herman numbers is stable under 
uniform changes to the maps $h_{\ga_n}$. 
More precisely, if one replaces $h_{\ga_n}$ by uniformly nearby maps, say $Y_n^{-1}$, the corresponding set 
of rotation numbers stays the same. 
Indeed, one may employ the estimate in \refP{P:coordinate-properties}-(v) to show that 
for integers $m > n \geq 0$ and $y \in (1, +\infty)$, 
\[\big| 2\pi \Im Y_{n} \circ \dots \circ Y_m( i  y/(2\pi)) - h_{\ga_n}^{-1} \circ  \dots  \circ  h_{\ga_m}^{-1}(y)\big| 
\leq 10 \pi,\]
provided $h_{\ga_n}^{-1} \circ  \dots  \circ  h_{\ga_m}^{-1}(y)$ is defined. 
This estimate, and the uniform contraction of the maps $Y_n$, is used to show that 
$\ga$ belongs to $\E{H}$, if and only if, for all $x>0$ there is $m \geq 1$ such that 
\[\Im Y_0 \circ  \dots  \circ Y_{m-1}(i \C{B}(\ga_{m})/(2\pi)) \leq x.\]
Thus, $\ga$ is a Herman number if and only if $\mathfrak{p}_n(0)=0$ for all $n\geq -1$. 
Combining with earlier arguments, one concludes that when $\ga$ is a Herman number, $\mathfrak{b}_n\equiv \mathfrak{p}_n$, 
and when $\ga$ is a Brjuno but not a Herman number, $\mathfrak{b}_n< \mathfrak{p}_n$ holds on a dense subset of $[0, 1/\ga_n]$. 

As each $M_n$ is closed, for every $x \in [0, 1/\ga_n)$, $\liminf_{s\to x^+} \mathfrak{b}_n(s) \geq \mathfrak{b}_n(x)$, 
and for all $x \in (0, 1/\ga_n]$, $\liminf_{s\to x^-} \mathfrak{b}_n(s) \geq \mathfrak{b}_n(x)$. 
In fact, both of ``$\geq$'' are ``$=$''. 
That is because, for large values of $m$, $\mathfrak{b}_{n+m}$ is $1$-periodic and by the equivariant property of the maps 
$\mathfrak{b}_j$, and the uniform contraction of the maps $Y_j$, one may obtain a sequence of points on the graph of $\mathfrak{b}_n$ 
which converges to $(x, \mathfrak{b}_n(x))$. 
That requires a detailed combinatorial analysis of the trajectories of points for consecutive iterates of the maps $Y_j$. 
These relations imply the property (b) in the definition of the hairy Jordan curve and the Cantor bouquet. 
\end{proof}

\subsection{Dynamics on the topological model}\label{SS:dynamics-model}
In this section we present the (toy) map 
\begin{equation}
T_\ga: A_\ga \to A_\ga, 
\end{equation}
which serves as the model for $f: \Lambda(f) \to \Lambda(f)$. 
This is also a brief description of the detailed arguments presented in \cite{Che23}. 

Let us fix $\ga \in \D{R}\setminus \D{Q}$, and let $M_{-1}$ be the corresponding topological model defined 
in \refS{SS:straight-model}.
Given $w_{-1} \in M_{-1}$, we inductively identify $l_i \in \mathbb{Z}$, and then $w_{i+1} \in M_{i+1}$,  
such that 
\begin{equation}
- \gep_{i+1} \Re (w_i -l_i) \in [0,1) \quad \text{ and } \quad Y_{i+1}(w_{i+1})+l_{i}= w_i.
\end{equation}
Then, for all $n\geq 0$, 
\begin{equation}\label{E:trajectory-condition-1}
w_{-1}=(Y_0+l_{-1}) \circ (Y_1 + l_0)  \circ \dots \circ (Y_n+1_{n-1})(w_n),
\end{equation}
and by \eqref{E:M_n^j--1} and \eqref{E:M_n^j-+1}, for all $i \geq 0$, 
\begin{equation}\label{E:trajectory-real-parts}
0 \leq l_i \leq a_i + \gep_{i+1}, \q \tand \q 0 \leq \Re w_i < 1/\ga_i.
\end{equation}
We refer to the sequence $(w_i ; l_i)_{i \geq -1}$ as the \textbf{trajectory} of $w_{-1}$.

Consider the map 
\begin{equation}
\tilde{T}_\ga:M_{-1} \to M_{-1}, 
\end{equation}
defined as follows. 
For $w_{-1}$ in $M_{-1}$ with trajectory $(w_i; l_i)_{i\geq -1}$,  
\begin{itemize}
\item[(i)] if there is $n \geq 0$ such that $w_n \in K_n$, and for all $0 \leq i \leq n-1$, $w_i \in M_i \setminus K_i$, then 
\[\tilde{T}_\ga(w_{-1})= \left(Y_0+\frac{\gep_0+1}{2}\right ) \circ \left(Y_1+\frac{\gep_1+1}{2}\right ) 
\circ \cdots \circ \left (Y_n +\frac{\gep_{n}+1}{2} \right)(w_n+1);\] 
\item[(ii)] if for all $n \geq 0$, $w_n \in M_n \setminus K_n$, then 
\[\tilde{T}_\ga(w_{-1})= \lim_{n \to +\infty} \left(Y_0+\frac{\gep_0+1}{2}\right ) \circ 
\left(Y_1+\frac{\gep_1+1}{2}\right ) 
\circ \cdots \circ \left (Y_n +\frac{\gep_{n}+1}{2} \right) (w_n+1-1/\ga_n).\]
\end{itemize}
Evidently, item (i) leads to continuous maps on pieces of $M_{-1}$. 
There might be a vertical half-infinite line where item (ii) applies. 
On that set, the uniform contractions of the maps $Y_j$ imply that the sequence of maps in item (ii) 
converges to a well-defined map. 
Moreover, it follows from \eqref{E:Y_n-comm-1} and \eqref{E:Y_n-comm-2} that these piece-wise defined maps 
match together and produce a well-define homeomorphism  
\[\tilde{T}_{\ga}: M_{-1}/\D{Z} \to M_{-1}/ \D{Z}.\]
One may compare the above definition of $\tilde{T}_\ga$ to the action of $f$ on its 
renormalisation tower in \refS{SS:iterates-shifts-lifts}. 
By the definition of $A_\ga$ in \refE{E:A_ga}, $\tilde{T}_\ga$ induces a homeomorphism 
\[T_\ga:\hat{A}_\ga \to \hat{A}_\ga,\]
which may be restricted to the homeomorphism 
\[T_\ga:A_\ga \to A_\ga.\]

Recall that $T_\ga: A_\ga \to A_\ga$ is called \textbf{topologically recurrent}, 
if for every $x \in A_\ga$ there is a strictly increasing sequence of positive integers $(m_i)_{i\geq 0}$ 
such that $T_\ga \co{m_i}(x) \to x$ as $i \to +\infty$. 
A set $K \subset A_\ga$ is called \textbf{invariant} under $T_\ga$, if $T_\ga (K) =K= T_\ga^{-1}(K)$. 
We use the notation $\omega(z)$ to denote the set of accumulation points of the orbit of a given point $z \in A_\ga$. 
That is, the set of limit points of all convergent subsequences of the orbit of $z$. 

Define $r_\ga \in [0,1]$ according to 
\[[ r_\ga, 1] = \{z\in A_\ga \mid \Im z =0, \Re z \geq 0\}.\]
By \refT{T:trichotomy-model}, when $\ga$ is a Herman number we have $r_\ga=1$, when $\ga$ is a Brjuno but not a 
Herman number we have $r_\ga \in (0, 1)$, and when $\ga$ is not a Brjuno number we have $r_\ga=0$. 

Below, we summarise the dynamical behaviour of $T_\ga$ on $A_\ga$ which are established in \cite{Che23}. 

\begin{thm}[\cite{Che23}]\label{T:model-dynamics}
For every $\ga \in \mathbb{R} \setminus \mathbb{Q}$, $T_\alpha: A_\alpha \to A_\alpha$ satisfies 
the following properties.
\begin{itemize}
\item[(i)] $T_\alpha: A_\alpha \to A_\alpha$ is a topologically recurrent homeomorphism.
\item[(ii)] The map 
\[\omega: [r_\ga, 1] \to \{X \subseteq A_\ga \mid X \text{ is non-empty, closed and invariant}\}\]
is a homeomorphism with respect to the Hausdorff metric on the range. 
In particular, every non-empty closed invariant subset of $A_\alpha$ is equal to $\omega(z)$, for some $z \in A_\ga$. 
\item[(iii)] The map $\omega$ on $[r_\ga, 1]$ is strictly increasing with respect to the linear order on $[r_\ga, 1]$ 
and the inclusion on the range. 
\item[(iv)] If $\ga$ is not a Brjuno number, $\omega(t)$ is a Cantor bouquet for every $t \in (r_\ga, 1]$, 
and $\omega(1)= A_\ga$. 
\item[(v)] If $\ga$ is a Brjuno but not a Herman number, $\omega(t)$ is a hairy Jordan curve for every $t \in (r_\ga, 1]$, 
and $\omega(r_\ga)$ is a Jordan curve. 
\end{itemize}
\end{thm} 
\section{Near-parabolic renormalisation scheme}\label{S:NP-renormalisation-scheme} 
In this section we present the near-parabolic renormalisation scheme introduced by Inou and Shishikura \cite{IS06}.
This consists of a class of maps discussed in \refS{SS:IS-class}, and a renormalisation operator acting on that class 
discussed in \refS{SS:renormalisation-def}.  
Our presentation of the renormalisation operator is slightly different from the one by Inou and Shishikura, 
but produces the same map. 

\subsection{Inou-Shishikura class of maps}\label{SS:IS-class}
Let $\rs$ denote the Riemann sphere. Consider the filled-in ellipse
\[E= \Big\{x+ i y\in \D{C} \; \Big | \; \Big(\frac{x+0.18}{1.24}\Big)^2+\Big(\frac{y}{1.04}\Big)^2\leq 1 \Big\},\]
and the domain
\begin{equation}\label{E:U}
U= g(\rs \setminus E), \text{ where } g(z)=-4z/(1+z)^2.
\end{equation}
The domain $U$ is simply connected and contains $0$. 

The restriction of the polynomial $P(z)=z(1+z)^2$ on $U$ has a specific covering structure which plays a 
central role in the near-parabolic renormalisation scheme. 
The polynomial $P$ has a parabolic fixed point at $0$ with multiplier $P'(0)=1$. 
It has a simple critical point at $\cp_P=-1/3 \in U$ and a critical point of order two at $-1 \in \D{C}\setminus \ol{U}$. 
The critical point $-1/3$ is mapped to $\cv_P=-4/27 \in U$, and $-1$ is mapped to $0$. See \refF{F:poly}. 

Let $\IS_0$ denote the class of all maps of the form 
\[h=P\circ \vfi^{-1}\!\!:U_h \rightarrow \D{C}\]
where 
\begin{itemize}
\item[(i)]$\vfi\colon U \ra U_h$ is holomorphic, one-to-one, onto; and 
\item[(ii)]$\vfi(0)=0$ and $\vfi'(0)=1$.
\end{itemize}
By (i), every map in $\IS_0$ has the same covering structure on its domain as the one of $P$ on $U$. 
By (ii), every map in $\IS_0$ has a fixed point at $0$ with multiplier $+1$, and a unique critical point 
at $\cp_h=\vfi(-1/3)\in U_h$ which is mapped to $\cv_h=-4/27$.

For $\ga\in \D{R}$, let $R_\ga(z)= \ea z$, and define 
\[\IS_\ga=\{ h \circ R_\ga \mid h\in \IS_0\}.\]
We continue to use the notation $U_f$ to denote the domain of definition of $f \in \IS_\ga$. 
That is, if $f=h \circ R_\ga$ with $h \in \IS_0$, then $U_f= R_{\ga}^{-1}(U_h)$. 

We equip $\bigcup_{\ga \in \D{R}} \IS_\ga$ with the topology of uniform convergence on compact sets. 
That is, given $h : U_h \to \D{C}$, a compact set $K\subset U_h$ and an $\eps>0$, 
a neighbourhood of $h$ (in the compact-open topology) is defined as the set of maps 
$g\in \cup_{\ga\in \D{R}} \IS_\ga$ such that $K \subset U_g$ and for all $z\in K$ we have $|g(z)-h(z)|<\eps$.
There is a one-to-one correspondence between $\IS_0$ and the space of normalised univalent maps on the unit disk. 
By the Koebe distortion theorem \cite[Thm 2.5]{Dur83}, for any closed set $A \subset \D{R}$,
$\bigcup_{\ga \in A} \IS_\ga$ is compact in this topology.

We normalise the family of quadratic polynomials by placing a fixed point at $0$ and the finite critical value at $-4/27$;  
\[ Q_\ga(z)=\ea z+\frac{27}{16} e^{4\pi \ga i}z^2.\]
Then, $Q_\ga'(0)= \ea$, $Q_\ga '(-8 e^{-2\pi \ga i} /27)=0$, and $Q_\ga(-8 e^{-2\pi \ga i} /27)= -4/27$. 
We set the notation 
\[\QIS_\ga= \IS_\ga \cup \{Q_\ga\}.\]
When $h=Q_\ga$, we set $U_h= \D{C}$.
We referred to the class of maps $\cup_{\ga \in \irr} \QIS_\ga$ as the class $\C{F}$ in the introduction. 
In \refP{P:sequence-renormalisations} we determine the value of $N$ for \refT{T:trichotomy-main-thm}. 

Let $h= h_0 \circ R_\ga \in \IS_{\ga}$ with $h_0 \in \IS_0$ and $\ga\in \mathbb{R}$. 
The map $h_0$ has a double fixed point at $0$. 
For $\ga$ small enough and non-zero, $h$ is a small perturbation of $h_0$, and hence, it has a non-zero 
fixed point near $0$ which has split from $0$ at $\ga=0$. 
We denote this fixed point by $\gs_h$. 
It follows that $\sigma_h$ depends continuously on $h_0$ and $\ga$, with asymptotic expansion 
$\sigma_h=-4\pi \ga i/h_0''(0)+o(\ga)$, as $\ga$ tends to $0$. 
Evidently, $\sigma_h \ra 0$ as $\ga \ra 0$.  

Given a set $X$ in a topological space, $\ol{X}$ denotes the closure of $X$, $\interior(X)$ its interior, and 
$\partial X$ its boundary. 

\begin{propo}[\cite{IS06}]\label{P:Ino-Shi1} 
There is $r_1 >0$ such that for every $h\colon U_h \ra \D{C}$ in $\bigcup_{\ga \in (0,r_1]} \QIS_\ga$,
there exist a simply connected domain $\C{P}_h \subset U_h$ and a univalent map 
\[\Phi_h\colon \C{P}_h \ra \D{C}\] 
satisfying the following properties:
\begin{itemize}  
\item[(a)] $\C{P}_h$ is bounded by piecewise smooth curves and $\ol{\C{P}_h} \subset U_h$;
\item[(b)] $\cp_h$, $0$, and $\sigma_h$ belong to $\partial \C{P}_h$, while $\cv_h$ belongs to $\interior(\C{P}_h)$;
\item[(c)] $\gF_h(\C{P}_h)$ contains the set $\{w\in \mathbb{C}\mid \Re w \in (0,2]\}$; 
\item[(d)] $\Phi_h(\cv_h)=1$, $\Im \Phi_h(z) \ra +\infty$ as $z \to 0$ in $\C{P}_h$, and $\Im \Phi_h(z)\ra-\infty$ 
as $z \to \sigma_h$ in $\C{P}_h$; 
\item[(e)] If $z$ and $h(z)$ belong to $\C{P}_h$, then 
\[\Phi_h(h(z))=\Phi_h(z)+1;\] 
\item[(f)] the induced map $\gF_h: \C{P}_h/{\sim} \to \D{C}/\D{Z}$, where $z \sim h(z)$, is a biholomorphism;  
\item[(g)] $\Phi_h$ is unique, and depends continuously on $h$.
\end{itemize}
\end{propo}

The class $\IS_0$ is denoted by $\mathcal{F}_1$ in \cite{IS06}. 
One may refer to Theorem 2.1 as well as Main Theorems 1 and 3 in \cite{IS06}, for further details on the 
above proposition. 

In the next proposition we state some crucial properties of $\C{P}_h$ and $\gF_h$.

\begin{figure}[ht]
\begin{center}
\begin{pspicture}(.5,.3)(8.5,5.2)
\rput(4.5,2.75){\includegraphics[width=8cm]{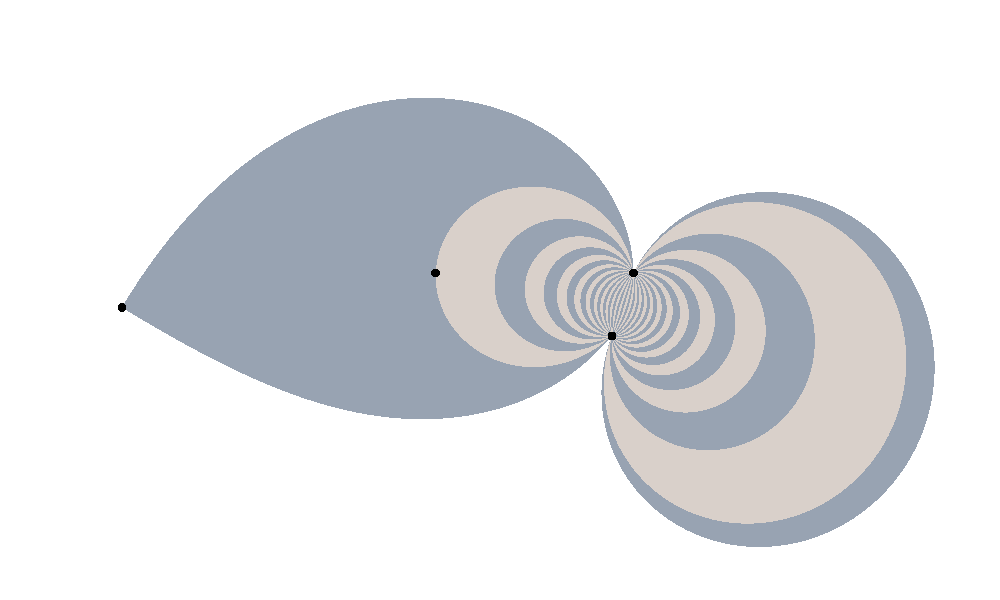}}
\rput(2.0,2.65){$\cp_h$}
\rput(3.65,2.95){$\cv_h$}
\rput(5.65,3.4){$0$}
\psline(5.2,2.05)(5.36,2.4)
\rput(5.1,1.9){$\sigma_h$}
\end{pspicture}
\end{center}
\caption{The domain $\C{P}_h$ and the special points associated to some $h\in\IS_\ga$. 
The alternating coloured croissants are the pre-images of vertical strips of width one by $\Phi_h$.}
\label{F:petal}
\end{figure}

\begin{propo}\label{P:petal-geometry}
There exist $r_2 \in (0, r_1]$, as well as integers $c_1\leq 1/r_2-3/2$ and $c_2$ such that for every 
$h\colon U_h \ra \D{C}$ in $\bigcup_{\ga \in (0,r_2]} \QIS_\ga$, the domain $\C{P}_h \subset U_h$ in 
\refP{P:Ino-Shi1} may be chosen to satisfy the additional properties:
\begin{itemize}  
\item[(a)] there exists a continuous branch of argument defined on $\C{P}_h$ such that 
\[\max_{w,w'\in \C{P}_h} |\arg(w)-\arg(w')|\leq 2 \pi c_2,\]
\item[(b)] $\Phi_h(\C{P}_h)=\{w \in \D{C} \mid 0 < \Re(w) < \ga^{-1} -c_1\}$. 
\end{itemize}
\end{propo}

See \cite[Prop.~2.4]{Che13}  or \cite[Prop.~12]{BC12} for proofs. 
The map $\Phi_h: \mathcal{P}_h\to \D{C}$ is called the perturbed Fatou coordinate or simply 
the \textbf{Fatou coordinate} of $h$. See \refF{F:petal}. 

\subsection{Near-parabolic renormalisation operator}\label{SS:renormalisation-def}
For $h\colon U_h \ra \D{C}$ in $\bigcup_{\ga \in (0,r_2]} \QIS_\ga$, with Fatou coordinate
$\Phi_h\colon \C{P}_h \ra \D{C}$, let 
\begin{equation}\label{E:sector-def}
\begin{gathered}
A_h=\{z\in \C{P}_h : 1/2 \leq \Re(\Phi_h(z)) \leq 3/2 \: ,\: -2 \leq \Im \Phi_h(z) \leq 2 \}, \\
B_h=\{z\in \C{P}_h : 1/2 \leq \Re(\Phi_h(z)) \leq 3/2 \: , \: 2\leq \Im \Phi_h(z) \}.
\end{gathered}
\end{equation}
By \refP{P:Ino-Shi1}, $\cv_h \in \interior(A_h)$ and $0 \in \partial B_h$. 
See \refF{F:sectorpix}.

It follows from \cite{IS06} (see \refR{R:alternative-sets} below) that there is a positive integer $k_h$, 
depending on $h$, such that the following four properties hold:
\begin{itemize}
\item[(i)] For every integer $k$, with $1 \leq k \leq k_h$, there exists a unique connected component of 
$h^{-k}(B_h)$ which is compactly contained in $U_h$ and contains $0$ on its boundary. We denote 
this component by $B_h^{-k}$. 
\item[(ii)] For every integer $k$, with $1\leq  k \leq k_h$, there exists a unique connected component of 
$h^{-k}(A_h)$ which has non-empty intersection with $B_h^{-k}$, and is compactly contained in $U_h$. 
This component is denoted by $A_h^{-k}$. 
\item[(iii)] The sets $A_h^{-k_h}$ and $B_h^{-k_h}$ are contained in 
\[\big \{z\in\C{P}_h \mid 1/2< \Re \Phi_h(z) < 1/\ga - c_1\big\}.\]
\item[(iv)] The maps $h:A_h^{-k}\to A_h^{-k+1}$, for $2\leq k \leq k_h$, and $h:B_h^{-k}\to B_h^{-k+1}$, 
for $1\leq k \leq k_h$, are one-to-one. 
The map $h: A_h^{-1} \to A_h$ is a degree two proper branched covering.
\end{itemize}
Assume that $k_h$ is the smallest positive integer for which the above properties hold. Then define
\[S_h=A_h^{-k_h}\cup B_h^{-k_h}.\]

\begin{propo}\label{P:k_n-bounded}
There is a constant $k \in \D{Z}$ such that for all $h \in \bigcup_{\ga \in (0,r_2]} \QIS_\ga$, $k_h \leq k$. 
\end{propo}

See \cite{Che13} or \cite{Che19} for the proof of the above proposition.

\begin{figure}[ht]
\begin{center}
 \begin{pspicture}(-.5,1.2)(11.4,9)
\epsfxsize=6.3cm
\rput(3.5,5.9){\epsfbox{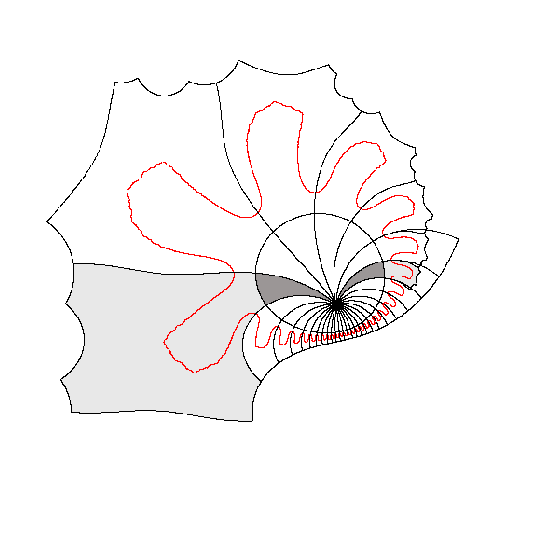}}  
  \psset{xunit=1cm}
  \psset{yunit=1cm}
    \pscurve[linewidth=.6pt,linestyle=dashed,linecolor=black]{->}(5.3,5.1)(5.3,5.5)(5.2,5.8)
    \pscurve[linewidth=.6pt,linestyle=dashed,linecolor=black]{->}(2.2,3.7)(2.3,4.1)(2.2,4.5)
    \pscurve[linewidth=.6pt,linestyle=dashed,linecolor=black]{->}(.7,6.4)(2,6.4)(3.2,6.2)(3.6,5.6)
    \rput(5.3,4.7){$S_h$}
    \rput(2.2,3.4){$A_h^{-1}$}
    \rput(.1,6.4){$B_h^{-1}$}
    \psdots[dotsize=2pt](2.3,5.04)(3.4,4.97)
    \rput(2.,5){\small{$\cp_h$}}
    \rput(3.33,4.77){\small{$\cv_h$}}
\newgray{Lgray}{.99}
\newgray{LLgray}{.88}
\newgray{LLLgray}{.70}
\psdots(7.6,5.5)(8,5.5)(11,5.5)
\pspolygon[fillstyle=solid,fillcolor=LLLgray](7.8,7.3)(7.8,6.1)(8.2,6.1)(8.2,7.3)
\pspolygon[fillstyle=solid,fillcolor=LLgray](7.8,6.1)(7.8,4.9)(8.2,4.9)(8.2,6.1)

\pspolygon[fillstyle=solid,fillcolor=LLLgray](10.2,7.3)(10.05,6.2)(10.45,6.2)(10.6,7.3)
\pspolygon[fillstyle=solid,fillcolor=LLgray](10.05,6.2)(9.8,5.1)(9.9,5)(10,4.97)(10.1,5)(10.2,5.1)(10.45,6.2)
\psdots(7.6,5.5)(8,5.5)
\psline{->}(7.6,5.5)(11.3,5.5)
\psline{->}(7.6,4.7)(7.6,7.3)
\rput(8,5.3){\tiny{$1$}}
\rput(10.9,5.3){\tiny{$\frac{1}{\ga}-c_1$}}
\rput(7.3,4.9){\tiny{$-2$}}

\psline{->}(6,5.9)(7.5,5.9)
\rput(6.8,6.1){$\Phi$}

\pscurve[linestyle=dashed]{<-}(7.9,6.7)(8.9,6.9)(10.3,6.7)
\rput(9.1,7.5){\tiny{$\gF_h \circ h\co{k_h}\circ \gF_h^{-1}$}}

\psline{->}(7.4,4.7)(6.6,3.2)
\rput(7.7,4.){$e^{2\pi i w}$}
\psellipse(6,2.1)(1.2,.9)
\psdots[dotsize=2pt](6,2.1)
\rput(4,2.7){$\rr (h)$}
\pscurve[linewidth=.5pt]{->}(5.2,3)(4.4,3.2)(4.8,2.6)
\rput(6.2,2.1){\small{$0$}}
\rput(3.4,9){$h$}
\pscurve[linewidth=.5pt]{->}(3.5,8.5)(3,9)(2.5,8.4)
 \end{pspicture}
\caption{Illustration of the sets $A_h$, $B_h$,..., $A_h^{-k_h}$, $B_h^{-k_h}$, and the sector $S_h$. 
The amoeba shaped red curve denotes a large number of iterates of $\cp_h$ under $h$.}
\label{F:sectorpix}
\end{center}
\end{figure}

Since $h\co{k_h}: S_h \to A_h \cup B_h$, the composition 
\begin{equation}\label{E:renorm-def}
E_h=\Phi_h \circ h\co{k_h} \circ \Phi_h^{-1}:\Phi_h(S_h) \ra \D{C}
\end{equation}
is a well-defined map. 
Also, consider the covering map 
\begin{equation}\label{E:exp}
\ex(w)=(-4/27) e^{2\pi i w}. 
\end{equation} 
By \refP{P:Ino-Shi1}-(e), $E_h(w+1)=E_h(w)+1$, when both $w$ and $w+1$ belong to the closure of 
$\gF_h(S_h)$. Thus, $E_h$ induces, via $\ex$, a unique map $\rr (h)$ defined on 
a set containing a punctured neighbourhood of $0$. 
It follows from \refP{P:Ino-Shi1}-(d) that $\rr(h)(z) \to 0$ as $z\to 0$. Therefore, $0$ is a removable singularity of 
$\rr(h)$. 
Basic calculations show that near $0$, 
\[\rr(h)(z)= e^{-2 \pi i/\ga} z+ O(z^2).\] 
The map $\rr (h)$, restricted to the interior of $\ex(\Phi_h(S_h))$, is called the 
\textbf{near-parabolic renormalisation} of $h$. 
We may simply refer to this operator as \textbf{renormalisation}. 

Because $\gF_h(\cv_h)=+1$ and $\ex(+1)= -4/27$, the critical value of $\rr(h)$ is placed at $-4/27$. 
See \refF{F:sectorpix}. 
It is also worth noting that the action of the renormalisation on the asymptotic rotation number at $0$ is 
\[\ga \mapsto -1/\ga \mod \D{Z}.\]

\begin{rem}\label{R:alternative-sets}
Inou and Shishikura give a somewhat different definition of this renormalisation operator using slightly 
different regions $A_h$ and $B_h$ compared to the ones here. 
However, the two processes produce the same map $\rr(h)$ modulo their domains of definition.
More precisely, there is a natural extension of $\gF_h$ onto the sets $A_h^{-k} \cup B_h^{-k}$, for 
$0\leq k\leq k_h$, such that each set $\Phi_h(A_h^{-k}\cup B_h^{-k})$ is contained in the union 
\[D^\sharp_{-k} \cup D_{-k} \cup D''_{-k} \cup D'_{-k+1} \cup D_{-k+1} \cup D^\sharp_{-k+1}\] 
in the notations used in \cite[Section 5.A]{IS06}.
\end{rem}

Consider the domain  
\begin{equation}\label{E:V}
V=P^{-1}\big (B(0,4e^{4\pi}/27) \big ) \setminus  \big ( (-\infty,-1]\cup B \big)
\end{equation}
where $B$ is the component of $P^{-1}(B(0, 4e^{-4\pi}/27))$ containing $-1$. 
By an explicit calculation (see \cite[Prop.~5.2]{IS06}) one can see that $\ol{U} \subset \interior(V)$.   
See \refF{F:poly}. 

\begin{figure}[ht]
\begin{center}
  \begin{pspicture}(8,3.2)
  \rput(4.5,1.6){\epsfbox{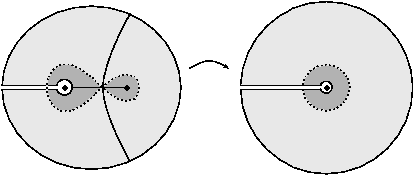}}
      \rput(6.12,1.6){{\small $\times$}}
      \rput(5.9,1.4){{\small $cv_P$}}
      \rput(6.7,1.8){{\small $0$}} 
      \rput(2.7,1.6){{\small $\times$}}
      \rput(2.7,1.25){{\small $cp_P$}}
      \rput(3.2,1.8){{\small $0$}}
      \rput(1.85,1.3){{\small $-1$}}
      \rput(4.5,2.2){{\small $P$}}
      \rput(1.1,2.7){{\small $V$}} 
\end{pspicture}
\caption{Covering structure of the polynomial $P$; similar colors and line styles are mapped on one another.}
\label{F:poly}
\end{center}
\end{figure}

\begin{thm}[\cite{IS06}]\label{T:Ino-Shi2} 
There exist $r_3 \in (0, r_2]$ such that for every $h \in \bigcup_{\ga \in (0, r_3]} \QIS_\ga$, $\rr(h)$ is 
defined and belongs to the class $\IS_{-1/\ga}$. 
That is, there exists a one-to-one holomorphic map $\psi:U\to \D{C}$ with $\psi(0)=0$ and $\psi'(0)=1$ so that 
\[\rr(h)(z)=P \circ \psi^{-1}(e^{-2\pi i/\ga} z), \; 
\forall z\in \psi(U)\cdot e^{2\pi i/ \ga}.\] 
Furthermore,  $\psi:U\to \D{C}$ extends to a univalent map on $V$.    
\end{thm}

\refT{T:Ino-Shi2} is a refinement of earlier constructions by Lavaurs~\cite{La89} and Shishikura~\cite{Sh98}. 
See also \cite{Ya08} for an alternative point of view. 
Two applications of a similar renormalisation led to the remarkable result of Shishikura \cite{Sh98} 
on the boundary of the Mandelbrot set. Even one application of the operator in a specific setting may be 
fruitful, as in \cite{ACE22}. 
The results stated in the Introduction, and the technical statements proved in this paper, 
apply to all the maps in $\QIS_\ga$, provided $\ga$ is of high type. 
Slightly modified renormalisation schemes are introduced for uni-singular maps in \cite{C14}, 
and for cubic maps in \cite{Yan15}. 
It is likely that suitable modifications of the arguments presented here may be applied in those settings as well.
\section{Comparing the changes of coordinates}
\label{S:coordinates-vs-model} 
Given $\ga \in (0,r_3]$ and $h\in \QIS_\ga$, the change of coordinate $\ex \circ \gF_h$ relates the iterates of $h$ 
to the iterates of $\rr(h)$. 
When studying repeated renormalisations, one needs to study long compositions of such changes of coordinates. 
As we shall do in later sections, it is convenient to consider suitable inverse branch of $\ex^{-1} \circ \gF_h^{-1}$, 
and study their compositions. 
In this section we show that $\ex^{-1} \circ \gF_h^{-1}$ behaves like the (model) map $Y_\ga$ 
introduced in \refS{S:arithmetic-model}. 

\subsection{Change of coordinates in the near-parabolic renormalisation}\label{SS:changes of coordinates} 
Let us fix an arbitrary $h: U_h \to \mathbb{C}$ in $\bigcup_{\ga \in (0,r_3]} \QIS_\ga$. 
Recall from Propositions~\ref{P:Ino-Shi1} and \ref{P:petal-geometry}, the Fatou coordinate 
\[\gF_h: \C{P}_h \to  \{w\in \D{C}  \mid 0 \leq \Re w \leq 1/\ga - c_1\}.\]
Also, recall from \refS{SS:renormalisation-def} the domain $S_h \subset \C{P}_h$ and the integer $k_h$ 
satisfying $h\co{k_h}(S_h) \subset \C{P}_h$.
Consider the set  
\begin{multline*}
\gP_h=\{w\in \D{C} \mid 0 \leq \Re w \leq 1/\ga-c_1, \Im w \geq -2\} 
\cup \{\gF_h(S_h)+l \mid l \in \D{Z}, 0 \leq l \leq k_h\}. 
\end{multline*}
The functional relation in \refP{P:Ino-Shi1}-(e) allows one to extend $\gF_h^{-1}$ onto $\gP_h$. 
For $w \in \gF_h(S_h)+l$ one defines $\gF_h^{-1}(w)=h\co{l} \circ\gF_h^{-1}(w-l)$. 
It follows from \refP{P:Ino-Shi1}-(e) that this is a well-defined holomorphic map, and 
satisfies $\gF_h^{-1}(w+1)= h \circ \gF_h^{-1}(w)$ whenever both sides are defined. 
However, $\gF_h^{-1}: \gP_h \to \D{C} \setminus \{0\}$ is not univalent any longer. 
It has a critical point which is mapped to $-4/27$. 

We may lift $\gF_h^{-1}: \gP_h \to \D{C}\setminus \{0\}$ via $\ex: \D{C} \to \D{C}\setminus \{0\}$ 
to define the holomorphic map
\begin{equation}\label{E:gc_h}
\gU_h=\ex^{-1} \circ \gF_h^{-1}: \gP_h \to \D{C}, \quad \gU_h(+1)=+1. 
\end{equation}

\subsection{Estimates on the change of coordinates}\label{SS:estimates-fatou-coordinates}
To understand the behaviour of $\gU_h$, we need some estimates on $\gF_h$. 
However, obtaining good estimates on $\gF_h$ are not trivial. 
Following \cite{Sh98}, a general idea is to compare $\gF_h$ to an explicit formula, as we explain below.

Recall that $h$ has two fixed points $0$ and $\gs_h$ on $\partial \C{P}_h$.
Consider the covering map 
\begin{equation*}\label{E:covering-formula}
\gt_h(\gz)= \frac{\gs_h}{1-e^{- 2 \pi \ga i \gz}}: \D{C} \to \hat{\D{C}} \setminus \{0, \gs_h\}.
\end{equation*}
Evidently, $\gt_h (\gz+1/\ga)=\gt_h(\gz)$, $\lim_{\Im \gz \to + \infty} \gt_h(\gz)=0$, 
and $\lim_{\Im \gz \to -\infty} \gt_h(\gz) = \gs_h$.

We may lift $\gF_h^{-1}: \Pi_h \to \D{C}\setminus \{0, \gs_h\}$ via the covering $\gt_h: \D{C} \to \hat{\D{C}} \setminus \{0, \gs_h\}$, to define 
\begin{equation*}\label{E:L_h-def}
L_h : \gP_h \to \D{C}. 
\end{equation*}
That is, $\gt_h \circ L_h = \gF_h^{-1}$ on $\gP_h$. 
However, this map is only determined up to translations by elements of $\D{Z}/\ga$. 
We choose the branch so that $L_h(\Pi_h)$ separates $0$ from $1/\ga$. 
Such branch exists because $L_h(\Pi_h)  \cap (\D{Z}/\ga)=\emptyset$, 
and $\gt_h^{-1}(\C{P}_h)$ is a $1/\ga$-periodic set, whose components are 
simply connected regions in $\D{C}\setminus (\D{Z}/\ga)$ which spread
from $+i \infty$ to $-i \infty$. 

Estimates on $L_h$ lead to estimates on $\gF_h$ through the explicit formula $\gt_h$. 
One may employ classic distortion estimates on univalent mappings from complex analysis 
to derive some estimates on $L_h$. 
One may refer to \cite{Sh98,Yoc95} for some general results about this. 
For specific estimates on $L_h$, one may refer to \cite{IS06}, \cite[Section 5]{Che13} and \cite[Section 6]{Che19}.
Below, we present only the estimates we need in this paper (these are only used in \refS{S:coordinates-vs-model}). 

\begin{propo}\label{P:estimates-L_h}
There is a constant $C_1$ such that for all $h \in \bigcup_{\ga \in (0,r_3]} \QIS_\ga$, we have 
\begin{itemize}
\item[(i)] for all $w \in \gP_h$ with $\Im w \geq 1$, $|L'_h(w)-1|\leq C_1 / \Im w$, 
\item[(ii)] for all $w \in \gP_h$, $|L_h(w) - w| \leq C_1 \log (1/\ga)$,
\item[(iii)] as $\Im w \to +\infty$ in $\gP_h$, $L_h(w)-w$ tends to a constant, 
\item[(iv)] for all $w \in \gP_h$, $|L_h(w) - w| \leq C_1 \log (2+ d(w, \D{Z}/\ga))$.
\end{itemize}
\end{propo}

Part (i) is an application of the Koebe distortion theorem and the functional equation for the Fatou coordinates, 
see for instance \cite[Lemma 6.7-(4)]{Che19}.
For (ii) see \cite[Proposition 6.19]{Che19}, and for (iii) see \cite[Lemma 6.9]{Che19}. 
To get (iv), when $\Im w \geq 1/\ga$ one uses part (ii), and when $\Im w < 1/\ga$ one 
uses \cite[Proposition 6.15, Proposition 6.17]{Che19}, integrates the bound in part (i), and uses the inequality 
$\log(a) + \log (b) \leq 2 \log (a+b)$ for all $a, b > 0$. 

In \cite{Che13}, quasi-conformal methods have been employed to obtain an exponentially decaying estimate 
on $|L_h'(w)-1|$. We do not need that fine estimate here. 

By \refP{P:estimates-L_h}-(i), and differentiation of the explicit formula $\tau_h$, we get
\begin{equation}\label{E:gU'-asymptote}
\lim_{\Im w \to +\infty, w\in \gP_h}  \gU_h'(w)= \ga.
\end{equation} 

\subsection{Dropping the non-linearity}
\label{SS:model-vs-change-of-coordinate}
Recall that $Y_\ga$ is defined on $\D{H}'=\{w\in \D{C} \mid \Im w> -1\}$. 
We aim to compare $\gU_h$ to $Y_\ga$, knowing that they have different domains of definitions. 
Below, we state a general form of such estimates, and later apply it to more specific domains. 

\begin{propo}\label{P:gc-vs-Y-value}
There is a constant $C_3$ such that for all $h \in \bigcup_{\ga \in (0,r_3]} \QIS_\ga$, all $w_1 \in \D{H}'$ and all $w_2 \in \gP_h$, we have 
\[|\gU_h(w_2) - Y_\ga(w_1) |  \leq C_3 \max \{1, |w_1-w_2|\}.\]
\end{propo}

\begin{proof}
We shall use the decomposition of $\gU_h$ as $\ex^{-1} \circ \gt_h \circ L_h$ on $\Pi_h$. 
Let $g_h= \ex^{-1} \circ \gt_h$. We divide the proof into two steps. 

\medskip

{\em Step 1.} 
There is a constant $D_1$, independent of $\ga$ and $h$, such that for every $w_1 \in \D{H}'$, we have 
\begin{equation}\label{E:P:gc-vs-Y-value-1}
|g_h(w_1+1/2+3i/2)-Y_\ga(w_1)|\leq D_1.
\end{equation}

\smallskip 

In the above inequality, the constant $1/2+ 3i/2$ is chosen to make sure that $g_h$ is defined, and also to 
simplify the calculations. Let us set $w_1'=w_1+1/2+3i/2$.
By explicit calculations, 
\begin{equation}\label{E:P:gc-vs-Y-value-5}
\Im g_h(w'_1) - \Im Y_\ga(w_1)
= \frac{1}{2\pi} \log \Big |\frac{4 (e^{-3\pi\ga}-e^{\pi\ga i})}{27 e^{-3\pi \ga} \gs_h}\Big|. 
\end{equation}
It follows from the Koebe distortion theorem that $\{h''(0)\mid h \in \IS_0\}$ is relatively compact in 
$\D{C} \setminus \{0\}$ (see \cite{IS06} for more details). This implies that there is a constant $D$, independent of 
$\ga$ and $h$, such that $\ga/D \leq  |\gs_h| \leq D\ga$. 
On the other hand, for all $\ga \in (0, 1/2]$, we have 
\begin{equation}
|e^{-3\pi \ga}-e^{\pi \ga i}| \leq |e^{-3\pi \ga} -1| + |1-e^{\pi \ga i}| \leq 3\pi \ga +\pi \ga= 4\pi \ga, 
\end{equation}
and 
\begin{equation} 
|e^{-3\pi \ga} - e^{\pi \ga i}| \geq |\Im (e^{-3\pi \ga}-e^{\pi \ga i})| =\sin (\pi \ga) \geq \pi \ga/2.
\end{equation}
These imply that $\pi/(2D) \leq |(e^{-3\pi\ga}-e^{\pi\ga i})/\gs_h| \leq 4\pi D$. 
Combining with $e^{-3\pi/2} \leq e^{-3\pi\ga} \leq 1$, we conclude that the left hand side of \refE{E:P:gc-vs-Y-value-5} is uniformly bounded from above 
and below. 

On the other hand, $g_h$ maps $\{w \in \D{C} \mid 0\leq  \Re w \leq 1/\ga, \Im w\geq 1/2\}$ into a vertical strip 
of width $+1$ whose projection onto the real axis contains $+1$. Similarly, $Y_\ga$ maps 
$\{w\in \D{H}' \mid 0 \leq \Re w \leq 1/\ga\}$ into the vertical strip $\{w\in \D{C} \mid 0 \leq \Re w \leq 1\}$. 
Using $Y_\ga(w+1/\ga)= Y_\ga(w)+1$ and $g_h(w+1/\ga)=g_h(w)+1$, one concludes that 
\begin{align*}
|\Re g_h(w'_1) - \Re Y_\ga(w_1)| & \leq | \Re g_h(w'_1)-\Re Y_\ga(w'_1)| + |\Re Y_\ga(w'_1)-\Re Y_\ga(w_1)| \\
& \leq 2+ \ga/2 \leq 9/4.
\end{align*}
This completes the proof of Step 1.   

\medskip

{\em Step 2.} There is a constant $D_2$, independent of $\ga$ and $h$, such that for all $w_3 \in \Pi_h$ 
and all $w_4 \in \D{C}$ with $\Im w_4 \geq 1/2$ we have 
\begin{equation}\label{E:P:gc-vs-Y-value-2}
\big|g_h \circ L_h(w_3) - g_h(w_4) \big| \leq D_2 \max \{1, |w_3-w_4|\}.
\end{equation} 

\smallskip 
 
Recall the constant $C_1$ from \refP{P:estimates-L_h}, and choose a constant $D$ such that 
$D /\ga - C_1 \log (1+1/\ga) \geq 1/\ga$ for all $\ga \in (0, 1/2]$.  
We break the proof into three cases. 

$\bullet$ $\Im w_3 \geq D/\ga$: 
By \refP{P:estimates-L_h}-(ii), $|L_h(w_3)-w_3|\leq C_1 \log (1+1/\ga)$, and hence 
$\Im L_h(w_3)\geq 1/\ga$. 
For $\Im w \geq 1/\ga$, $|g_h'(w)|=O(\ga)$, and for $\Im w\geq 1/2$, $|g_h'(w)|= O(1)$, with uniform constant in 
$O$ independent of $\ga$, $h$ and $w$. 
Then, 
\begin{align*}
|g_h\circ L_h(w_3)- g_h(w_4)| & \leq |g_h \circ L_h(w_3) - g_h(w_3)|+ |g_h(w_3)-g_h(w_4)| \\
& \leq O(\ga) \cdot |L_h(w_3)-w_3|+ O(1) \cdot |w_3-w_4| \\
& \leq O(\ga) \cdot O(\log (1/\ga))+ O(|w_3-w_4|) \leq O(1) + O(|w_3-w_4|).
\end{align*}

$\bullet$ $1/2\leq \Im w_3 \leq D/\ga$: 
By \refP{P:estimates-L_h}-(iv), $|L_h(w_3)-w_3| \leq C_1 \log (2+ d(w_3, \D{Z}/\ga))$.  
For $1/2 \leq \Im w \leq D/\ga$, $|g_h'(w)|=O(1/d(w, \D{Z}/\ga))$. 
Because $\log (1+d(w,\D{Z}/\ga)) \cdot 1/d(w,\D{Z}/\ga)$ is uniformly bounded from above, 
we conclude that $|g_h\circ L_h(w_3)- g_h(w_3)|$ is uniformly bounded from above. 
As in the above equation, one obtains the desired inequality in this case as well. 

$\bullet$ $\Im w_3\leq 1/2$: Let $w_3'=\Re w_3+i/2$. 
It follows from the Koebe distortion theorem that $\Im w_3$ is uniformly bounded from below (see proof of 
Proposition 2.7 in \cite{Che19}), and hence $|w_3' -w_3| = O (1)$. 
As in the previous cases, we get $|g_h \circ L_h(w_3) - g_h(w_3')| = O(1)$, and therefore 
\begin{align*}
|g_h \circ L_h(w_3) - g_h(w_4)|  & \leq |g_h \circ L_h(w_3) - g_h(w_3')|+ |g_h(w_3') - g_h(w_4)| \\
& = O(1) + O(1) |w_3'-w_4| \\
& = O(1)+O(1+ |w_3-w_4|).
\end{align*}

\medskip

To complete the proof of the proposition, one uses \eqref{E:P:gc-vs-Y-value-1}, \eqref{E:P:gc-vs-Y-value-2}, 
and the triangle inequality  
\begin{equation*}
\big|\gU_h (w_2) - Y_\ga(w_1)\big|  
 \leq \big|g_h \circ L_h(w_2) - g_h(w'_1)\big| + \big|g_h(w'_1) - Y_\ga(w_1)\big|.\qedhere
\end{equation*}
\end{proof}

\begin{propo}\label{P:asymptotics-similar}
For all $\ga \in (0,r_3]$ and all $h \in \QIS_\ga$,
\[\lim_{\Im w\to +\infty; w\in \Pi_h} (\gU_h(w) - Y_\ga(w)),\]
exists and is finite. 
\end{propo}

\begin{proof}
Recall that $\gU_h= \ex^{-1} \circ \gt_h \circ L_h$ on $\Pi_h$, and define $g_h= \ex^{-1} \circ \gt_h$. 
By elementary calculations one may see that $g_h(w) - Y_\ga(w)$ tends to a finite constant as $\Im w \to +\infty$.
Also, for all $w_1$ and $w_2\in \D{C}$, $g_h(w_1+w_2) - g_h(w_1)$ tends to a finite constant  
as $\Im w_1 \to +\infty$. 
On the other hand, by \refP{P:estimates-L_h}-(iii), $L_h(w) - w$ tends to a constant as $\Im w \to +\infty$ within 
$\Pi_h$. 
\end{proof} 
\section{Marked critical curve}\label{S:parametrised curve}
In this section we identify a collection of simple curves with special parametrisation (marking), which satisfy 
some geometric and equivariant properties under the renormalisation. 
These will be employed to build partitions of the post-critical set in \refS{S:modified nest}. 

\subsection{Repeated renormalisations}\label{SS:sequence-of-renormalisations}
Fix an arbitrary $\ga \in \D{R}\setminus \D{Q}$, and let $(\ga_i)_{i\geq 0}$ denote the sequence generated 
by the modified continued fraction algorithm in \refS{SS:modified-fractions}.  
Recall the complex conjugation map $s(z)=\ol{z}$. 

\begin{propo}\label{P:sequence-renormalisations}
There exists a positive integer $N$ such that for all $\ga$ in $\irr$ and all $f$ in $\QIS_\ga$, 
the following sequence of maps is defined  
\[f_0 = 
\begin{cases}
f & \tif \gep_0=+1, \\
s \circ f \circ s & \tif \gep_0=-1,
\end{cases}
 \qquad f_{n+1}= 
\begin{cases}
\rr (f_n) & \text{ if } \gep _{n+1}=-1, \\
s \circ \rr(f_n) \circ s &   \text{ if } \gep_{n+1}=+1, 
\end{cases}\] 
and for all $n\geq 0$ we have $f_0\in \QIS_{\ga_0}$, $f_{n+1} \in \IS_{\ga_{n+1}}$, $f_n: U_{f_n} \to  \D{C}$, 
$f_n(0)=0$, $f_n'(0)=e^{2\pi i \ga_n}$.
\end{propo}

\begin{proof}
Let $N \geq 1/r_3+1/2$, where $r_3$ is introduced in \refT{T:Ino-Shi2}.
Assume that $\ga$ has modified continued fraction $a_{-1} + \gep_0/(a_0 + \gep_1/(a_1+ \dots ))$ with $a_i\geq N$, 
for all $i\geq 0$. 
For all $i\geq 0$, 
\[1/\ga_i=a_i + \gep_{i+1}\ga_{i+1} \geq N - 1/2 \geq 1/r_3,\]
and hence $\ga_{i}\in (0,r_3]$. 

First we note that for every $\gamma\in \mathbb{R}$, $h \in \QIS_\gamma$ iff $s \circ h \circ s \in \QIS_{-\gamma}$. 
When $h=Q_\gamma$, we have $s \circ Q_\gamma \circ s= Q_{-\gamma}$.
So assume that $h \in \IS_\gamma$, with $h(z)=P \circ \gy^{-1} (e^{2\pi i \gamma} z)$. 
Since, $s \circ P= P \circ s$, and $s(U)=U$, we have 
\[s \circ h \circ s(z)= s \circ P \circ \gy^{-1} (e^{2\pi i \gamma} \ol{z})
=P \circ s \circ \gy^{-1} \circ s(e^{-2\pi i \gamma} z),\]
where $s \circ \gy \circ s: U \to \D{C}$ is holomorphic, maps $0$ to $0$, and has derivative $+1$ at $0$. 

If $\gep_0=+1$, then $\ga= a_{-1}+ \ga_0$, and hence $f_0=f\in \QIS_\ga=\QIS_{\ga_0}$.  
If $\gep_0=-1$, then $\ga= a_{-1} - \ga_0$, and hence $f\in \QIS_\ga= \QIS_{-\ga_0}$. Then, by 
the above paragraph, $f_0= s \circ f \circ s \in \QIS_{\ga_0}$. 

Now assume that $f_n$ is defined and belongs to $\IS_{\ga_n}$. 
Since $\ga_n\in (0,r_3]$, by \refT{T:Ino-Shi2}, $\rr(f_n)$ is defined and belongs to $\IS_{-1/\ga_n}$. 
Recall that $-1/\ga_{n}= -a_{n} -\gep_{n+1} \ga_{n+1}$, which gives $\rr(f_n) \in \IS_{-\gep_{n+1}\ga_{n+1}}$. 
If $\gep_{n+1}=+1$, by the above paragraph, $f_{n+1}=s \circ \rr(f_n) \circ s \in \IS_{\ga_{n+1}}$. 
If $\gep_{n+1}=-1$, then $f_{n+1}= \rr(f_n) \in \IS_{\ga_{n+1}}$.
\end{proof}

\subsection{Successive changes of coordinates}\label{SS:coordinates-tower}
Recall $\Pi_h$, $\gF_h^{-1}: \Pi_h \to \D{C}\setminus \{0\}$, and $\gU_h:\Pi_h \to \D{C}$ from 
\refS{SS:changes of coordinates}. 
For $n\geq 0$, we use the notations 
\begin{equation}\label{E:gU-defn}
\gU_{n}=
\begin{cases}
\ex^{-1} \circ \gF_{n}^{-1} & \tif  \gep_{n}=-1, \\
\ex^{-1} \circ s \circ \gF_{n}^{-1} & \tif  \gep_{n}=+1,
\end{cases}, 
\quad \Pi_n=\Pi_{f_n}, \quad \gF_n^{-1}= \gF_{f_n}^{-1}: \Pi_n \to \D{C}\setminus \{0\}. 
\end{equation}
with the normalisations 
\begin{equation}\label{E:gc_n(0)}
\gU_n(+1)=+1. 
\end{equation}
\begin{figure}[h]
\centering
\begin{minipage}{.4\textwidth}
\[\xymatrix{
& \D{C} \ar[dl]_\ex \\ 
\mathbb{C}\setminus \{0\} & \gP_n \ar[l]^{\, \, \, \,\, s \circ \gF_n^{-1}} \ar[u]_{\gU_n}}
\]   
\end{minipage}
\qquad
\begin{minipage}{.4\textwidth}
\[\xymatrix{
& \D{C} \ar[dl]_\ex \\ 
\mathbb{C}\setminus \{0\} & \gP_n \ar[l]^{\, \, \, \, \, \, \gF_n^{-1}} \ar[u]_{\gU_n}}
\]
\end{minipage}
\caption{The change of coordinate $\gU_n$, with $\gep_n=+1$ on the left and $\gep_n=-1$ on the right.}
\label{F:successive-coordinates}
\end{figure}
Each $\gU_n$ is either holomorphic or anti-holomorphic. See \refF{F:successive-coordinates}.   

Consider the set  
\[\gP= \{w\in \D{C}   \mid 1/2 \leq \Re w \leq 3/2, \Im w\geq -2\}.\]

\begin{lem}\label{L:well-inside-lifts}
There is a constant $\gd_1>0$ such that for all $h \in \bigcup_{\ga \in (0, r_3]}\QIS_\ga$, we have
\[B_{\gd_1} (\gU_h(\gP)) \subseteq \gP.\] 
\end{lem}

\begin{proof}
By \cite{IS06}, for every $h \in \QIS_0$, $A_h^{-k}$ and $B_h^{-k}$ are contained in the repelling Fatou 
coordinate of $h$ for large enough $k$, and hence they are defined for all values of $k\geq 0$. 
Comparing to their notations, $A_h \cup B_h$ is contained in the union 
$\psi_0(D_0) \cup \psi_0(D_1) \cup \psi_0(D_0^{\sharp}) \cup \psi_0(D_1^{\sharp})$, where $\psi_0(z)=-4/z$. 
See Section 5.A--Outline of the proof in \cite{IS06}. 
They prove in Propositions~5.6 and 5.7 that the closure of
$D_0 \cup D_1 \cup D_0^{\sharp} \cup D_1^{\sharp}$ does not intersect the negative real axis. 
In particular, it follows that for all $z\in A_h \cup B_h$, $d(\Re \ex^{-1}(z), \D{Z}) < 1/2$. 
By the compactness of $\QIS_0$, there is a constant $C<1/2$ such that for all 
$h\in \QIS_0$ and all $z\in A_h \cup B_h$, $d(\Re \ex^{-1}(z), \D{Z}) < C$. 
Then, by the continuous dependence of the Fatou coordinate on the map, one may see that there is $C'< 1/2$
such that for all small enough $\ga$, all $h\in \QIS_\alpha$, and all $z\in A_h \cup B_h$, 
$d(\Re \ex^{-1}(z), \D{Z}) < C' < 1/2$. 
Since, $\gF_h^{-1}(\Pi)= A_h \cup B_h$, we conclude that there is $\gd_1 > 0$ such that 
$B_{\gd_1}(\gU_h(\Pi))$ is contained in the vertical strip  $1/2 \leq \Re w \leq  3/2$.
On the other hand, $A_h \cup B_h$ is contained in the image of $h$, and the image of $h$ is contained well-inside 
the disk of radius $4 e^{4\pi}/27$ centred at $0$. Therefore, by making $\gd_1$ smaller if necessary, 
$\Im (B_{\delta_1}(\gU_h(\gP)) \subset [-2, +\infty)$. 

In \cite{IS06}, the constant $r_3$ in \refT{T:Ino-Shi2} is obtained from a continuity property of the locations of 
$D_0 \cup D_1 \cup D_0^{\sharp} \cup D_1^{\sharp}$ with respect to $h$. 
As such, it is implicitly assumed that the inclusion in the lemma also holds for small perturbations of 
$h \in \QIS_0$. Because of this we do not introduce a new constant for small enough $\ga$, and assume that the same 
constant $r_3$ works here as well.
\end{proof}

Let $\gr(z) |dz|$ denote the Poincar\'e metric of constant curvature $-1$ on $\interior(\gP)$.  
By classic complex analysis, when $f: \D{D} \to \D{D}$ is holomorphic with $f(\D{D})$ compactly contained in 
$\D{D}$, $f$ is uniformly contracting with respect to the Poincar\'e metric on $\D{D}$. 
However, here $\gU_n(\Pi)$ is not compactly contained in $\Pi$. 
Despite that, the uniform space provided by \refL{L:well-inside-lifts} allows us to recover the uniform contraction, 
as we discuss below. 
Let $\gU_n^* \gr$ denote the pull-back of $\gr$ on $\interior(\gP)$ by $\gU_n: \gP \to \gP$. 

\begin{propo}\label{P:uniform-contraction}
There is a constant $\gd_2 < 1$ such that for every $n \geq 0$ and every $w \in \interior(\gP)$, 
\[(\gU_n^* \gr)(w) \leq \gd_2 \gr(w).\]
\end{propo}

\begin{proof}
Let $\rho_n(z)|d z|$ denote the Poincar\'e metric on $\gU_n(\interior(\gP))$. 
By the Schwartz-Pick Lemma, $\gU_n: (\interior(\Pi), \gr) \to (\gU_n(\interior(\Pi)), \rho_n)$ is non-expanding. 
It is enough to show that the inclusion map from $(\gU_n(\interior(\Pi)), \rho_n)$ to $(\interior(\Pi), \gr)$ is 
uniformly contracting.   

Fix an arbitrary $\xi_0$ in $\gU_n(\Pi)$. 
Using $\gd_1$ from \refL{L:well-inside-lifts}, consider $H: \gU_n(\Pi) \to \D{C}$ defined as 
\[H(\xi)=\xi+\frac{\delta_1 (\xi-\xi_0)}{\xi-\xi_0+2}.\]
For $\gx \in \gU_n(\Pi)$ we have $|\Re(\xi-\xi_0)| < 1$. 
This implies that $|\xi-\xi_0|<|\xi-\xi_0+2|$, and hence  $|H(\xi)-\xi| < \delta_1$.
It follows from \refL{L:well-inside-lifts} that $H$ maps $\gU_n(\Pi)$ into  $\Pi$. 
By Schwartz-Pick Lemma, $H$ is non-expanding with respect to the corresponding metrics. 
In particular, at $H(\xi_0)=\xi_0$,  
\[\rho(\xi_0) |H'(\xi_0)|= \rho(\xi_0)(1+\delta_1/2) \leq \rho_n(\xi_0).\]
Hence,  
\[\rho(\xi_0) \leq \left(\frac{2}{2+\delta_1} \right) \rho_n(\xi_0).\]
The uniform (independent of $n$) contraction factor is $\gd_2=2/(2+\gd_1)$. 
\end{proof}

\subsection{Critical curve}\label{SS:u_n-inductive-definition}
Inductively, we define the curves 
\[u_n^j: i [0, +\infty) \to \Pi,\] 
for $j\geq 0$ and $n\geq -1$. 
For $j=0$ and all $n \geq -1$, let $u_n^0(it)= 1+  i  t$.
Assume that for some $j \geq 0$, and all $n\geq -1$, $u_n^j$ is defined. Then, for all $n\geq -1$, let 
\begin{equation}\label{E:u_n^j-induction}
u_n^{j+1}= \gU_{n+1} \circ  u_{n+1}^j \circ  Y_{n+1}^{-1}. 
\end{equation}

\begin{lem}\label{L:well-defined-u_n^j}
For all $n\geq -1$ and $j\geq 0$, $u_n^j: i [0, +\infty) \to \Pi$ is a well-defined analytic map satisfying $u_n^j(0)=+1$. 
\end{lem}

\begin{proof}
Recall that each $\gU_n$ is either holomorphic or anti holomorphic, and each $Y_n$ is real analytic. 
By \refE{E:invariant-imaginary-line}, for every $t\geq 0$, $\Im Y_{n+1}^{-1}(it) \in  [0, +\infty)$, and by 
\refL{L:well-inside-lifts}, $\gU_n(\gP) \subset \gP$. 
These imply that each $u_n^j$ is well-defined and analytic. 
Also, since $Y_{n+1}(0)=0$ and $\gU_{n+1}(+1)=+1$, for all $n \geq -1$, inductively, one concludes that 
$u_n^{j}(0)=1$. 
\end{proof}

Recall the constant $\gd_2$ introduced in \refP{P:uniform-contraction}. 

\begin{propo}\label{P:u_n-top}
There is a constant $C_2$ such that for all $n \geq -1$, all $j\geq 0$, and all $t \geq 0$, 
\[|u_n^{j+1}(it) - u_n^{j}(it)| \leq C_2 (\gd_2)^j.\]
In particular, for every $n\geq -1$, as $j \to +\infty$, $u_n^j$ converges to a continuous map 
$u_n: i[0, +\infty) \to \Pi$. 
\end{propo}

\begin{proof}
By \refP{P:gc-vs-Y-value}, with $\ga=\ga_{n+1}$, $h=f_{n+1}$, $w_1=Y_{n+1}^{-1}(it)$ and 
$w_2=Y_{n+1}^{-1}(it)+1$,
\begin{align*}
|u_n^1(it) - u_n^0(it)| &= |\gU_{n+1} (Y_{n+1}^{-1}(it)+1) - (it+1)| \\
& \leq |\gU_{n+1} (Y_{n+1}^{-1}(it)+1) - Y_{n+1}(Y_{n+1}^{-1}(it))|+1 
\leq C_3+1.
\end{align*}
Recall the Poincar\'e metric $\gr(z) |dz|$ on $\interior(\gP)$. 
One has the classic bounds $1/(2d(z, \partial \Pi)) \leq \gr(z) \leq 2/d(z, \partial \gP)$. 
In particular, $\gr\geq 1$ on $\interior(\Pi)$. 
On the other hand, by \refL{L:well-inside-lifts}, $d(u_n^1(it), \partial \Pi)\geq \gd_1$ which implies that 
$d_{\gr}(u_n^1(it), u_n^0(it)) \leq 2 (C_3+1)/\gd_1$. 
Now we apply \refP{P:uniform-contraction}, to see that for $j\geq 1$, 
\begin{align*}
d_\gr(u_n^{j+1}(it), u_n^j(it))  
&=d_\gr \left(\gU_{n+1} \circ u_{n+1}^j \circ Y_{n+1}^{-1}(it), 
\gU_{n+1} \circ u_{n+1}^{j-1} \circ Y_{n+1}^{-1}(it)\right)  \\
& \leq \gd_2 d_\gr \left( u_{n+1}^j\left (Y_{n+1}^{-1}(it)\right ),u_{n+1}^{j-1}\left (Y_{n+1}^{-1}(it)\right) \right). 
\end{align*}
Then, by induction, 
\begin{align*}
d_\gr(u_n^{j+1}(it), u_n^j(it)) \leq (\gd_2)^{j} d_\gr \left(u_{n+j}^1(it'), u_{n+j}^0(it')\right),
\end{align*}
where $it'= Y_{n+j}^{-1} \circ \dots \circ Y_{n+1}^{-1}(it)$. 
Therefore, as $\gr \geq 1$ on $\interior(\Pi)$, 
\begin{equation}\label{E:u_n^j+1-u_n^j}
|u_n^{j+1}(it) - u_n^{j}(it)|  \leq d_\gr(u_n^{j+1}(it), u_n^j(it)) 
\leq  (\gd_2)^{j} d_\gr (u_{n+j}^1(it'), u_{n+j}^0(it')) 
\leq  (\gd_2)^{j} 2(C_3+1)/\gd_1.
\end{equation}
For each $n\geq -1$, $u_n^j$ forms a Cauchy sequence on $i [0, +\infty)$, which implies that 
$u_n^j$ converges to a continuous map $u_n$, as $j \to +\infty$. 
\end{proof}

\begin{propo}\label{P:comm-u-Y-upsilon}
For all $n \geq 0$ and all $t\geq 0$ we have 
\[\gU_n \circ u_n(it)= u_{n-1} \circ  Y_n ( i  t), \; \tand \; u_n(0)=1.\]
\end{propo}

\begin{proof}
These follow from taking limits as $j \to +\infty$ in \refE{E:u_n^j-induction} and \refL{L:well-defined-u_n^j}. 
\end{proof}

\begin{propo}\label{P:near-identity-u_n}
For every $n\geq -1$, and every $t\geq 0$, we have 
\[|u_n(it)- (1+ i  t)| \leq C_2/(1-\gd_2).\]
\end{propo}

\begin{proof}
By \refP{P:u_n-top}, for $j\geq 1$ and $t\geq 0$, we have 
\begin{equation}\label{E:u_n^j-near-identity}
\begin{aligned}
|u_n^j(it)-(1+ i  t)| & = \left | \textstyle{\sum_{l=1}^j} (u_n^l(it) - u_n^{l-1}(it)) \right| 
 \leq \textstyle{\sum_{l=1}^j} C_2 (\gd_2)^{l-1} \leq C_2/(1-\gd_2).
\end{aligned}
\end{equation}
Taking limit as $j\to + \infty$, we conclude the inequality in the proposition. 
\end{proof}

\begin{propo}\label{P:injective-u_n}
For every $n\geq -1$, $u_n: i [0, +\infty) \to \Pi$ is injective. 
\end{propo}

\begin{proof}
Fix an arbitrary $n\geq -1$. 
Let $0\leq t_n < s_n$  be arbitrary real values. 
Inductively, define the sequence of numbers $t_{l+1}= \Im Y_{l+1}^{-1}( i  t_l)$ and 
$s_{l+1}= \Im Y_{l+1}^{-1}( i  s_l)$, for $l \geq n$. 
By \refE{E:uniform-contraction-Y}, for $l \geq n$, $|t_l-s_l| \geq (10/9)^{l-n} |t_n-s_n|$. 
In particular, for large enough $l$, $|t_l - s_l| \geq 3 C_2/(1-\gd_2)$. 
By virtue of \refP{P:near-identity-u_n}, this implies that $u_l(i t_l) \neq u_l(i s_l)$.  
Now, inductively using the commutative relation in \refP{P:comm-u-Y-upsilon}, and the injectivity of $Y_k$ 
and $\gU_k$ for all $k$, we conclude that $u_n(i t_n) \neq u_n(i s_n)$. 
\end{proof}

\begin{propo}\label{P:u_n-asymptotics}
For every $n\geq -1$, $\lim_{t\to +\infty} (u_n(it)- (1+it))$ exists and is finite.
\end{propo}

\begin{proof}
By \refP{P:asymptotics-similar} and the explicit formula for $Y_{n+1}$, the following limit exists and is finite 
\begin{align*}
\lim_{t \to + \infty} (u_n^1(it)- u_n^0(it))
& = \lim_{t \to + \infty}  \left (\gU_{n+1}(Y_{n+1}^{-1}(it)+1) - (1+it) \right )    \\
&=  \lim_{t \to + \infty}  \left (\gU_{n+1}(Y_{n+1}^{-1}(it)+1) -   Y_{n+1}(Y_{n+1}^{-1}(it)+1) \right) \\
& \qquad + \lim_{t \to + \infty}  \left (Y_{n+1}(Y_{n+1}^{-1}(it)+1) - Y_{n+1} (Y_{n+1}^{-1}(it)) \right) -1. 
\end{align*}
By an inductive argument, one may see that for every $j\geq 1$, $\lim_{t \to + \infty} (u_n^j(it)- u_n^{j-1}(it))$ 
exists and is finite. 
Indeed, by \refE{E:u_n^j+1-u_n^j}, the absolute value of this limit is bounded from above by 
$(\gd_2)^{j-1} 2(C_3+1)/\gd_1$.  
It follows that $\lim_{t \to + \infty} (u_n^j(it)- (1+it))$ exists and is finite. 
Since $u_n^j$ converges to $u_n$ uniformly on $i [0,+\infty)$, we conclude the proposition.  
\end{proof}

\begin{rem}
By the argument in this section the limiting curves $u_n$ and their parametrisations 
do not depend on the particular choice of $u_n^0$. 
Any other choice for $u_n^0$ which lies within some uniform distance from $u_n^0$ leads to the same limiting 
curve $u_n$. 
For this reason, one may see that when $\ga \in \E{B}$, the intersection of $\gF_n^{-1}(u_n)$ and the 
Siegel disk of $f_n$ coincides with an internal ray of the Siegel disk of $f_n$. 
\end{rem}

\subsection{An equivariant extension of the critical curve}\label{SS:auxiliary-part-u_n}
We need to extend the curves $u_n$ at the end points $u_n(0)$, while maintaining the functional relation 
in \refP{P:comm-u-Y-upsilon}. 
There are many choices for such extensions, as we present the details below. 
This is rather arbitrary, and will be only used for technical aspects in the later parts of the paper. 

Let us define the numbers $t_n^j$, for $n\geq -1$ and $j\geq 0$ according to 
\[t_n^0= -1,\; \tfor n\geq -1, \q \tand \q  t_n^j= \Im Y_{n+1}( i  t_{n+1}^{j-1}), \; \tfor n\geq -1 \tand j\geq 1.\] 

\begin{lem}\label{L:t_n^j-monotone}
For every $n\geq -1$, we have $t_n^0 < t_n^1 < t_n^2 < \dots < 0$ with $t_n^j \to 0$ as $j \to + \infty$. 
\end{lem}

\begin{proof}
By \refP{P:coordinate-properties}-(i) and the definition of $Y_n$ in \refE{E:Y_n}, for all $n\geq -1$, 
$\Im Y_{n+1}(- i )> -1$.  This implies that $t_n^0 < t_n^1$, for all $n\geq -1$. 
Since each $Y_n$ is injective and maps $ i  [-1, +\infty)$ into itself, $t \mapsto \Im Y_n ( i  t)$ is order preserving, 
for all $n\geq 0$. This implies that for all $n\geq -1$ and $j\geq 0$, $t_n^j < t_n^{j+1}$. 
Also, by \refE{E:uniform-contraction-Y}, $|t_n^j|\leq (9/10)^j$, which implies the latter part of the lemma. 
\end{proof}

Recall the set $\gP_n=\gP_{f_n}$ defined in \refS{SS:coordinates-tower}. 

\begin{lem}\label{L:tail-unit-u_n-bottom}
For each $n\geq -1$, there is a continuous and injective map
\[u_n^0: i [t_n^0, t_n^1] \to \gP \setminus \interior \left(\gU_{n+1}(\Pi_{n+1})\right )\] 
such that 
\begin{itemize}
\item[(i)] $u_n^0(it_n^0)= 1-2  i$ and $u_n^0(i t_n^1)= \gU_{n+1}(1-2 i)$, 
\item[(ii)] $\sup\{\Im u_n^0(is)  \mid  n\geq -1, t_n^0 \leq s \leq t_n^1\} < + \infty$, 
\item[(iii)] $u_n^0(i(t_n^0, t_n^1)) \subset \interior (\gP \setminus \gU_{n+1}(\Pi_{n+1})$. 
\end{itemize}
\end{lem}

\begin{proof}
Recall that $\ex (1-2i)= -4 e^{4\pi}/27$, and $\ex(\gU_{n+1}(1-2i))$ is either equal to $\gF_{n+1}^{-1} (1-2i)$ or 
$s \circ \gF_{n+1}^{-1} (1-2i)$ depending on $\gep_{n+1}$. 
Also, recall that $\gF_{n+1}^{-1}(\Pi_{n+1})$ is a finite union of sectors bounded by analytic curves landing at $0$. 
Moreover, this set contains a punctured neighbourhood of $0$, is compactly contained in $B(0, 4e^{4\pi}/27)$, 
and $\gF_{n+1}^{-1}(1-2i)$ lies on its boundary. 
Let us assume that $\gep_{n+1}=-1$. 
There is a continuous curve $\gamma: [0,1] \to B(0, 4e^{4\pi}/27)$ such that 
$\gamma(0)=-4 e^{4\pi}/27$, $\gamma(1)=\gF_{n+1}^{-1} (1-2i)$, and $\gamma((0,1))$ does not meet the sets 
$\gF_{n+1}^{-1}(\Pi_{n+1})$ and $[0, +\infty)$. We may choose this curve to be uniformly away from $0$. 
One may lift the curve $\gamma$ via $\ex$ to define the desired curve $u_n^0$, which may be re-parametrised 
on $i[t_n^0,t_n^1]$. 

When $\gep_{n+1}=+1$, one only needs to insert the complex conjugation map $s$ in the 
appropriate places in the above argument. 
\end{proof}

Now, by induction on $j\geq 0$, we define the maps $u_n^j$ on $i[t_n^j, t_n^{j+1}]$, for all $n\geq -1$. 
Assume that for some $j \geq 0$ and all $n\geq -1$, $u_n^j$ is defined on $i[t_n^{j}, t_n^{j+1}]$. 
For all $n\geq -1$, we define $u_n^{j+1}$ on $i[t_n^{j+1}, t_n^{j+2}]$ as 
\begin{equation}\label{E:commutator-auxiliary-part-u_n}
u_n^{j+1}= \gU_{n+1} \circ u_{n+1}^j \circ Y_{n+1}^{-1}.
\end{equation}
Note that, by \refL{L:tail-unit-u_n-bottom}-(i),  
\[u_n^1(i t_n^1)=\gU_{n+1} \circ u_{n+1}^0 \circ Y_{n+1}^{-1}(i t_n^1)
=\gU_{n+1} \circ u_{n+1}^0 (it_{n+1}^0)= \gU_{n+1} (1-2i)= u_n^0(it_n^1).\]
In other words, $u_n^0$ and $u_n^1$ match at the intersection of their respective domains of definitions. 
Repeating the above argument inductively, one may see that for all $n\geq -1$ and $j\geq 0$, 
$u_n^{j+1}(i t_n^{j+1})=u_n^j(it_n^{j+1})$. 
Thus we may define $u_n$ on $i[-1, 0)$ as 
\[u_n(it)=u_n^j(it), \; \tfor t\in [t_n^j, t_n^{j+1}].\]
We set $u_n(0)=+1$, for each $n\geq -1$. 

\begin{lem}\label{L:u_n-bottom-injective-near-identity}
For every $n\geq -1$, $u_n : i [-1, 0] \to \gP$ is continuous. 
Moreover, there is $C_5>0$ such that for all $n\geq -1$ and all $t\in [-1,0]$ we have $|u_n(i t)- (1+i t)| \leq C_5$.
\end{lem}

\begin{proof}
Fix an arbitrary $n\geq -1$. By \refL{L:tail-unit-u_n-bottom} and \refE{E:commutator-auxiliary-part-u_n}, 
the restriction of $u_n$ to each closed interval $i[t_n^j, t_n^{j+1}]$ is continuous, for $j\geq 0$. 
Hence, $u_n$ is continuous on $i[-1, 0)$.

Fix an arbitrary $n\geq -1$ and an arbitrary $j\geq 1$. 
By \refL{L:tail-unit-u_n-bottom}-(ii), the Euclidean diameter of the curve $u_{n+j}(i[t_{n+j}^0, t_{n+j}^1])$ 
is uniformly bounded from above. 
By the compactness of $IS$, it follows that the Euclidean diameter of the curve 
$\gU_{n+j}(u_{n+j}(i[t_{n+j}^0, t_{n+j}^1]))$ is uniformly bounded from above, independent of $n+j$. 
On the other hand, this curve also lies in $\gU_{n+j}(\gP)$, which is contained well inside $\gP$, 
by \refL{L:well-inside-lifts}.
These imply that the hyperbolic diameter of the curve $\gU_{n+j}(u_{n+j}(i[t_{n+j}^0, t_{n+j}^1]))$ in 
$\interior(\gP)$ is uniformly bounded from above, independent of $n+j$. 
Then, we employ \refP{P:uniform-contraction} 
to conclude that there is a constant $C$, independent of $n$ and $j$, such that the hyperbolic diameter of the curve 
$\gU_{n+1} \circ \dots \circ \gU_{n+j}(u_{n+j}(i[t_{n+j}^0, t_{n+j}^1]))$ is bounded from above by 
$C (\gd_2)^{j-1}$. 
Similarly, by making $C$ larger if necessary, we also conclude that the hyperbolic distance from 
$\gU_{n+1} \circ \dots \circ \gU_{n+j}(u_{n+j}(it_{n+j}^1))$ to $+1$ is bounded from above by $C (\gd_2)^{j-1}$. 

By \refE{E:commutator-auxiliary-part-u_n}, 
\begin{equation*}
\begin{gathered}
\gU_{n+1} \circ \dots \circ \gU_{n+j}(u_{n+j}(i[-1, t_{n+j}^1])) = u_n(i[t_n^j, t_n^{j+1}]), \\
\gU_{n+1} \circ \dots \circ \gU_{n+j}(u_{n+j}(it_{n+j}^1)) = u_n(i t_n^{j+1}).
\end{gathered}
\end{equation*}
Thus, by the above paragraph, the 
hyperbolic diameter of $u_n(i[t_n^j, t_n^{j+1}])$ is bounded from above by $C (\gd_2)^{j-1}$, and the 
hyperbolic distance from $u_n(i t_n^{j+1})$ to $+1$ is bounded from above by $C (\gd_2)^{j-1}$. 
These imply that $u_n$ is continuous at $0$. Moreover, the hyperbolic
diameter of $u_n(i[t_n^1, 0])$ is bounded from above by $C \sum_{j=1}^\infty (\gd_2)^{j-1}= C/(1-\gd_2)$. 
In particular, combining with \refL{L:tail-unit-u_n-bottom}-(ii), we conclude that the Euclidean diameter of the curve 
$u_n(i[-1, 0])$ is uniformly bounded from above, independent of $n$. This implies the latter part of the lemma. 
\end{proof}

\begin{propo}\label{P:injective-u_n-2}
For every $n\geq -1$, $u_n: i [-1, +\infty) \to \Pi$ is injective. 
\end{propo}

\begin{proof}
Fix an arbitrary $n\geq -1$. We already proved in \refP{P:injective-u_n} that  $u_n$ is injective on $i [0, +\infty)$. 

Let us fix arbitrary points $y_n < x_n$ in $[-1, +\infty)$ with $y_n < 0$. 
We aim to show that $u_n(iy_n) \neq u_n(i x_n)$.
First assume that there is $j\geq 0$ such that both $x_n$ and $y_n$ belong to the same interval $[t_n^j, t_n^{j+1}]$. 
Then, $u_n$ on $i[t_n^j, t_n^{j+1}]$ is given by 
\[u_n= \gU_{n+1} \circ \dots   \circ   \gU_{n+j} \circ u_{n+j} \circ Y_{n+j}^{-1}  \circ  \dots \circ  Y_{n+1}^{-1},\]
while 
\[Y_{n+j}^{-1}  \circ  \dots \circ  Y_{n+1}^{-1} (i[t_n^j, t_n^{j+1}])= i[t_{n+j}^0, t_{n+j}^1].\]
However, each $\gU_l$ is injective on $\Pi$, $Y_l$ is injective on its domain, and by \refL{L:tail-unit-u_n-bottom}, 
$u_n$ is injective on $i[t_{n+j}^0, t_{n+j}^1]$.  
In particular, $u_n(iy_n) \neq u_n(i x_n)$. 

Now assume that both $x_n$ and $y_n$ do not belong to one interval $[t_n^j, t_n^{j+1}]$. 
By \refL{L:t_n^j-monotone}, there is $j\geq 0$ such that $t_n^j \leq y_n < t_n^{j+1} < x_n$. 
Let $y_{l+1}= \Im Y_{l+1}^{-1}( i  y_l)$ and $x_{l+1}= \Im Y_{l+1}^{-1}( i  x_l)$, for $n \leq l \leq n+j-1$. 
By the choice of $j$, we have $-1 \leq y_{n+j} < t_{n+j}^1 < x_{n+j}$. 
Then, by \refL{L:well-inside-lifts}, 
\[u_{n+j}( i  y_{n+j})\in \gP \setminus \gU_{n+j+1}(\gP_{n+j+1}),\]
while 
\[u_{n+j}(i x_{n+j}) = \gU_{n+j+1} \circ u_{n+j+1} \circ Y_{n+j+1}^{-1}(i x_{n+j}) 
\subset  \gU_{n+j+1} (\gP) \subset \gU_{n+j+1}(\gP_{n+j+1}).\] 
In particular, $u_{n+j}( i  x_{n+j}) \neq u_{n+j}( i  y_{n+j})$. 
Now, inductively one uses \refE{E:commutator-auxiliary-part-u_n}, and the injectivity of 
$Y_l$ and $\gU_l$, to conclude that $u_n( i  x_n)\neq u_n( i  y_n)$. 
\end{proof}

The particular choice of the curve $u_n$ on $[-1, 0]i$ and its parametrisation does not play any role in the sequel. 
But the feature we shall use is the equivariant property 
\begin{equation}\label{E:commutator-auxiliary-part-u_n-complete}
\gU_{n+1} \circ u_{n+1} (it)= u_n \circ Y_{n+1}(it), \; t\in [-1,0].
\end{equation}
\section{Marked dynamical partitions for the post-critical set}\label{S:modified nest}
In this section we build a nest of dynamical partitions for the post-critical set. 
The elements of the partition are Jordan domains, whose boundaries are equipped with a parameterisation (marking). 
The markings match where the boundaries of pieces meet. 
This is analogous to the Yoccoz puzzle pieces in polynomial dynamics. 
This nest of partitions is similar to the nest of partitions $\Omega_n^j$ introduced in \cite{Che13,Che19,AC18}. 
However, the nest introduced here has simpler combinatorial and geometric features, and enjoys equivariant 
properties with respect to the renormalisation. 

\subsection{Marked curves \texorpdfstring{$w_n^{\pm}$}{w-n} and \texorpdfstring{$v_n^{\pm}$}{v-n}}
\label{SS:w_n-and-v_n}
Recall the curves $u_n: i[-1, +\infty) \to \gP$, for $n\geq -1$, defined in Sections \ref{SS:u_n-inductive-definition} and
\ref{SS:auxiliary-part-u_n}, as well as the sets $\gP_n$ and the maps $\gU_n: \gP_n \to \D{C}$ defined 
in \refS{SS:coordinates-tower}. 

For $\gU_n$, $u_n(i[-1, +\infty))$ plays the role of $i[-1, +\infty)$ for $Y_n$. 
Here we define two other curves for $\gU_n$ which play the analogous role of the lines $1/\ga_n+i[-1, +\infty)$ 
and $1/\ga_n-1 + i[-1, +\infty)$ for $Y_n$ (see \eqref{E:Y_n-comm-1} and \eqref{E:Y_n-comm-2}). 
Due to the presence of a critical point of $\gU_n$, as opposed to the 
injectivity of $Y_n$, we need to consider a pair of curves for the role of $1/\ga_n-1 + i[-1, +\infty)$. 
See \refF{F:w_n-v_n}. 
For the sake of simplifying the arguments, as in $u_n$, we parametrise these curves on the corresponding 
lines $1/\ga_n+i[-1, +\infty)$ and $1/\ga_n-1+i[-1, +\infty)$. 

\begin{propo}\label{P:w_n-marking}
For every $n\geq 0$, there are continuous and injective maps 
\[w_n^+: 1/\ga_n+ i  [-1, +\infty) \to \gP_n, \qquad w_n^-: 1/\ga_n+ i  [-1, +\infty) \to \gP_n\] 
such that 
\begin{itemize}
\item[(i)] for all $t \in [-1, +\infty)$ we have 
\[\gU_n \circ w_n^+(1/\ga_n+  i  t) = \gU_n \circ w_n^-(1/\ga_n+  i  t) = \gU_n \circ u_n ( i  t)-\gep_n,\]
\item[(ii)] on $1/\ga_n+ i  [0, +\infty)$, $w_n^+=w_n^-$, 
\item[(iii)]$w_n^+(1/\ga_n+i[-1, 0)) \cap w_n^-(1/\ga_n+i[-1, 0))=\emptyset$, 
\item[(iv)] $\gU_n$ has a critical point at $w_n^+(1/\ga_n)=w_n^-(1/\ga_n)$.  
\end{itemize}
\end{propo}

\begin{proof}
From \refS{SS:renormalisation-def}, recall the sets $A_n= A_{f_n}$, $B_n= B_{f_n}$, and $S_n= S_{f_n}$, 
as well as the integer $k_n=k_{f_n}$, associated to $h=f_n$. 
Also, recall from \refS{S:parametrised curve} that $u_n(i[-1,+\infty))$ is contained in 
$\gP=\gF_n(A_n\cup B_n) \subset \gP_n$. 
Then, $\gF_n^{-1}\circ u_n (i[-1, +\infty))$ is contained in $A_n \cup B_n$.  

Recall that $\gF_n^{-1}: \gP_n \to \D{C}$ covers $A_n \cup B_n$ several times, while 
$\gF_n^{-1}: \gP \to A_n \cup B_n$ is univalent. 
The restriction $\gF_n^{-1}:\gF_n(S_n) +k_n \to A_n \cup B_n$ has a specific covering structure; it covers $B_n$ 
in a one-to-one fashion, and covers $A_n$ by a two-to-one proper branched covering map. 
The branch point is mapped to $-4/27=\gF_n^{-1}(+1)$; the critical value of $f_n$. 
See the discussion in \refS{SS:renormalisation-def}. 
In particular, there is a unique continuous curve $w_n^+: 1/\ga_n+ i [0,+\infty) \to \gF_n (S_n)+ k_n$, such that 
\begin{equation}\label{E:P:w_n-marking-1}
\gF_n^{-1} (w_n^+(1/\ga_n+ i  t)) = \gF_n^{-1}(u_n( i  t)). 
\end{equation}
This defines $w_n^+$ on $1/\ga_n+ i [0,+\infty)$. We let $w_n^-=w_n^+$ on $1/\ga_n+ i [0,+\infty)$. 

There are two ways to extend $w_n^+$ on $1/\ga_n + i[-1, 0)$ so that the above equation holds. 
These come from the double covering structure of $\gF_n^{-1}$ from $\gF_n(A_n^{-k_n})+k_n$ onto $A_n$. 
Let us denote these maps by $w_n^+$ and $w_n^-$. 
The three curves $w_n^+(1/\ga_n+i[-1, 0])$, $w_n^-(1/\ga_n+i[-1, 0])$, and $w_n^+(1/\ga_n + i [0, +\infty))$ land at 
$w_n^+(1/\ga_n)$. There is a cyclic order on these curves consistent with the positive orientation on an infinitesimal 
circle at $w_n^+(1/\ga_n)$. 
We relabel these curves so that 
$w_n^+(1/\ga_n + i [0, +\infty)) < w_n^+(1/\ga_n + i [-1, 0]) < w_n^-(1/\ga_n + i [-1, 0])$. 
See \refF{F:w_n-v_n}. 
Evidently, $\gU_n$ has a critical point at $w_n^+(1/\ga_n)$, 
and $w_n^+(1/\ga_n+ i [-1,0)) \cap w_n^-(1/\ga_n+ i [-1, 0))= \emptyset$.  

By the above argument, the images of the curves $w_n^{\pm}$ are contained in $\gF_n(S_n) +k_n$, that is, 
for all $s\in \{+,-\}$, 
\begin{equation}\label{E:P:w_n-marking-2}
w_n^s: 1/\ga_n + i[-1, +\infty) \to \gF_n(S_n)+k_n.
\end{equation}
Recall that $\gU_n=\ex^{-1} \circ s\circ \gF_n^{-1}$ or $\gU_n =\ex^{-1} \circ \gF_n^{-1}$, depending on the sign 
$\gep_n$. 
Therefore, by \refE{E:P:w_n-marking-1} and the continuity of $u_n$, $w_n^+$, and $\gU_n$, there is an integer 
$i_n$ such that for $s \in \{+,-\}$ and all $t \geq -1$, $\gU_n \circ w_n^s(1/\ga_n+  i  t) = \gU_n \circ u_n ( i  t)+i_n$. 
On the other hand, the region bounded by the curves $u_n$ and $w_n^+$ near $+i\infty$ is mapped by 
$\gF_n^{-1}$ to a slit neighbourhood of $0$ in a one-to-one fashion. It follows that $i_n=-\gep_n$. 
\end{proof}

\begin{propo}\label{P:near-translation-w_n}
There exists a constant $C_6$ such that for every $n\geq 0$, and every $t\geq -1$, 
\[|w_n^+(it+1/\ga_n)- (u_n( i  t)+1/\ga_n)| \leq C_6.\] 
Moreover, for every $\gep>0$ there is $C_\gep$ such that for every $n\geq 0$ and every $t\geq C_\gep$, we have 
\[|w_n^+(1/\ga_n+ i  t)-(u_n( i  t)+ 1/\ga_n)| \leq \gep.\] 
\end{propo}

\begin{proof}
Consider $\mathcal{E}_n=\gF_n  \circ \gF_n^{-1} : \gF_n(S_n)+k_n \to \gP$, (compare with \refE{E:renorm-def}). 
By \refP{P:Ino-Shi1}-(e), $\mathcal{E}_n(w+1)= \mathcal{E}_n(w)+1$ whenever both sides are defined. 
Therefore, $\mathcal{E}_n$ induces a holomorphic map, say $\tilde{\mathcal{E}}_n$, from 
$\gF_n(S_n)/\D{Z} \subset \D{C}/\D{Z}$ onto $\{ w\in \D{C}/\D{Z} \mid \Im w \geq -2\}$. 
Moreover, by \refE{E:P:w_n-marking-1}, $\tilde{\mathcal{E}}_n$ maps the curve $w_n^+/\D{Z}$ to the 
curve $u_n/\D{Z}$. 

The image of $\tilde{\mathcal{E}}_n$ covers the region above the circle $\Im w =+2$ in a univalent fashion. 
Also, we have $\lim_{\Im w \to +\infty} \Im \tilde{\mathcal{E}}_n(w)=+\infty$. 
We may apply the Koebe distortion theorem to $\tilde{\mathcal{E}}_n^{-1}$ on the annulus 
$\{w \in \D{C}/\D{Z} \mid \Im w > 2\}$. 
It implies that there is a constant $c_n$ such that $|u_n(it)+c_n - w_n^+(1/\ga_n+it)|$ converges to $0$, uniformly 
independent of $n$, as $t \to +\infty$. 
Moreover, $|u_n(it)+c_n -w_n^+(1/\ga_n+it)|$ is uniformly bounded from above when $\Im u_n(it)\geq 3$. 
By \refE{E:gU'-asymptote}, $c_n=1/\ga_n$. 

By \refP{P:near-identity-u_n}, if $t\geq C_2/(1-\gd_2)+3$, $\Im u_n(it)\geq 3$. 
On the other hand, for $-1 \leq t \leq C_2/(1-\gd_2)+3$, by the compactness of $\QIS$ and the continuous 
dependence of $\gF_h$ on $h\in \QIS$, $|u_n(it)+1/\ga_n - w_n^+(1/\ga_n+it)|$ is uniformly bounded from above.
\end{proof}

For $n\geq 0$, we define the maps 
\[v_n^+: 1/\ga_n -1 + i[-1, +\infty) \to \Pi_n, \quad v_n^-: 1/\ga_n -1 + i[-1, +\infty) \to \Pi_n\] 
according to 
\begin{equation}\label{E:v_n-def}
v_n^+(1/\ga_n-1+it)= w_n^+(1/\ga_n+it) -1, \quad v_n^-(1/\ga_n-1+it)= w_n^-(1/\ga_n+it) -1.
\end{equation}

See \refF{F:lifting-partial-uniformisations} for an illustration of the following proposition. 

\begin{propo}\label{P:v_n-marking}
For every $n\geq 0$ and every $t \geq -1$ the following hold. 
\begin{itemize}
\item[(i)] If $\gep_{n+1}=-1$, then
\[\gU_{n+1} \circ v_{n+1}^+(1/\ga_{n+1}-1+it) + a_n+\gep_{n+1}  = w_n^+ (Y_{n+1}(it) + 1/\ga_n),\]
\[\gU_{n+1} \circ v_{n+1}^-(1/\ga_{n+1}-1+it) + a_n+\gep_{n+1}  = w_n^-(Y_{n+1}(it) + 1/\ga_n).\]
\item[(ii)] If $\gep_{n+1}=+1$, then 
\[\gU_{n+1} \circ v_{n+1}^+(1/\ga_{n+1}-1+it) +a_n+\gep_{n+1}  = w_n^- (Y_{n+1}(it) + 1/\ga_n),\]
\[\gU_{n+1} \circ v_{n+1}^-(1/\ga_{n+1}-1+it) +a_n+\gep_{n+1}  = w_n^+ (Y_{n+1}(it) + 1/\ga_n).\]
\end{itemize}
\end{propo}

\begin{proof}
Fix an arbitrary $n\geq 0$ and $s\in \{+, -\}$. Let us first assume that $\gep_{n+1}=-1$ so that $f_{n+1}=\rr(f_{n})$.  

By \refE{E:P:w_n-marking-2}, for all $t'\geq -1$, $w_{n}^s(1/\ga_{n}+ it')-k_{n} \in \gF_{n}(S_{n})$. 
By \refP{P:Ino-Shi1}-(e), and \refE{E:P:w_n-marking-1}, 
\[f_{n}^{k_{n}} \left( \gF_{n}^{-1} \left (w_{n}^s(1/\ga_{n}+it')-k_{n} \right) \right)
=\gF_{n}^{-1}(w_{n}^s(1/\ga_{n}+it'))=\gF_{n}^{-1}(u_{n}(it')).\]
Hence, by the definition of renormalisation $\rr(f_{n})=f_{n+1}$, see \refE{E:renorm-def}, the above relation 
implies that 
\begin{equation}\label{E:P:v_n-marking-1}
\begin{aligned}
f_{n+1} \left(\ex \left (w_{n}^s(1/\ga_{n}+it')\right) \right)
& =f_{n+1} \left(\ex \left (w_{n}^s(1/\ga_{n}+it')-k_{n} \right) \right) \\
&=\ex (u_{n}(it')).
\end{aligned}
\end{equation}
Let $it'= Y_{n+1}(it)$. The right hand side of the above equation becomes
\begin{align*}
\ex (u_{n}(it'))  =\ex (u_{n}(Y_{n+1}(it))) 
&=\ex \circ \gU_{n+1}  \circ u_{n+1}(it)  && (\text{\refP{P:comm-u-Y-upsilon}}) \\
&=\gF_{n+1}^{-1} (u_{n+1}(it))  &&  (\text {\refE{E:gU-defn}})  \\
&=\gF_{n+1}^{-1}(w_{n+1}^s (1/\ga_{n+1}+it)) && (\text{\refE{E:P:w_n-marking-1}}) \\ 
&= f_{n+1} \circ \gF_{n+1}^{-1} (w_{n+1}^s (1/\ga_{n+1}+it) -1) && (\text{\refP{P:Ino-Shi1}-e})\\
&=f_{n+1} \circ \gF_{n+1}^{-1} (v_{n+1}^s (1/\ga_{n+1}-1+it)). && (\text{\refE{E:v_n-def}}) 
\end{align*}
Combining the above equations, we conclude that 
\[f_{n+1} \left(\ex \left (w_{n}^s(1/\ga_{n}+Y_{n+1}(it)\right) \right)
=f_{n+1} \left( \gF_{n+1}^{-1}(v_{n+1}^s (1/\ga_{n+1}-1+it))\right).\] 
The above equation implies that 
\begin{align*}
\ex \left(w_{n}^s (1/\ga_{n}+Y_{n+1}(it))\right) & =\gF_{n+1}^{-1} (v_{n+1}^s (1/\ga_{n+1}-1+it)) && \\
&= \ex \circ \gU_{n+1} (v_{n+1}^s (1/\ga_{n+1}-1+it)) && (\text{\refE{E:gU-defn}})
\end{align*} 
Then, for any $t\geq -1$, there must be an integer $l_t$, such that 
\[w_{n}^s (1/\ga_{n}+Y_{n+1}(it))=\gU_{n+1} (v_{n+1}^s (1/\ga_{n+1}-1+it)) + l_t.\]
However, since $\gU_{n+1} (v_{n+1}^s (1/\ga_{n+1}-1+it))$ and $w_{n}^s (1/\ga_{n}+Y_{n+1}(it))$ 
depend continuously on $t$, $l_t$ must be independent of $t$. In order to identify the value of $l_t$ we look 
at the limiting behaviour of the relation as $t \to +\infty$. 

By \refP{P:near-translation-w_n}, 
\begin{align*}
\lim_{t\to +\infty} &  \left (\Re \left(v_{n+1}^s(1/\ga_{n+1}-1+it) - u_{n+1}(it)\right ) \right) \\
& =\lim_{t\to +\infty} \left (\Re (w_{n+1}^s (1/\ga_{n+1}+it) - u_{n+1}(it))\right) -1 
= 1/\ga_{n+1}-1.
\end{align*}
Applying $\gU_{n+1}$ and using \refE{E:gU'-asymptote}, 
\[\lim_{t\to +\infty} \left(\Re \left (\gU_{n+1}(v_{n+1}^s (1/\ga_{n+1}-1+it)) - \gU_{n+1}(u_{n+1}(it))\right) \right) 
= \ga_{n+1} (1/\ga_{n+1}-1) = 1-\ga_{n+1}.\]
On the other hand, by Propositions \ref{P:comm-u-Y-upsilon} and \ref{P:near-translation-w_n}, we have 
\begin{align*}
\lim_{t \to +\infty} \Re & (w_{n}^s (1/\ga_n+Y_{n+1}(it)) - \gU_{n+1} (u_{n+1}(it)) \\
&= \lim_{t \to +\infty} \Re (w_{n}^s (1/\ga_n+Y_{n+1}(it))- u_{n}(Y_{n+1}(it)))
=1/\ga_{n}.
\end{align*}
Hence, we must have $1-\ga_{n+1} +l_t= 1/\ga_{n}$, which by \refE{E:rotations-relations} and $\gep_{n+1}=-1$, 
implies that $l_t=a_{n} - 1= a_{n} +\gep_{n+1}$. 
This completes the proof of the desired relation in the proposition when $\gep_{n+1}=-1$. 

Now assume that $\gep_{n+1}=+1$. 
The argument is similar in this case, so we only look at the relation for $v_n^+$, and emphasis the differences 
with the above argument. 
As in the above argument, we have 
\[f_{n}^{k_{n}} \left( \gF_{n}^{-1} \left (w_{n}^-(1/\ga_{n}+it')-k_{n} \right) \right)= \gF_{n}^{-1}(u_{n}(it')).\]
When $\gep_{n+1}=+1$, $\rr(f_{n})=s \circ f_{n+1} \circ s$. Therefore, 
\begin{equation*}
\begin{aligned}
s\circ f_{n+1}  \circ s \left(\ex \left (w_{n}^-(1/\ga_{n}+it')\right) \right)
&= \ex (u_{n}(it')).
\end{aligned}
\end{equation*}
Let $it'= Y_{n+1}(it)$. The right hand side of the above equation becomes
\begin{align*}
\ex (u_{n}(it'))  =\ex (u_{n}(Y_{n+1}(it)))
&=\ex \circ \gU_{n+1}  \circ u_{n+1}(it)  && (\text{\refP{P:comm-u-Y-upsilon}}) \\
&=s \circ \gF_{n+1}^{-1} (u_{n+1}(it))  &&  (\text {\refE{E:gU-defn}})  \\
&=s \circ  f_{n+1} \circ \gF_{n+1}^{-1} (v_{n+1}^- (1/\ga_{n+1}-1+it)). && (\text{\refE{E:v_n-def}}) 
\end{align*}

Combining the above equations, we conclude that 
\[s \circ f_{n+1} \circ s \left(\ex \left (w_{n}^-(1/\ga_{n}+Y_{n+1}(it)\right) \right)
=s \circ f_{n+1} \left( \gF_{n+1}^{-1}(v_{n+1}^- (1/\ga_{n+1}-1+it))\right).\] 
The above equation implies that 
\begin{align*}
s \circ \ex \left(w_{n}^- (1/\ga_{n}+Y_{n+1}(it))\right) & =\gF_{n+1}^{-1} (v_{n+1}^+ (1/\ga_{n+1}-1+it)) && \\
&=s \circ  \ex \circ \gU_{n+1} (v_{n+1}^+ (1/\ga_{n+1}-1+it)) && (\text{\refE{E:gU-defn}}).
\end{align*} 
Note that the change from $v_{n+1}^-$ to $v_{n+1}^+$ in the above relation is due to the orientation reversing 
effect of $s\circ \ex$ on the left-hand side as opposed to the orientation preserving effect of 
$s \circ \ex \circ \gU_{n+1}$ on the right hand side of the equation. 

As in the previous case, there is an integer $l_t$, independent of $t$, such that  
\[w_{n}^+ (1/\ga_{n}+Y_{n+1}(it))=\gU_{n+1} (v_{n+1}^- (1/\ga_{n+1}-1+it)) + l_t.\]
Since $\gU_{n+1}'$ is asymptotically equal to $-\ga_{n+1}$ near $+i \infty$, in this case we obtain 
\[\lim_{t\to +\infty} \left(\Re \left (\gU_{n+1}(v_{n+1}^- (1/\ga_{n+1}-1+it)) - \gU_{n+1}(u_{n+1}(it))\right) \right) 
= -\ga_{n+1} (1/\ga_{n+1}-1) = \ga_{n+1}-1.\]
Hence, $\ga_{n+1}-1 + l_t= 1/\ga_{n}$, which by \refE{E:rotations-relations} and $\gep_{n+1}=+1$, implies that 
$l_t= a_{n} +\gep_{n+1}$. 
\end{proof}

Recall the numbers $t_n^1= \Im Y_{n+1}(-i) \in (-1, 0)$, for $n\geq -1$. 

\begin{propo}\label{P:disjoint-curves-u_n-v_n-w_n}
For all $n\geq 0$ and $s\in \{+,-\}$, we have 
\begin{equation*}
\begin{gathered}
w_n^+(1/\ga_n+ i  [t_n^1, +\infty)) \cap (u_n( i [t_n^1, +\infty))+ \D{Z})=\emptyset, \\
v_n^s(1/\ga_n-1 +  i  [t_n^1, +\infty)) \cap (u_n( i  [t_n^1, +\infty)) + \D{Z}) =\emptyset, \\ 
v_n^s (1/\ga_n-1 +  i  [-1, +\infty)) \cap w_n^+(1/\ga_n+  i  [-1, +\infty)) =\emptyset.
\end{gathered}
\end{equation*}
\end{propo}

\begin{proof}
First assume that $\gep_{n+1}=-1$ so that $\rr(f_n)=f_{n+1}$. 
Recall from \refS{SS:auxiliary-part-u_n} that $u_{n+1}(i(-1,+\infty))$ lies in $\interior (\gP)$. 
Moreover, by \refP{P:comm-u-Y-upsilon} and \refE{E:commutator-auxiliary-part-u_n-complete}, we have 
\[\ex \circ u_n (i(t_n^1, +\infty))= \gF_{n+1}^{-1} (u_{n+1}(i(-1, +\infty)).\] 
Thus, $\ex \circ u_n (i(t_n^1, +\infty))$ is contained in $\interior(A_{n+1} \cup B_{n+1}) = \interior(\gF_{n+1}^{-1}(\gP))$. 
It follows from \refE{E:P:v_n-marking-1} that 
$\ex \circ w_n^{\pm}(1/\ga_n+i(t_n^1, +\infty)) \subset \interior(A_{n+1}^{-1} \cup B_{n+1} ^{-1})$. 
As $\interior(A_{n+1} \cup B_{n+1}) \cap \interior(A_{n+1}^{-1} \cup B_{n+1}^{-1}) = \emptyset$,
\[\ex \circ u_n (i(t_n^1, +\infty)) \cap  \ex \circ w_n^{\pm}(1/\ga_n+i(t_n^1,+\infty)) =\emptyset.\] 
At the end of the interval $t_n^1$, by \refL{L:tail-unit-u_n-bottom}-(i) and \refE{E:P:v_n-marking-1}, 
\[\ex \circ u_n (it_n^1)=\gF_{n+1}^{-1} (u_{n+1}(-i))=\gF_{n+1}^{-1} (1-2i) \neq \gF_{n+1}^{-1} (-2i) \ni
\ex \circ w_n^{\pm} (1/\ga_n+ it_n^1).\]
Combining the above equations, we have 
\[\left( \ex \circ u_n (i[t_n^1, +\infty) \right) \cap \left( \ex \circ w_n^{\pm}(1/\ga_n+i[t_n^1, +\infty)) \right) 
=\emptyset.\] 
By the definition of $v_n^{\pm}$ in terms of $w_n^{\pm}$, the above equation implies the first two properties 
in the proposition.

Since $u_n(i(-1,+\infty))$ lies in $\interior(\gP)$, $w_n^{\pm}(1/\ga_n+ i (-1, +\infty))$ must lie in the interior of 
$\gF_n(S_n)+k_n$. The sets $\interior (\gF_n(S_n)+k_n)$ and $\interior (\gF_n(S_n)+k_n-1)$ do not meet. 
Hence, $v_n^{\pm}(1/\ga_n-1 + i(-1, +\infty))$ and $w_n^+(1/\ga_n+ i (-1, +\infty))$ are disjoint.
Evidently, $v_n^{\pm}(1/\ga_n-1-i) \neq w_n^+(1/\ga_n-i)$. 
These imply the last property in the proposition.

The proof for $\gep_{n+1} = +1$ is similar. 
\end{proof}

\subsection{The dynamical partition}\label{SS:nest-phase-space}
For $n\geq 0$, $u_n(i(-1, +\infty))$, $v_n^{\pm}(1/\ga_n-1+i(-1, +\infty))$ and $w_n^+(1/\ga_n+i(-1, +\infty))$ 
are pairwise disjoint, and lie in $\interior(\gP_n)$. 
Moreover, $u_n(-i)$, $v_n^{\pm}(1/\ga_n-1-i)$ and $w_n^+(1/\ga_n-i)$ are also pairwise disjoint, 
but all belong to $\partial \gP_n$. 
The union of $u_n$ and $w_n^+$ separate a connected region of $\gP_n$ which contains $v_n^{\pm}$. 
Also, $v_n^+(1/\ga_n-1+i[-1, 0]) \cup v_n^-(1/\ga_n-1+i[-1, 0])$ divides $\gP_n$ into two components.  
Let $\C{M}_n^0$ denote the closure of the connected component of 
\[\gP_n \setminus u_n(i[-1, \infty)) \bigcup w_n^+ (1/\ga_n+ i [-1, \infty)) 
\bigcup v_n^+(1/\ga_n-1+i[-1, 0]) \bigcup v_n^-(1/\ga_n-1+i[-1, 0])\] 
which contains $v_n^+(1/\ga_n -1 + i (0, \infty))$.  

\begin{figure}[ht]
\begin{center}
\begin{pspicture}(0,-.3)(7,4)
\pscustom[linecolor=cyan,fillstyle=solid,fillcolor=cyan]
{
\pscurve(1,4)(1,2)(.9,1.1)(1,1)(1.1,.9)(1,0) 
\psline(4,0)(4,4)
}

\pscustom[linecolor=cyan,fillstyle=solid,fillcolor=cyan]
{
\pscurve(5,4)(5,2)(4.9,1.2)(5,1)
\pscurve(5,1)(4.8,.85)(4.7,.3)
\pscurve(4.7,.3)(4.3,.35)(4,.2)
\psline(4,.2)(4,4)
}

\pscustom[linecolor=orange,fillstyle=solid,fillcolor=orange]
{
\pscurve(5,4)(5,2)(4.9,1.2)(5,1)
\pscurve(5,1)(5.2,1.05)(5.3,.3)
\pscurve[liftpen=1](5.7,.3)(5.8,.85)(6,1)
\pscurve(6,1)(5.9,1.2)(6,2)(6,4)
}

\psdot[dotsize=1pt](1,1)
\pscurve[linewidth=.7pt](1,4)(1,2)(.9,1.1)(1,1)(1.1,.9)(1,0) 

\pscurve[linewidth=.7pt](5,4)(5,2)(4.9,1.2)(5,1)
\pscurve[linewidth=.7pt](5,1)(4.8,.85)(4.7,.3)
\pscurve[linewidth=.7pt](5,1)(5.2,1.05)(5.3,.3)

\pscurve[linewidth=.7pt](6,4)(6,2)(5.9,1.2)(6,1)
\pscurve[linewidth=.7pt](6,1)(5.8,.85)(5.7,.3)
\pscurve[linewidth=.7pt](6,1)(6.2,1.05)(6.3,.3)

\rput(1,-.2){\small $u_n$}
\rput(4.7,-.2){\small $v_n^+$}
\rput(5.25,-.2){\small $v_n^-$}
\rput(5.75,-.2){\small $w_n^+$}
\rput(6.3,-.2){\small $w_n^-$}

\rput(3,3){\small $\mathcal{K}_n^0$}
\rput(5.5,3){\small $\mathcal{J}_n^0$}
\end{pspicture}
\caption{Illustration of the curves $u_n$, $v_n^\pm$ and $w_n^\pm$. 
The set $\mathcal{K}_n^0$ is coloured in blue and $\mathcal{J}_n^0$ is coloured in orange.}
\label{F:w_n-v_n}
\end{center}
\end{figure}

Similarly, $v_n^+(1/\ga_n-1+i[0, +\infty))$ divides $\C{M}_n^0$ into two components. 
Let $\C{J}_n^0$ denote the closure of the component of $\C{M}_n^0 \setminus v_n^+(1/\ga_n-1+i[0, +\infty))$ 
which contains $w_n^+$, and let $\C{K}_n^0$ denote the closure of the component of 
$\C{M}_n^0 \setminus v_n^+(1/\ga_n-1+i[0, +\infty))$ which contains $u_n$. 
Evidently, $\C{M}_n^0=\C{K}_n^0 \cup \C{J}_n^0$. 
These are analogous to $M_n^0$, $J_n^0$ and $K_n^0$ in \refS{SS:straight-model}. 
See \refF{F:w_n-v_n}.  

\begin{lem}\label{L:gU-injective}
For every $n\geq 0$, $\gU_n$ is injective on $\C{M}_n^0$. 
\end{lem}

\begin{proof}
Recall that $\gF_n^{-1}$ is injective on $\C{M}_n^0 \setminus (u_n \cup w_n^+)$, 
(see the proof of \refP{P:w_n-marking}). 
This implies that $\gU_n$ is injective on $\C{M}_n^0 \setminus (u_n \cup w_n^+)$. 
Also, by \refP{P:w_n-marking}, $\gU_n$ maps $u_n$ and $w_n^+$ to disjoint curves on the boundary of 
$\gU_n(\interior \C{M}_n^0)$. 
Therefore, $\gU_n$ is injective on $\C{M}_n^0$. 
\end{proof}

Now we define the sets $\C{M}_n^j$, $\C{J}_n^j$, and $\C{K}_n^j$, for $n\geq 0$ and $j\geq 0$, 
in analogy with $M_n^j$, $J_n^j$ and $K_n^j$ in \refS{SS:straight-model}. 
Assume that $\C{M}_n^j$, $\C{J}_n^j$, and $\C{K}_n^j$ are defined for some $j\geq 0$ and all $n \geq 0$. 
We define $\C{M}_n^{j+1}$, $\C{J}_n^{j+1}$, and $\C{K}_n^{j+1}$ for all $n\geq 0$ by following the 
below two cases. 

If $\gep_{n+1}=-1$, we define 
\begin{equation}\label{E:CM_n^j--1}
\C{M}_n^{j+1} 
=\textstyle{\bigcup_{l=0}^{a_n-2}} \big(\gU_{n+1}(\C{M}_{n+1}^j)+ l \big) 
\bigcup \big(\gU_{n+1}(\C{K}_{n+1}^j)+a_n-1\big).
\end{equation}

If $\gep_{n+1}=+1$, we define 
\begin{equation}\label{E:CM_n^j-+1}
\C{M}_n^{j+1} 
=\textstyle{\bigcup_{l=1}^{a_n}} \big(\gU_{n+1} (\C{M}_{n+1}^j)+ l \big) 
\bigcup \big(\gU_{n+1}(\C{J}_{n+1}^j)+ a_n+1\big ).
\end{equation}

Recall that when $\gep_{n+1}=-1$, by \refP{P:w_n-marking}, $\gU_{n+1}(\C{M}_{n+1}^j)$ lies between 
$u_n$ and $u_n+1$, and by \refP{P:v_n-marking}, $\gU_{n+1}(\C{K}_{n+1}^j)+a_n-1$ lies between 
$u_n+a_n-1$ and $w_n^+$. 
Also, when $\gep_{n+1}=+1$, by \refP{P:w_n-marking}, $\gU_{n+1}(\C{M}_{n+1}^j)$ lies between $u_n-1$ 
and $u_n$, and by \refP{P:v_n-marking}, $\gU_{n+1}(\C{J}_{n+1}^j)+a_n+1$ lies between $u_n+a_n$ 
and $w_n^+$. 
As in \refS{S:arithmetic-model}, it follows that $\C{M}_n^{j+1}$ is closed, bounded by piece-wise analytic curves, 
and $\interior (\C{M}_n^{j+1})$ is connected.
Moreover, $\C{M}_n^{j+1} \subset \C{M}_n^j$. 

Recall the numbers $t_n^j=\Im Y_{n+1}(i t_{n+1}^{j-1})$, with $t_n^0=-1$, for $n\geq -1$ and $j\geq 0$. 
One may see that when $\gep_{n+1}=-1$, 
\[u_n \left (i[t_n^j, +\infty)\right) \subset \partial \C{M}_n^j,\; 
\gU_{n+1}\left(v_{n+1}^+\left(1/\ga_{n+1}-1+i[t_{n+1}^{j-1},\infty) \right)\right) + a_n-1 \subset \partial \C{M}_n^j.\]
When $\gep_{n+1}=+1$, 
\[\gU_{n+1}\left(w_{n+1}^+\left (1/\ga_{n+1}+i[t_{n+1}^{j-1},+\infty) \right) \right)+1 \subset \partial \C{M}_n^j,\]
and 
\[\gU_{n+1}\left(v_{n+1}^-\left(1/\ga_{n+1}-1+i[t_{n+1}^{j-1}, \infty)\right)\right) 
+ a_n+1 \subset \partial \C{M}_n^j.\]
By \refP{P:v_n-marking}, 
\[\gU_{n+1} ( v_{n+1}^+(1/\ga_{n+1}-1 +i[0, +\infty))+ a_n +\gep_{n+1}-1= v_n^+(1/\ga_n-1+i[0, +\infty)).\] 
In particular, $v_n^+(1/\ga_n-1+i[0, +\infty))$ divides $\C{M}_n^{j+1}$ into two 
connected components. Let $\C{J}_n^{j+1}$ denote the closure of the connected component of 
$\C{M}_n^{j+1} \setminus v_n^+(1/\ga_n-1+i[0, +\infty))$ which meets $w_n^+$, 
and let $\C{K}_n^{j+1}$ denote the closure of the connected component of 
$\C{M}_n^{j+1} \setminus v_n^+(1/\ga_n-1+i[0, +\infty))$ which meets $u_n$. 
This completes the induction step to define $\C{M}_n^j$, $\C{K}_n^j$, and $\C{J}_n^k$. 

For $n=-1$, we may let 
\[\C{M}_{-1}^j = \gU_0(\C{M}_{0}^{j+1})+ (\gep_0+1)/2.\]    
For each $n \geq -1$, define 
\begin{equation}
\C{M}_n= \textstyle{\bigcap_{j=0}^{\infty}} \C{M}_n^j.
\end{equation} 

\subsection{Hyperbolic contraction of the changes of coordinates}\label{SS:uniform-contraction}
In this section we establish a uniform contraction of the maps $\gU_n$ on $\C{M}_n^0$, 
with respect to suitable hyperbolic metrics on the domain and range.
This will be similar to \refL{L:well-inside-lifts} and \refP{P:uniform-contraction}. 

\begin{lem}\label{L:extension-gU_n}
There is $\gd_3>0$ such that for every $n\geq 0$, there are open sets $\tilde{\C{M}}_n^0$, $\tilde{\C{K}}_n^0$ and 
$\tilde{\C{J}}_n^0$ satisfying the following properties: 
\begin{itemize}
\item[(i)]$\C{M}_n^0 \subset \tilde{\C{M}}_n^0$, $\C{K}_n^0 \subset \tilde{\C{K}}_n^0$, 
$\C{J}_n^0 \subset \tilde{\C{J}}_n^0$, $\tilde{\C{J}}_n^0 \cup \tilde{\C{K}}_n^0 = \tilde{\C{M}}_n^0$,
\item[(ii)] for all integers $l$ with $(\gep_{n+1}+1)/2 \leq l \leq a_n + \gep_{n+1}-1$, 
$B_{\gd_3}(\gU_{n+1}(\tilde{\C{M}}_{n+1}^0)+l) \subset \tilde{\C{M}}_n^0$, 
\item[(iii)] if $\gep_{n+1}=-1$, $B_{\gd_3}(\gU_{n+1}(\tilde{\C{K}}_{n+1}^0)+a_n-1) \subset \tilde{\C{M}}_n^0$, 
\item[(iv)]if $\gep_{n+1}=+1$, $B_{\gd_3}(\gU_{n+1}(\tilde{\C{J}}_{n+1}^0)+a_n+1) \subset \tilde{\C{M}}_n^0$.
\end{itemize}
\end{lem}

\begin{proof}
Recall the Poincar\'e metric $\rho$ on $\interior(\gP)$. 
Fix an arbitrary constant $R>0$, and let $\C{U}_n$ denote the set of all points in $\interior(\gP)$ 
which lie at hyperbolic distance at most $R$ from $u_n$. 
(Because $u_n(-i)=1-2i \in \partial \gP$, and $\rho$ behaves like $1/(\Im w+2)$ near the bottom end of $\gP$, 
$\C{U}_n$ asymptotically resembles a cone at $u_n(-i)$.) 
The curve $u_n$ divides $\C{U}_n$ into two components; we denote the one on the right hand side of 
$u_n$ by $\C{U}_n^+$ and the one on the left hand side of $u_n$ by $\C{U}_n^-$. 

As in the proof of \refP{P:w_n-marking}, we may use $\gF_n \circ \gF_n^{-1}$ to lift $\C{U}_n^+$ and $\C{U}_n^-$
to define the sets $\C{W}_n^{\pm} \subset \gP_n$ such that 
\[\gF_{n}^{-1}(\C{U}_n^+)= \gF_n^{-1}(\C{W}_n^+), \quad 
\gF_{n}^{-1}(\C{U}_n^-)= \gF_n^{-1}(\C{W}_n^-),\]
with $w_n^-(1/\ga_n-i) \in \C{W}_n^+$, and $w_n^+(1/\ga_n-i) \in \C{W}_n^-$. 
Then, $\C{W}_n^+$ is on the right side of $w_n^-$ and $\C{W}_n^-$ is on the left side of $w_n^+$. 
See \refF{F:L:extension-gU_n} for an illustration of these sets.
Now, let  
\[\C{V}_n^+=\C{W}_n^+-1, \quad \C{V}_n^-=\C{W}_n^- -1.\]
Then, $\C{V}_n^+$ lies on the right hand side of $v_n^-$ with $v_n^-(1/\ga_n-1-i) \in \C{V}_n^+$, and 
$\C{V}_n^-$ lies on the left hand side of $v_n^+$ with $v_n^+(1/\ga_n-1-i) \in \C{V}_n^-$.

\begin{figure}[ht]
\begin{center}
\begin{pspicture}(0, -.5)(7,5)
\pscustom[]
{
\pscurve(.8,4)(.8,2)(.7,1.1)(.8,1)(.95,.85)(1,0) 
\psline(4,0)(4,.22)
}

\pscustom[]
{
\pscurve(4.8,4)(4.8,2)(4.75,1.2)(4.8,1)(4.6,.85)(4.7,.33) 
\pscurve(4.7,.33)(4.6,.35)(4,.2)
\psline(4,.2)(4,.22)
}

\pscustom[]
{
\pscurve(4.8,4)(4.8,2)(4.75,1.2)(4.8,1)(4.6,.85)(4.7,.3) 
\pscurve(4.7,.3)(5.2,.37)(5.3,.3)(5.7,.3)(6.3,.3)
\pscurve[liftpen=1](6.3,.3)(6.4,1.05)(6.1,1.2)(6.2,2)(6.2,4)  
}

\pscurve[linestyle=dashed,linewidth=.4pt](1,4)(1,2)(.9,1.1)(1,1)(1.1,.9)(1,0) 
\rput(1,4.3){\small $u_n$}

\pscurve[linewidth=.4pt](.8,4)(.8,2)(.7,1.1)(.8,1)(.95,.85)(1,0) 
\psline[linewidth=.4pt]{->}(.6,4.6)(.9,3.9)
\rput(.6,4.8){\small $\C{U}_n^-$}

\pscurve[linewidth=.4pt](1.2,4)(1.2,2)(1.1,1.1)(1.2,1)(1.25,.85)(1,0) 
\psline[linewidth=.4pt]{->}(1.4,4.6)(1.1,3.9)
\rput(1.4,4.8){\small $\C{U}_n^+$}

\pscurve[linewidth=.4pt,linestyle=dashed](5,4)(5,2)(4.9,1.2)(5,1) 
\pscurve[linewidth=.4pt,linestyle=dashed](5,1)(4.8,.85)(4.7,.3) 
\pscurve[linewidth=.4pt,linestyle=dashed](5,1)(5.2,1.05)(5.3,.3) 
\rput(5,4.3){\small $v_n^\pm$}

\pscurve[linewidth=.4pt,linecolor](4.8,4)(4.8,2)(4.75,1.2)(4.8,1)(4.6,.85)(4.7,.31) 
\psline[linewidth=.4pt,linecolor]{->}(4.6,4.6)(4.9,3.9)
\rput(4.6,4.8){\small $\C{V}_n^-$}

\pscurve[linewidth=.4pt](5.2,4)(5.2,2)(5.1,1.2)(5.4,1.05)(5.3,.3) 
\psline[linewidth=.4pt]{->}(5.4,4.6)(5.1,3.9)
\rput(5.3,4.8){\small $\C{V}_n^+$}

\pscurve[linewidth=.4pt,linestyle=dashed](6,4)(6,2)(5.9,1.2)(6,1) 
\pscurve[linewidth=.4pt,linestyle=dashed](6,1)(5.8,.85)(5.7,.3) 
\pscurve[linewidth=.4pt,linestyle=dashed](6,1)(6.2,1.05)(6.3,.3)
\rput(6,4.3){\small $w_n^\pm$}

\pscurve[linewidth=.4pt](5.8,4)(5.8,2)(5.75,1.2)(5.8,1)(5.6,.85)(5.7,.3) 
\psline[linewidth=.4pt]{->}(5.6,4.6)(5.9,3.9)
\rput(5.8,4.8){\small $\C{W}_n^-$}

\pscurve[linewidth=.4pt](6.2,4)(6.2,2)(6.1,1.2)(6.4,1.05)(6.3,.3) 
\psline[linewidth=.4pt]{->}(6.4,4.6)(6.1,3.9)
\rput(6.4,4.8){\small $\C{W}_n^+$}

\rput(3,3){\small $\tilde{\mathcal{K}}_n^0$}
\rput(5.6,3){\small $\tilde{\mathcal{J}}_n^0$}

\psline[linewidth=.4pt]{->}(5,-.3)(5,.7)
\rput(4.8,-.3){\small $\C{V}_n^*$}

\psline[linewidth=.4pt]{->}(6,-.3)(6,.7)
\rput(5.8,-.3){\small $\C{W}_n^*$}
\end{pspicture}
\caption{Illustration of the sets $\C{U}_n^\pm$, $\C{V}_n^\pm$ and $\C{W}_n^\pm$ in the proof 
of \refL{L:extension-gU_n}. 
}
\label{F:L:extension-gU_n}
\end{center}
\end{figure}

Recall that $w_n^+(1/\ga_n + i[-1, 0]) \cup w_n^-(1/\ga_n+ i [-1,0])$ divides $\gP_n$ into two components, 
one of which is bounded, and denoted by $\C{W}_n^*$ here. Let $\C{V}_n^*= \C{W}_n^*-1$. 

We define 
\begin{equation}
\begin{gathered}
\tilde{\C{M}}_n^0 
= \interior \left( \C{M}_n^0 \cup \C{U}_n^- \cup \C{W}_n^+ \cup \C{W}_n^* \cup \C{V}_n^* \right), \\
\tilde{\C{K}}_n^0 = \interior \left (\C{K}_n^0 \cup \C{U}_n^- \cup \C{V}_n^+ \cup \C{V}_n^*\right), \\
\tilde{\C{J}}_n^0 = \interior\left(\C{J}_n^0 \cup \C{V}_n^- \cup \C{W}_n^+ \cup \C{W}_n^* \cup \C{V}_n^*\right).
\end{gathered}
\end{equation}

Since $\gU_{n+1}(u_{n+1}) \subset u_n$, and $\gU_{n+1}$ is uniformly contracting with respect to $\gr$, 
$\gU_{n+1}(\C{U}_{n+1}) \subset \C{U}_n$. 
Indeed, it follows from \refL{L:well-inside-lifts} and \refP{P:uniform-contraction} that there is a uniform $\gd>0$ 
such that 
\begin{itemize}
\item $B_\gd(\gU_{n+1}(\C{U}_{n+1})) \subset \C{U}_n$, 
\item when $\gep_{n+1}=-1$, $B_\gd (\gU_{n+1}(\C{U}_{n+1}^-)) \subset \tilde{\C{M}}_n^0$, 
$B_\gd(\gU_{n+1}(\C{V}_{n+1}^+))+a_n-1 \subset \tilde{\C{M}}_n^0$, 
\item when $\gep_{n+1}=+1$, $B_\gd(\gU_{n+1}(\C{W}_{n+1}^-))+1 \subset \tilde{\C{M}}_n^0$ and 
$B_\gd (\gU_{n+1}(\C{V}_{n+1}^+))+a_n+1 \subset \tilde{\C{M}}_n^0$. 
\end{itemize}

The set $\cup_{i=0}^{k_n} f_n^{\circ i}(S_n)$ is compactly contained in both the domain of $f_n$ and the image 
of $f_n$. 
By \refP{P:k_n-bounded}, $k_n$ is uniformly bounded from above, independent of $n$.
By the compactness of $\QIS$, there is $\gd'>0$, independent of $n$, such that $\gd'$-neighbourhood of 
$\cup_{i=0}^{k_n} f_n^{\circ i}(S_n)$ is contained in the domain and also in the image of $f_n$. 
This implies that there is $\gd''>0$, independent of $n$, such that $\gd''$-neighbourhood of 
$\gU_{n}(\partial \tilde{\C{M}}_n^0 \setminus (\C{U}_n^- \cup \C{W}_n^+))$ is contained in $\C{M}_{n-1}^0$. 

We define $\gd_3= \min \{\gd, \gd''\}$. 
\end{proof}

Let $\varrho_n |dz|$ denote the hyperbolic metric of constant curvature $-1$ on $\tilde{\C{M}}_n^0$, for $n\geq 0$. 

\begin{propo}\label{P:uniform-contraction-gU_n}
There is a constant $\gd_5\in (0,1)$ such that for all $n\geq 0$ we have, 
\begin{itemize}
\item[(i)] for all integers $l$ with $(\gep_{n+1}+1)/2 \leq l \leq a_n+ \gep_{n+1}-1$, and all $z\in \C{M}_{n+1}^0$,  
\[(\gU_{n+1}+l)^* \varrho_n(z) \leq \gd_5  \varrho_{n+1}(z);\]
\item[(ii)] if $\gep_{n+1}=-1$, for all $z\in \C{K}_{n+1}^0$, 
$(\gU_{n+1}+a_n-1)^* \varrho_n(z) \leq \gd_5  \varrho_{n+1}(z)$;
\item[(iii)] if $\gep_{n+1}=+1$, for all $z\in \C{J}_{n+1}^0$, 
$(\gU_{n+1}+a_n+1)^* \varrho_n(z) \leq \gd_5  \varrho_{n+1}(z)$.
\end{itemize}
\end{propo}

\begin{proof}
One may repeat the proof of \refP{P:uniform-contraction}; replacing \refL{L:well-inside-lifts} 
by \refL{L:extension-gU_n}.
We obtain the uniform contractions with respect to $\varrho_{n+1}$ on $\tilde{\C{M}}_{n+1}^0$ and 
$\varrho_n$ on $\tilde{\C{M}}_n^0$. 
In particular, the uniform contractions hold on $\C{M}_{n+1}^0$, $\C{K}_{n+1}^0$ and $\C{J}_{n+1}^0$. 
\end{proof}

\subsection{Iterates, shifts, and lifts}\label{SS:iterates-shifts-lifts}
In this section we relate the iterates of $f$ to integer translations and lifts in the renormalisation tower of $f$. 
To do that, we need to define the notion of trajectory of a given point, analogous to the one for the toy model 
in \refS{SS:dynamics-model}.

Given $z_{-1} \in \C{M}_{-1}$, there is a unique $z_0 \in \C{M}_0$ such that $z_{-1}=\gU_0(z_0)+({\gep_0}+1)/2$.
Then, inductively for $i\geq 0$ we identify $l_i \in \mathbb{Z}$ and $z_{i+1} \in \C{M}_{i+1}$ so that 
\begin{equation}\label{E:trace-condition-3}
z_i -l_i \in \gU_{i+1}(\C{M}_{i+1}) \qquad \text{and} \qquad \gU_{i+1}(z_{i+1})+l_{i}= z_i.
\end{equation}
It follows that for all $n\geq 0$, we have 
\begin{equation}\label{E:trace-condition-1}
z_{-1}=(\gU_0+(\gep_0+1)/2) \circ (\gU_1 + l_0) \circ \dots \circ (\gU_n+l_{n-1})(z_n),
\end{equation}
and by \eqref{E:CM_n^j--1} and \eqref{E:CM_n^j-+1}, for all $i \geq 0$, 
\begin{equation}\label{E:trace-condition-2}
(1+\gep_{i+1})/2 \leq l_i \leq a_i + \gep_{i+1}.
\end{equation}
We refer to the sequence $(z_i ; l_i)_{i\geq 0}$ as the \textbf{trajectory} of $z_{-1}$.  
Although the trajectory is not uniquely determined, (for some $z_i$ there might be two integers $l_i$ satisfying 
\refE{E:trace-condition-3}), we refer to any sequence $(z_i; l_i)_{i\geq 0}$ which satisfies both 
\eqref{E:trace-condition-1} and \eqref{E:trace-condition-2} as the trajectory of $z_{-1}$.

\begin{lem}\label{L:renormalisation-edges}
Let $p\geq 0$, and assume that for some $w_1 \in \C{J}_p^0$ and 
$w_2 \in \gU_{p+1}(\C{M}_{p+1}^0)+(\gep_p+1)/2$ we have $\rr(f_p)(\ex(w_1))=\ex(w_2)$. 
Then, $f_p(\gF_p^{-1}(w_1))= \gF_p^{-1}(w_2)$. 
\end{lem}

\begin{proof}
Recall $S_p= S_{f_p}$ from \refS{SS:renormalisation-def}, and that $\ex (\gF_p(S_p))=\Dom \rr (f_p)$. 
Since $\ex (w_1)$  belongs to $\Dom \rr(f_p)$, there is an integer $l_1$ such that  $w_1- l_1 \in \gF_p(S_p)$. 
By the definition of renormalisation, see \refS{SS:renormalisation-def}, 
\[\ex \circ\gF_p \circ f_p\co{k_p} \circ  \gF_p^{-1}(w_1- l_1) =\rr(f_p)(\ex(w_1-l_1)) = \rr(f_p)(\ex(w_1)).\]
Therefore, by the hypothesis in the lemma, we must have 
\[\ex \circ\gF_p \circ f_p\co{k_p} \circ  \gF_p^{-1}(w_1- l_1) = \ex (w_2).\]
This implies that there is $l_2 \in \D{Z}$ such that 
\begin{equation}\label{E:L:renormalisation-edges}
\gF_p \circ f_p\co{k_p} \circ  \gF_p^{-1}(w_1- l_1)+l_2 = w_2.
\end{equation}

Recall that $\C{J}_p^0$ is enclosed by the curves $w_p^+$, $v_p^-=w_p^- -1$ and $\partial (\gF_p(S_p)+\D{Z})$. 
By \refE{E:P:w_n-marking-2}, $w_p^+$  is contained in $\gF_p(S_p)+k_p$ and $v_p^-$ is contained 
in $\gF_p(S_p)+k_p-1$. 
Thus, $\C{J}_p^0 \subset (\gF_p(S_p)+k_p-1) \cup (\gF_p(S_p)+k_p)$.  
Because $w_1-l_1\in \gF_p(S_p)$ and $w_1 \in \C{J}_n^0$, either $l_1=k_p$ or $l_1 = k_p-1$. 

Recall that $u_p$ and $u_p+1$ lie on the boundary of $\gU_{p+1}(\C{M}_{p+1}^0)+(\gep_p+1)/2$, and 
$u_p$ is contained in $\gP=\gF_p(A_{f_p} \cup B_{f_p})$. 
This implies that $l_2\in \{0,1\}$. 
However, using \refE{E:P:w_n-marking-1}, one notes that when $l_1=k_p-1$ we must have $l_2=0$, and 
when $l_1=k_p$ we must have $l_2=1$. 
When $l_1=k_p-1$ and $l_2=0$, \refE{E:L:renormalisation-edges} implies the desired relation in the proposition, 
using the functional relation in \refP{P:Ino-Shi1}-(e). 
When $l_1=k_p$ and $l_2=1$, we apply $\gF_p^{-1}$ to both sides of \eqref{E:L:renormalisation-edges} and use 
\refP{P:Ino-Shi1}-(e), to conclude that 
\[\gF_p^{-1}(w_2)= \gF_p^{-1} \left (\gF_p \circ f_p\co{k_p} \circ  \gF_p^{-1}(w_1- k_p) +1\right )  
= f_p \circ \left(f_p\co{k_p} \circ  \gF_p^{-1}(w_1- k_p)\right )
=f_p(\gF_p^{-1}(w_1)).\qedhere\]
\end{proof}

Recall that for $n\geq 0$, $\C{M}_n \subset \C{M}_n^0= \C{K}_n^0 \cup \C{J}_n^0$. 
For $n\geq 0$, we define the map $\mathcal{E}_n: J_n^0 \to \C{M}_n^0$ as 
\begin{equation}\label{E:E_n}
\mathcal{E}_n = \gF_n \circ f_n \circ  \gF_n^{-1}.  
\end{equation}
Compare the above map to the one in \refE{E:renorm-def}. 

\begin{propo}\label{P:iterates-lifts-one-step}
Assume that $\ga \in \irr$, $f\in \QIS_\ga$, and $z_{-1} \in \C{M}_{-1}$ is an arbitrary point with 
trajectory $(z_i ; l_i)_{i\geq 0}$. 
The following hold: 
\begin{itemize}
\item[(i)] if there is $n\geq 0$ such that $z_n \in \C{K}_n$ and for all $0 \leq i \leq n-1$, $z_i \in \C{M}_i \setminus \C{K}_i$, 
then 
\[s\circ f \circ s(\ex(z_{-1})) 
= \ex \circ \left(\gU_0+\frac{\gep_0+1}{2}\right ) \circ \left(\gU_1+\frac{\gep_1+1}{2}\right ) 
\circ \cdots \circ \left (\gU_n +\frac{\gep_{n}+1}{2} \right) (z_n+1).\]
\item[(ii)] if for all $i\geq 0$, $z_i \in \C{M}_i \setminus \C{K}_i$, then for all $n\geq 0$,  
\[s\circ f \circ s(\ex(z_{-1})) 
= \ex \circ \left(\gU_0+\frac{\gep_0+1}{2}\right ) \circ \left(\gU_1+\frac{\gep_1+1}{2}\right ) 
\circ \cdots \circ \left (\gU_n +\frac{\gep_{n}+1}{2} \right) (\mathcal{E}_n(z_n)).\]
\end{itemize}
\end{propo}

\begin{proof}
Part (i). 
Since $z_n\in \C{K}_n^0$, $z_n+1 \in \C{M}_n^0$ and $\gU_n(z_n +1)+ (\gep_n+1)/2$ is defined and belongs to 
$\C{M}_{n-1}^0$. 
By \refP{P:Ino-Shi1}-(e), we have 
\begin{equation}\label{E:P:iterates-lifts-one-step-1}
f_n \circ \gF_{n}^{-1}(z_n)= \gF_n^{-1}(z_n+1).
\end{equation}
Now we consider two cases, based on whether $n=0$ or $n\geq 1$. 

First assume that $n=0$. If $\gep_0=-1$, then, by \refP{P:sequence-renormalisations}, $f_0= s \circ f \circ s$, and by 
\refE{E:gU-defn}, $\ex \circ \gU_0= \gF_0^{-1}$. 
Then, by the definition of trajectory, $\ex (z_{-1}) = \gF_0^{-1}(z_0)$. 
Using \refE{E:P:iterates-lifts-one-step-1} with $n=0$, and \refP{P:Ino-Shi1}-(e), we get 
\begin{align*}
s \circ f \circ s (\ex (z_{-1}))
=f_0(\gF_0^{-1}(z_0))
&= \gF_0^{-1}(z_0+1) \\
&=\ex \circ \gU_0 (z_0+1)
=\ex \circ (\gU_0+ (\gep_0+1)/2) (z_0+1).
\end{align*}
Similarly, if $\gep_0=+1$, $f_0=f$ and $\ex \circ \gU_0= s \circ \gF_0^{-1}$. 
Then, $\ex (z_{-1}) =s\circ \gF_0^{-1}(z_0)$. 
Therefore, 
\begin{align*}
s \circ f \circ s (\ex (z_{-1}))
= s \circ f (\gF_0^{-1}(z_0)) 
& = s\circ  f_0 (\gF_0^{-1}(z_0)) \\
&= s \circ  \gF_0^{-1}(z_0+1)
=\ex \circ (\gU_0+ (\gep_0+1)/2) (z_0+1).
\end{align*}

Now assume that $n\geq 1$. 
By considering two cases based on $\gep_n=\pm 1$, as in the previous case, one may see that 
\eqref{E:P:iterates-lifts-one-step-1} and \eqref{E:trace-condition-3} imply that 
\begin{equation}\label{E:P:iterates-lifts-one-step-2}
\rr(f_{n-1})(\ex(z_{n-1})) = \ex \circ  \left(\gU_n +\frac{\gep_{n}+1}{2} \right) (z_n+1).
\end{equation}
Then, applying \refL{L:renormalisation-edges} with $p=n-1$, $w_2=(\gU_n +(\gep_{n}+1)/2) (z_n+1)$ 
and $w_1=z_{n-1}$ we get  
\begin{equation}\label{E:P:iterates-lifts-one-step-3}
f_{n-1} \circ \gF_{n-1}^{-1}(z_{n-1})=\gF_{n-1}^{-1} \circ (\gU_n +(\gep_{n}+1)/2) (z_n+1).
\end{equation}
Compare the above relation to the one in \refE{E:P:iterates-lifts-one-step-1}. 

We repeat the above paragraph, replacing the relation in \refE{E:P:iterates-lifts-one-step-1} with the one in 
\refE{E:P:iterates-lifts-one-step-3}. 
If $n-1=0$, we get 
\begin{equation*}
s \circ f \circ s(\ex(z_{n-2})) 
= \ex \circ \left(\gU_{n-1} + \frac{\gep_{n-1}+1}{2} \right) \circ \left(\gU_n +\frac{\gep_{n}+1}{2} \right) (z_n+1).
\end{equation*}
which is the desired relation in Part (i). 
If $n-1 \geq 1$, \refE{E:P:iterates-lifts-one-step-3} implies that   
\begin{equation*}
\rr(f_{n-2})(\ex(z_{n-2})) 
= \ex \circ  \left(\gU_{n-1} + \frac{\gep_{n-1}+1}{2} \right) \circ \left(\gU_n +\frac{\gep_{n}+1}{2} \right) (z_n+1).
\end{equation*}
Repeating the above process, until we reach level $0$, leads to the desired relation in Part (i).  

Part(ii). By the definition of renormalisation in \refS{SS:renormalisation-def}, $\mathcal{E}_n$ induces the relation 
\[\rr(f_n) (\ex(z_n))= \ex (\mathcal{E}_n(z_n)).\]
By \refL{L:renormalisation-edges}, the above relation implies that 
\begin{equation}\label{E:P:iterates-lifts-one-step-4}
f_n \circ \gF_n^{-1}(z_n) = \gF_n^{-1}(\mathcal{E}_n(z_n)).
\end{equation}
Now one may repeat the argument in Part (i); replacing \refE{E:P:iterates-lifts-one-step-1} 
with \refE{E:P:iterates-lifts-one-step-4}. 
\end{proof}

\begin{propo}\label{P:iterates-lifts-general}
Assume that $\ga \in \irr$, $f\in \QIS_\ga$, and $z_{-1} \in \C{M}_{-1}$ with trajectory $(z_i ; l_i)_{i\geq 0}$. 
For every $\ell \geq 1$, there is a finite sequence of integers $(j_i)_{i=0}^v$ such that either 
\begin{multline}\label{E:P:iterates-lifts-general-10}
s \circ f\co{\ell} \circ s (\ex (z_{-1})) \\
= \ex \circ (\gU_0+(\gep_0+1)/2) \circ (\gU_1+j_0) \circ (\gU_2 + j_1) \circ \cdots \circ (\gU_v +j_{v-1}) (z_v+j_v),
\end{multline}
or 
\begin{multline}\label{E:P:iterates-lifts-general-11}
s \circ f\co{\ell} \circ s(\ex (z_{-1})) \\
= \ex \circ (\gU_0+(\gep_0+1)/2) \circ (\gU_1+j_0) \circ (\gU_2 +j_1) 
\circ \cdots \circ (\gU_v +j_{v-1}) (\mathcal{E}_v(z_v)). 
\end{multline}
\end{propo}

Each map $\gU_j$ on the right hand side of \eqref{E:P:iterates-lifts-general-10} and 
\eqref{E:P:iterates-lifts-general-11} is only considered on $\C{M}_j$. 

\begin{proof}
\refP{P:iterates-lifts-one-step} readily implies the statement for $\ell=1$. 
Assume that the statement holds for some $\ell -1 \geq 1$. We aim to prove it for $\ell$.  
By the induction hypothesis, there is a finite sequence of integers $(j_i)_{i=0}^n$ such that either 
\begin{multline}\label{E:P:iterates-lifts-general-1}
s\circ  f\co{(\ell-1)} \circ s(\ex (z_{-1})) \\
= \ex \circ (\gU_0+(\gep_0+1)/2) \circ (\gU_1+j_0) \circ (\gU_2 + j_1) \circ \cdots \circ (\gU_n +j_{n-1}) (z_n+j_n), 
\end{multline}
or 
\begin{multline}\label{E:P:iterates-lifts-general-2}
s\circ f\co{(\ell-1)} \circ s(\ex (z_{-1})) \\
= \ex \circ (\gU_0+(\gep_0+1)/2) \circ (\gU_1+j_0) \circ (\gU_2 + j_1) \circ \cdots \circ (\gU_n +j_{n-1}) (\mathcal{E}_n(z_n)).
\end{multline}

We consider two cases, based on which of \eqref{E:P:iterates-lifts-general-1} 
and \eqref{E:P:iterates-lifts-general-2} holds. 

\medskip 
{\em Case 1.} Assume that \refE{E:P:iterates-lifts-general-1} holds. 

\smallskip 

Define 
\[w_{-1}= (\gU_0+(\gep_0+1)/2) \circ (\gU_1+j_0) \circ (\gU_2 + j_1) \circ \cdots \circ (\gU_n +j_{n-1}) (z_n+j_n).\]
The point $w_{-1}$ has a trajectory $(w_i ; s_i)_{i\geq 0}$ which satisfies the relations 
\begin{equation}\label{E:P:iterates-lifts-general-3}
w_i= (\gU_{i+1}+j_{i}) \circ \cdots \circ (\gU_n +j_{n-1}) (z_n+j_n), \quad \tfor 0 \leq i \leq n-1,
\end{equation} 
\begin{equation}\label{E:P:iterates-lifts-general-5}
w_n = z_n+j_n, 
\end{equation}
\begin{equation}\label{E:P:iterates-lifts-general-6}
w_i=z_i, \quad \tfor i\geq n+1. 
\end{equation}
By \refE{E:P:iterates-lifts-general-1}, 
\begin{equation}\label{E:P:iterates-lifts-general-7}
s \circ  f\co{\ell} \circ s(\ex (z_{-1}))
= (s \circ f \circ s) \circ s \circ  f\co{(\ell-1)} \circ s(\ex (z_{-1}))
=s \circ f \circ s(\ex (w_{-1})).
\end{equation}
Now, we consider two scenarios. 

{\em Case 1.1:} There is $m \geq 0$ such that $w_m \in \C{K}_m^0$, and for all $0 \leq i \leq m-1$, 
$w_i \in \C{M}_i \setminus \C{K}_i$. 

We may employ \refP{P:iterates-lifts-one-step}, to obtain 
\begin{multline}\label{E:P:iterates-lifts-general-9}
s \circ f \circ s(\ex (w_{-1})) \\
=\ex  \circ \left(\gU_0+\frac{\gep_0+1}{2}\right ) \circ \left(\gU_1+\frac{\gep_1+1}{2}\right ) 
\circ \cdots \circ \left (\gU_m + \frac{\gep_{m}+1}{2} \right) (w_m+1).
\end{multline} 
There are three scenarios based on the value of $m$ relative $n$. 

{\em Case 1.1.1:} $m \leq n-1$. 
By \refE{E:P:iterates-lifts-general-3}, the right hand side of \eqref{E:P:iterates-lifts-general-9} may be written as 
\begin{multline*}
\ex \circ \left(\gU_0+\frac{\gep_0+1}{2}\right ) \circ \left(\gU_1+\frac{\gep_1+1}{2}\right ) 
\circ \cdots \circ \left (\gU_m + \frac{\gep_{m}+1}{2} \right) \\
\circ  \left(\gU_{m+1}+ j_m+1\right ) \circ \cdots \circ \left (\gU_n + j_{n-1} \right) (z_n+j_n).
\end{multline*} 
{\em Case 1.1.2:} $m=n$. 
By \refE{E:P:iterates-lifts-general-5}, the right hand side of \eqref{E:P:iterates-lifts-general-9} may be written as 
\begin{equation*}
\ex \circ \left(\gU_0+\frac{\gep_0+1}{2}\right ) \circ \left(\gU_1+\frac{\gep_1+1}{2}\right ) 
\circ \cdots \circ \left (\gU_n + \frac{\gep_{n}+1}{2} \right) (z_n+j_n+1).
\end{equation*}
{\em Case 1.1.3:} $m \geq n+1$. 
By \refE{E:P:iterates-lifts-general-6}, the right hand side of \eqref{E:P:iterates-lifts-general-9} may be written as 
\begin{equation*}
\ex \circ \left(\gU_0+\frac{\gep_0+1}{2}\right ) \circ \left(\gU_1+\frac{\gep_1+1}{2}\right ) 
\circ \cdots \circ \left (\gU_m + \frac{\gep_{m}+1}{2} \right) (z_m+1).
\end{equation*}

{\em Case 1.2:} There is no $m\geq 0$ satisfying $w_m \in \C{K}_m$. 

By \refP{P:iterates-lifts-one-step}, and \refE{E:P:iterates-lifts-general-6}, we obtain 
\begin{multline*}
s \circ f \circ s(\ex (w_{-1})) \\
= \ex \circ \left(\gU_0+\frac{\gep_0+1}{2}\right ) \circ \left(\gU_1+\frac{\gep_1+1}{2}\right ) 
\circ \cdots \circ \left (\gU_{n+1} + \frac{\gep_{n+1}+1}{2} \right) (\mathcal{E}_{n+1}(w_{n+1})) \\
= \ex \circ \left(\gU_0+\frac{\gep_0+1}{2}\right ) \circ \left(\gU_1+\frac{\gep_1+1}{2}\right ) 
\circ \cdots \circ \left (\gU_{n+1} + \frac{\gep_{n+1}+1}{2} \right) (\mathcal{E}_{n+1}(z_{n+1})).
\end{multline*}

\medskip

{\em Case 2.} Assume that \refE{E:P:iterates-lifts-general-2} holds. 

\smallskip 

Define 
\[w_{-1}= 
(\gU_0+(\gep_0+1)/2) \circ (\gU_1+j_0) \circ (\gU_2 + j_1) \circ \cdots \circ (\gU_n +j_{n-1}) (\mathcal{E}_n(z_n)).\]
The point $w_{-1}$ has a trajectory $(w_i ; s_i)_{i\geq 0}$ which satisfies 
\begin{equation}\label{E:P:iterates-lifts-general-4}
w_i= (\gU_{i+1}+j_{i}) \circ \cdots \circ (\gU_n +j_{n-1}) (\mathcal{E}_n(z_n)), \quad \tfor 0 \leq i \leq n-1,
\end{equation}
\begin{equation}\label{E:P:iterates-lifts-general-8}
w_n = \mathcal{E}_n(z_n).
\end{equation}
By \refE{E:P:iterates-lifts-general-2}, 
\begin{equation}
s \circ f\co{\ell} \circ s(\ex (z_{-1}))
= (s\circ f \circ s) \circ s \circ f\co{(\ell-1)} \circ s(\ex (z_{-1}))
= s \circ f \circ s(\ex (w_{-1})).
\end{equation}
Now we consider two scenarios. 

{\em Case 2.1:} There is $m \leq n-1$, such that $w_m \in \C{K}_m^0$, and for all $0 \leq i \leq m-1$, 
$w_i \in \C{M}_i^0 \setminus \C{K}_i^0$. 

Using \refP{P:iterates-lifts-one-step}, and \refE{E:P:iterates-lifts-general-4}, 
\begin{multline*}
s \circ f \circ s(\ex (w_{-1}))\\
=\ex \circ \left(\gU_0+\frac{\gep_0+1}{2}\right ) \circ \left(\gU_1+\frac{\gep_1+1}{2}\right ) 
\circ \cdots \circ \left (\gU_m + \frac{\gep_{m}+1}{2} \right) (w_m+1)  \\
= \ex \circ \left(\gU_0+\frac{\gep_0+1}{2}\right ) \circ \left(\gU_1+\frac{\gep_1+1}{2}\right ) 
\circ \cdots \circ \left (\gU_m + \frac{\gep_{m}+1}{2} \right) \\
\circ  \left(\gU_{m+1}+ j_m+1\right ) \circ \cdots \circ \left (\gU_n + j_{n-1} \right) (\mathcal{E}_n(z_n)).
\end{multline*} 

{\em Case 2.2:} For all $m$ with $0 \leq m \leq n-1$, $w_m \in \C{M}_m^0 \setminus \C{K}_m^0$. 

By \refP{P:iterates-lifts-one-step}, and using $w_n = \C{E}_n(z_n) \in \C{K}_n^0$, we get 
\begin{multline}\label{E:P:iterates-lifts-general-12}
s \circ f \circ s(\ex (w_{-1})) \\
=\ex \circ \left(\gU_0+\frac{\gep_0+1}{2}\right ) \circ \left(\gU_1+\frac{\gep_1+1}{2}\right ) 
\circ \cdots \circ \left (\gU_n + \frac{\gep_{n}+1}{2} \right) (w_n+1) \\
=\ex \circ \left(\gU_0+\frac{\gep_0+1}{2}\right ) \circ \left(\gU_1+\frac{\gep_1+1}{2}\right ) 
\circ \cdots \circ \left (\gU_n + \frac{\gep_{n}+1}{2} \right) (\mathcal{E}_n(z_n)+1).
\end{multline} 
There are two scenarios based on whether $z_{n+1} \in \C{K}_{n+1}^0$, or not. 

{\em Case 2.2.1:} $z_{n+1} \in \C{K}_{n+1}$. Then $z_{n+1} +1 \in \C{M}_{n+1}$, and hence 
$f_{n+1} \circ \gF_{n+1}^{-1}(z_{n+1})= \gF_{n+1}^{-1} (z_{n+1}+1)$. 
The latter relation implies that 
\[\rr(f_n)(\ex (z_n))= \ex \circ (\gU_{n+1}+(\gep_{n+1}+1)/2)(z_{n+1}+1).\]
We may apply \refL{L:renormalisation-edges}, to get  
\[f_n  \circ \gF_n^{-1}(z_n)   = \gF_n^{-1} \circ (\gU_{n+1}+(\gep_{n+1}+1)/2)(z_{n+1}+1).\]
This implies that 
\[\mathcal{E}_n(z_n)=\gF_n \circ f_n  \circ \gF_n^{-1}(z_n) = (\gU_{n+1}+(\gep_{n+1}+1)/2)(z_{n+1}+1).\]
By \refE{E:P:iterates-lifts-general-8}, and the above relation, the right hand side 
of \eqref{E:P:iterates-lifts-general-12} 
may be written as 
\[\ex \circ \left(\gU_0+\frac{\gep_0+1}{2}\right ) \circ \left(\gU_1+\frac{\gep_1+1}{2}\right ) 
\circ \cdots \circ \left (\gU_n + \frac{\gep_{n}+1}{2} \right) 
\circ  \left(\gU_{n+1}+ \frac{\gep_{n+1}+1}{2}+1 \right ) (z_{n+1}+1).\]
{\em Case 2.2.2:} $z_{n+1} \in \C{M}_{n+1}\setminus \C{K}_{n+1}$. As in the previous case, one may see that 
\[\mathcal{E}_n(z_n)= (\gU_{n+1}+(\gep_{n+1}+1)/2)(\mathcal{E}_{n+1}(z_{n+1})).\] 
Therefore, by \refE{E:P:iterates-lifts-general-8}, and the above relation, the right hand side 
of \eqref{E:P:iterates-lifts-general-12} 
may be written as  
\begin{equation*}
\begin{aligned}
\ex \circ \left(\gU_0+\frac{\gep_0+1}{2}\right ) \circ \left(\gU_1+\frac{\gep_1+1}{2}\right ) 
\circ & \cdots \circ \left (\gU_n + \frac{\gep_{n}+1}{2} \right) \\
& \circ  \left(\gU_{n+1}+ \frac{\gep_{n+1}+1}{2}+1 \right ) (\mathcal{E}_{n+1}(z_{n+1})).
\end{aligned} \qedhere
\end{equation*}
\end{proof}

\begin{rem}
Assume that $z_{-1}$ in $\C{M}_{-1}$ has a trajectory $(z_i; l_i)_{i\geq 0}$ such that for infinitely 
many distinct $n\geq0$ we have $z_n \in \C{K}_n^0$. 
It is evident from the proof of \refP{P:iterates-lifts-general} that for every $\ell \geq 1$, 
\refE{E:P:iterates-lifts-general-10} holds for some $(j_i)_{i=0}^v$. 
For instance, if $\ga$ is an irrational number with $\gep_i=-1$ for all $i \geq 0$, one may 
employ \refP{P:uniform-contraction-gU_n} to conclude that for every $z_{-1}\in \C{M}_{-1}$, infinitely 
often $z_n \in \C{K}_n^0$. 
However, if $\ga$ is an irrational number with $\gep_i=+1$ for all $i\geq 0$, there are $z_{-1}\in \C{M}_{-1}$ 
such that for all $n\geq 0$, $z_n \in \C{M}_n \setminus \C{K}_n^0$. 
For such $z_{-1}$, it is not possible to have \refE{E:P:iterates-lifts-general-10} for all $\ell \geq 1$. 
Once we establish the relation between $\C{M}_{-1}$ and $M_{-1}$ in \refS{S:uniformisation}, 
it becomes clear that the set of such $z_{-1}$ forms a countable union of arcs in $\C{M}_{-1}$. 
\end{rem}

The inverse of the statement in \refP{P:iterates-lifts-general} is also true, which we state below. 

\begin{propo}\label{P:lifts-iterates-general}
Assume that $\ga \in \irr$, $f\in \QIS_\ga$, and $z_{-1} \in \C{M}_{-1}$ is an arbitrary point with 
trajectory $(z_i; l_i)_{i\geq 0}$. 
Then, the following hold: 
\begin{itemize}
\item[(i)] for every sequence of integers $(j_i)_{i=0}^v$ with $j_v \geq 1$, 
there is $\ell \geq 1$ such that \refE{E:P:iterates-lifts-general-10} holds, provided each $\gU_j$ in the right hand side 
of \refE{E:P:iterates-lifts-general-10} is considered on $\C{M}_j$. 
\item[(ii)] for every sequence of integers $(j_i)_{i=0}^{v-1}$, there is $\ell \geq 1$ such 
that \refE{E:P:iterates-lifts-general-11} holds, provided each $\gU_j$ in the right hand side 
of \refE{E:P:iterates-lifts-general-11} is considered on $\C{M}_j$. 
\end{itemize}
\end{propo}

We will not use the above proposition in this paper, it is only stated for the record.

\begin{proof}
By \refP{P:Ino-Shi1}-(e), each translation in $\C{M}_j$ corresponds to 
an iterate of $f_j$. Each iterate of $f_j$ corresponds to an iterate of $\rr(f_{j-1})$. 
Each iterate of $\rr(f_{j-1})$ corresponds to a finite number of iterates by $f_{j-1}$. Combining these steps, 
one concludes that the translations and lifts correspond to some iterate of $f$ under the changes of coordinates. 
The argument is similar to the proof of \refL{L:renormalisation-edges}, so we leave the details to the reader. 
One may consult the proofs of similar statements in \cite{Che13,Che19,AC18}. 
\end{proof}

The following lemma will not be formally used in this paper, but it sheds some light on the argument 
presented in \refS{S:uniformisation}. 

\begin{lem}\label{L:tails-disjoint-CM_n}
For every $n\geq 0$, and every integer $l \in [0, a_n+\gep_n]$, we have
\begin{equation*}
\begin{gathered}
(u_n(i[-1, 0))+ l)  \cap \C{M}_n^0= \emptyset, \; w_n^{\pm}(1/\ga_n+ i[-1, 0)) \cap \C{M}_n^0=\emptyset, \; 
v_n^{\pm}(1/\ga_n-1+ i[-1, 0)) \cap \C{M}_n^0=\emptyset.
\end{gathered}
\end{equation*}
\end{lem}

\begin{proof}
Fix arbitrary $n \geq 0$ and $0 \leq l \leq a_n+\gep_n$. 
Recall from \refP{P:w_n-marking} and \eqref{E:v_n-def} that for all $i\geq 0$, $w_i^+(1/\ga_i)= w_i^-(1/\ga_i)$ 
and $v_i^+(1/\ga_i-1)= v_i^-(1/\ga_i-1)$. 
By propositions \ref{P:w_n-marking} and \ref{P:v_n-marking}, for every $z_i$ in 
$\{u_i(0)+ l, w_i^+(1/\ga_i),v_i^+(1/\ga_i-1)\}$, there are $z_{i+1}$ in 
$\{u_{i+1}(0), w_{i+1}^+(1/\ga_i),v_{i+1}^+(1/\ga_i-1)\}$ and $j_i \in \mathbb{Z}$ such that 
$\gU_{i+1}(z_{i+1})+j_i=z_i$. 
This implies that $z_n$ has a trajectory $(z_i; j_i)_{i \geq n+1}$, such that for all $i \geq n+1$, 
$z_i$ belongs to $\{u_i(0), w_i^+(1/\ga_i), v_i^+(1/\ga_i-1)\}$. 

Now assume that there is $t \in [-1, 0)$ such that $u_n(it)+l$ belongs to $\C{M}_n$. 
Let $w_n =u_n(it)+l$ and $z_n=u_n(0)+l$. 
It follows from Propositions \ref{P:w_n-marking} and \ref{P:v_n-marking}, that $z_n$ and $w_n$ 
have trajectories $(z_i; j_i)_{i\geq n+1}$ and $(w_i; j_i)_{i\geq n+1}$, respectively. 
That is, the integers $j_i$ in the corresponding trajectories are identical. 
By the definition of trajectories, for all $m \geq n+2$, we have 
\[z_n=(\gU_{n+1}+(\gep_{n+1}+1)/2+ l) \circ (\gU_{n+2}+j_{n+1}) \circ \dots \circ (\gU_m+j_{m-1})(z_m), \]
\[w_n=(\gU_{n+1}+(\gep_{n+1}+1)/2+ l) \circ (\gU_{n+2}+j_{n+1}) \circ \dots \circ (\gU_m+j_{m-1})(w_m).\]
Recall the hyperbolic metric $\varrho_m$ on $\tilde{\C{M}}_m^0 \supset \C{M}_m^0$, 
discussed in \refS{SS:uniform-contraction}. 
Since $z_{m+1}$ and $w_{m+1}$ belong to $\C{M}_{m+1} \subset \C{M}_{m+1}^0 \subset \tilde{\C{M}}_{m+1}^0$, 
it follows from \refL{L:extension-gU_n}, that the distance between $w_m$ and $z_m$ with respect to $\varrho_m$ 
is uniformly bounded from above, independent of $m$. 
Then, by \refP{P:uniform-contraction-gU_n}, we must have $w_n=z_n$. 
That is, $u_n(0)=u_n(it)$. 
This contradicts the injectivity of $u_n$ on $[-1, +\infty)$, proved in \refP{P:injective-u_n}.  

The latter two properties in the lemma are proved similarly, where one employs the injectivity of 
$w_n^{\pm}$ and $v_n^{\pm}$ in Propositions \ref{P:w_n-marking} and \ref{P:v_n-marking}. 
\end{proof}

\subsection{Capturing the post-critical set}\label{SS:capturing-pc}
Let us define the sets 
\begin{equation}
\hat{\C{A}}_f= s \circ \ex(\C{M}_{-1}) \cup \{0\}, \qquad \text{and} \qquad \C{A}_f= \partial  \hat{\C{A}}_f.
\end{equation}
When $\ga \in \E{B}$, $s \circ \ex(\C{M}_{-1})$ contains a punctured neighbourhood of $0$, 
so we have added $0$ to that set, to avoid including $0$ in $\C{A}_f$. 
When $\ga \notin \E{B}$, as we shall see in the next section, $s \circ \ex(\C{M}_{-1}) \cup\{0\}$ has empty interior, 
so its boundary is itself. 
Recall from \refS{S:intro} that the post-critical set of $f$ is denoted by $\gL(c_f)$. 
In this section we prove that $\gL(c_f) \subseteq \hat{\C{A}}_f$. 

\begin{propo}\label{P:pc-covered}
For every $\ga \in \irr$ and $f\in \QIS_\ga$, $f(\hat{\C{A}}_f) \subseteq \hat{\C{A}}_f$, and 
$\gL(c_f) \subseteq \hat{\C{A}}_f$. 
\end{propo}

\begin{proof}
The first part of the proposition follows from \refP{P:iterates-lifts-general}. 
For the latter part of the proposition we note that $1=u_{-1}(0) \in \C{M}_{-1}$, and $s \circ \ex(1)$ is the critical 
value of $f$. Thus, by the first part of the proposition, the orbit of the critical value of $f$ is contained 
in $\hat{\C{A}}_f$. Because $\hat{\C{A}}_f$ is a closed set, it must contain $\gL(c_f)$. 
\end{proof}
 \section{Uniformisation of the post-critical set}\label{S:uniformisation}
In this section we complete the proofs of the theorems stated in the Introduction. 
There remains to show that $\C{A}_\ga$ is homeomorphic to the topological model $A_\ga$.

Recall $M_n^j$ and $Y_n$ employed to build the topological model in \refS{S:arithmetic-model}, and the 
correspond dynamical objects $\C{M}_n^j$ and $\gU_n$ from \refS{S:modified nest}. 
We aim to build homeomorphisms from each $\C{M}_n^j$ to $M_n^j$ which collectively enjoy 
equivariant properties with respect to $Y_n$ and $\gU_n$.
Then we pass to limit to obtain a homeomorphism from $\C{M}_n$ to $M_n$, which in turn will induce a topological 
conjugacy from $\C{A}_f$ to $A_\ga$, conjugating $f$ to the model map $T_\ga$. 

\subsection{Summary of the markings}\label{SS:markings}
In Sections~\ref{S:parametrised curve} and \ref{S:modified nest} we introduced the curves $u_n$, $v_n^{\pm}$, 
and $w_n^\pm$, for $n\geq 0$. 
We also established some remarkable equivariant properties of these curves, which play a fundamental role in 
this section. 
For the convenience of the reader, we collect (and reformulate some of) those relations, 
and present them below. 

For every $n\geq 0$, the following hold:
\begin{itemize}
\item[(E1)] $u_n:  i [-1,+\infty) \to \partial \C{M}_n^0$, and for all $t \geq -1$, 
\begin{equation*}\label{E:comm-u_n-Y_n-gU_n}
\gU_n \circ u_n( i t)= u_{n-1} \circ  Y_n ( i t).
\end{equation*}
\item[(E2)] $w_n^+:(1/\ga_n+ i [-1,+\infty)) \to \partial \C{M}_n^0$, 
$w_n^- : (1/\ga_n+ i [-1,+\infty)) \to \gP_n$, and for all $t \geq -1$, 
\begin{equation*}\label{E:comm-u_n-w_n-Y_n-gU_n}
\gU_n \circ  w_n^+(1/\ga_n+  i t) = \gU_n \circ  w_n^-(1/\ga_n+  i t) = u_{n-1} \circ Y_n( i t) -\gep_n;
\end{equation*}
\item[(E3)] $v_n^{\pm}:(1/\ga_n-1 +  i [-1, \infty)) \to \C{M}_n$, and for all $t \geq -1$,   
\begin{equation*}\label{E:comm-v_n-w_n-Y_n-gU_n}
\begin{gathered}
\gU_n \circ v_n^{\pm}(1/\ga_n-1+it)+a_{n-1}+\gep_n= w_{n-1}^{\pm} (1/\ga_{n-1}+Y_n(it)), \quad \tif \gep_n=-1, \\
\gU_n \circ v_n^{\pm}(1/\ga_n-1+it)+a_{n-1}+\gep_n= w_{n-1}^{\mp} (1/\ga_{n-1}+Y_n(it)), \quad \tif \gep_n=+1. 
\end{gathered}
\end{equation*}
\end{itemize}

See \refF{F:lifting-partial-uniformisations} for an illustration of the above relations. 

\subsection{Partial uniformisations matching the markings}\label{SS:partial-uniformisation}
Let $n\geq 0$ and $j\geq 0$. We say that a map $\gO: \C{M}_n^j  \to M_n^j$ \textbf{matches} 
$(u_n, v_n^{\pm}, w_n^+)$, if the following four properties hold:
\begin{itemize}
\addtolength{\itemsep}{.5em}
\item[(M1)] for every $z \in M_n^j$ with $\Re z=0$ we have $\gO \circ u_n(z) = z$; 
\item[(M2)] for every $z \in M_n^j$ with $\Re z=1/\ga_n$ we have $\gO \circ w_n^+(z) = z$; 
\item[(M3)] for every $z \in M_n^j$ with $\Re z=1/\ga_n-1$ we have $\gO \circ v_n^+(z) = z$;
\item[(M4)] for every $z \in M_n^j$ with $\Re z=1/\ga_n-1$ we have $\gO \circ v_n^-(z) = z$; 
\end{itemize}
In other words, $\gO: \C{M}_n^j  \to M_n^j$ matches $(u_n, v_n^{\pm}, w_n^+)$ if it is equal to the inverses 
of the maps $u_n$, $v_n^+$, $v_n^-$, and $w_n^+$, where they are defined.  

By Propositions \ref{P:w_n-marking} and \ref{P:disjoint-curves-u_n-v_n-w_n}, 
$v_n^+= v_n^-$ on $1/\ga_n-1+i[0, \infty)$. 
This implies that (M3) and (M4) do not contradict.
However, because $v_n^+(1/\ga_n-1+i[-1, 0)) \cap v_n^-(1/\ga_n-1+i[-1, 0))= \emptyset$, $\Omega_n$ may 
not be injective. 
For this reason, maps from $M_n^j$ to $\C{M}_n^j$ matching the markings must be 
multivalued (hence the reason for working with maps from $\C{M}_n^j$ to $M_n^j$). 
As we shall see in a moment, this does not cause any problems, since by \refL{L:tails-disjoint-CM_n} 
the curves $v_n^+(1/\ga_n-1+i[-1, 0))$ and $v_n^-(1/\ga_n-1+i[-1, 0))$ do not meet $\C{M}_n$. 

\begin{propo}\label{P:U_n^0}
There is a constant $C_8$ such that for every $n\geq 0$ there exists a continuous and surjective map 
\[\gO_n^0: \C{M}_n^0 \to M_n^0\] 
which matches $(u_n, v_n^{\pm}, w_n^+)$, and for all $z \in \C{M}_n^0$, $|\gO_n^0(z)-z| \leq C_8$.
\end{propo}

In light of \refL{L:tails-disjoint-CM_n}, we may choose the map $\gO_n^0$ in the above proposition to 
be injective on $\C{M}_n$. But this will not be needed for the overall argument to work in this section. 

\begin{proof}
Recall that $\C{M}_n^0$ is bounded by the curves $u_n$, $w_n^+$, and a continuous curve $\gm_n$ 
which connects $u_n(-i)$ to $w_n^+(1/\ga_n-i)$. 
By Propositions \ref{P:near-identity-u_n} and \ref{P:near-translation-w_n}, the maps $u_n$, $v_n^{\pm}$ and 
$w_n^+$ are uniformly close to the identity map. 
By \refT{T:Ino-Shi2}, $\mu_n$ may also be parameterised to be uniformly close to the identity map. 
These imply that one may extend the inverses of these maps to a continuous surjective map from $\C{M}_n^0$ to 
$M_n^0$. 
This may be carried out by partitioning the sets $M_n^0$ and $\C{M}_n^0$ into Jordan domains with uniformly 
bounded diameters, and mapping the corresponding pieces one to another, while respecting the boundary conditions. 
We present more details below. 

The curves $v_n^{\pm}$ lie in $\C{M}_n^0$, and $v_n^{\pm}(1/\ga_n-1-i)$ belong to $\gm_n$. 
By \refP{P:disjoint-curves-u_n-v_n-w_n}, $v_n^{\pm}$ are disjoint from $w_n^+$ and $u_n$, and moreover 
$u_n +\mathbb{Z}$ is disjoint from $v_n^{\pm}$ and $w_n^+$.
Recall that $\{w\in \D{C} \mid   0 \leq \Re w \leq 1/\ga_n-c_1\}$ is contained in $\C{M}_n^0$, where $c_1$ is the 
constant in \refP{P:petal-geometry}.   
Thus, for integers $j$ with $0 \leq j \leq 1/\ga_n-c_1-1$, the curves $u_n+j$ are contained in $\C{M}_n^0$ and lie 
on the left hand side of $v_n^+$. Moreover, by the definition of $\gm_n$, all those curves $u_n+j$ meet the 
curve $\gm_n$. 
Let us define $l_n$ as the largest integer less than or equal to $1/\ga_n-c_1-1$. 
The curves $u_n+j$, for $0\leq j \leq l_n$, $v_n^+$, and $w_n^+$ partition $\C{M}_n^0$ into $l_n+2$ pieces, 
say $A_{n, k}^0$, for $1\leq k \leq l+2$. 
Each $A_{n, k}^0$ is uniformly close to a half-infinite vertical strip of width one. 
Next, we divide each $A_{n, k}^0$ into infinitely many nearly-square Jordan domains with 
uniformly bounded diameters. To see this, we note that for each $A_{n, k}^0$, with ``vertical'' boundary curves, 
say $u_n+j$ and $u_n+j+1$, and any parameter $t \in [-1, +\infty)$, the points $u_n(it)+j$ and $u_n(it)+j+1$ may be 
connected by a path in $A_{n, k}^0$ which has a uniformly bounded diameter. 
We may use countably many such paths to partition each $A_{n, k}^0$ into Jordan domains with uniformly bounded 
diameters.
\end{proof}

\subsection{Lifting partial uniformisations}
\begin{propo}\label{P:lifting-interpolations}
Let $n\geq 1$ and $j\geq 0$. Assume that $\gO_n^j: \C{M}_n^j \to M_n^j$ is a continuous and surjective map 
which matches 
$(u_n,v_n^{\pm},w_n^+)$.
Then, there exists a continuous and surjective map 
\[\gO_{n-1}^{j+1}: \C{M}_{n-1}^{j+1} \to M_{n-1}^{j+1}\] 
which matches $(u_{n-1}, v_{n-1}^{\pm}, w_{n-1}^+)$, and for all integers $l$ satisfying 
$(\gep_n+1)/2 \leq l \leq a_{n-1}+ (\gep_n-1)/2$, 
\[\gO_{n-1}^{j+1} \circ (\gU_{n}+l) = Y_n \circ \gO_n^j+l,\]
whenever both sides of the equation are defined. 
\end{propo}

See \refF{F:lifting-partial-uniformisations} for an illustration of the proof of \refP{P:lifting-interpolations}.

\begin{figure}[ht]
\begin{center}
\begin{pspicture}(0.8,-.5)(13.8,10)
\pscustom[linecolor=cyan,fillstyle=solid,fillcolor=cyan]
{
\pscurve(1,4)(1,2)(.9,1.1)(1,1)(1.1,.9)(1,0) 
\psline(4,0)(4,4)
}

\pscustom[linecolor=cyan,fillstyle=solid,fillcolor=cyan]
{
\pscurve(5,4)(5,2)(4.9,1.2)(5,1)
\pscurve(5,1)(4.8,.85)(4.7,.3)
\pscurve(4.7,.3)(4.3,.35)(4,.2)
\psline(4,.2)(4,4)
}

\pscustom[linecolor=orange,fillstyle=solid,fillcolor=orange]
{
\pscurve(5,4)(5,2)(4.9,1.2)(5,1)
\pscurve(5,1)(5.2,1.05)(5.3,.3)
\pscurve[liftpen=1](5.7,.3)(5.8,.85)(6,1)
\pscurve(6,1)(5.9,1.2)(6,2)(6,4)
}

\pspolygon[linecolor=cyan,fillstyle=solid,fillcolor=cyan](8,0)(12,0)(12,4)(8,4)  
\pspolygon[linecolor=orange,fillstyle=solid,fillcolor=orange](12,0)(13,0)(13,4)(12,4)

\pscustom[origin={0,6},linecolor=cyan,fillstyle=solid,fillcolor=cyan]
{
\pscurve(1,4)(1,2)(.95,1.1)(1,1)(1.05,.9)(1,0) 
\pscurve(1,0)(1.22,1.55)(1.25,1.58)(1.28,1.55)(1.48,0.31)(1.5,0.3)
\pscurve[liftpen=1](1.5,.3)(1.6,.8)(1.65,.97)(1.7,1)
\pscurve[liftpen=1](1.7,1)(1.65,1.2)(1.7,2)(1.7,4)
}

\pscustom[origin={0,6},linecolor=orange,fillstyle=solid,fillcolor=orange]
{
\pscurve(2,4)(2,2)(1.95,1.1)(2,1)(2.05,.9)(2,0) 
\pscurve[liftpen=1](1.8,.3)(1.75,.97)(1.7,1)
\pscurve[liftpen=1](1.7,1)(1.65,1.2)(1.7,2)(1.7,4)
}

\pscustom[origin={3.5,6},linecolor=cyan,fillstyle=solid,fillcolor=cyan]
{
\pscurve(1,4)(1,2)(.95,1.1)(1,1)(1.05,.9)(1,0) 
\pscurve(1,0)(1.22,1.55)(1.25,1.58)(1.28,1.55)(1.48,0.31)(1.5,0.3)
\pscurve[liftpen=1](1.5,.3)(1.6,.8)(1.65,.97)(1.7,1)
\pscurve[liftpen=1](1.7,1)(1.65,1.2)(1.7,2)(1.7,4)
}

\pscustom[linecolor=cyan,fillstyle=solid,fillcolor=cyan]
{
\psline(8,10)(8,6) 
\pscurve[liftpen=1](8,6)(8.4,7)(8.8,6)
\psline[liftpen=1](8.8,6)(8.8,10)
}

\pscustom[linecolor=orange,fillstyle=solid,fillcolor=orange]
{
\psline(9,10)(9,6) 
\pscurve[liftpen=1](9,6)(8.9,5.9)(8.8,6)
\psline[liftpen=1](8.8,6)(8.8,10)
}

\pscustom[origin={3.2,0},linecolor=cyan,fillstyle=solid,fillcolor=cyan]
{
\psline(8,10)(8,6) 
\pscurve[liftpen=1](8,6)(8.4,7)(8.8,6)
\psline[liftpen=1](8.8,6)(8.8,10)
}

\psdot[dotsize=1pt](1,1)
\pscurve[linewidth=.4pt](1,4)(1,2)(.9,1.1)(1,1)(1.1,.9)(1,0)

\pscurve[linewidth=.4pt](5,4)(5,2)(4.9,1.2)(5,1)
\pscurve[linewidth=.4pt](5,1)(4.8,.85)(4.7,.3)
\pscurve[linewidth=.4pt](5,1)(5.2,1.05)(5.3,.3)

\pscurve[linewidth=.4pt](6,4)(6,2)(5.9,1.2)(6,1)
\pscurve[linewidth=.4pt](6,1)(5.8,.85)(5.7,.3)

\psline[linewidth=.4pt]{->}(7.6,.3)(13.6,.3)
\psline[linewidth=.4pt](8,0)(8,4)
\psline[linewidth=.4pt](12,0)(12,4)
\psline[linewidth=.4pt](13,0)(13,4)

\pscurve[linewidth=.4pt,origin={0,6}](1,4)(1,2)(0.95,1.1)(1,1)(1.05,.9)(1,0) 
\pscurve[linewidth=.4pt,origin={0,6}](2,4)(2,2)(1.95,1.1)(2,1)(2.05,.9)(2,0)
\pscurve[linewidth=.4pt,origin={0,6}](4.5,4)(4.5,2)(4.45,1.1)(4.5,1)(4.55,.9)(4.5,0)

\pscurve[linewidth=.4pt,origin={2.5,6}](1.5,.3)(1.6,.8)(1.65,.97)(1.7,1) 
\pscurve[linewidth=.4pt,origin={2.5,6}](1.7,1)(1.65,1.2)(1.7,2)(1.7,4)
\pscurve[linewidth=.4pt,origin={2.5,6}](1.8,.3)(1.75,.97)(1.7,1)

\pscurve[linewidth=.4pt,origin={3.5,6}](1.5,.3)(1.6,.8)(1.65,.97)(1.7,1)
\pscurve[linewidth=.4pt,origin={3.5,6}](1.7,1)(1.65,1.2)(1.7,2)(1.7,4)
\pscurve[linewidth=.4pt,origin={3.5,6}](1.8,.3)(1.75,.97)(1.7,1)

\psline[linewidth=.4pt,origin={0,6}]{->}(7.6,.3)(12.6,.3)
\psline[linewidth=.4pt,origin={0,6}](8,0)(8,4)
\psline[linewidth=.4pt,origin={0,6}](11,0)(11,4)
\psline[linewidth=.4pt,origin={0,6}](12,0)(12,4)

\rput(8,-.3){\small $0$}
\rput(12,-.3){\small $\frac{1}{\alpha_n}-1$}
\rput(13,-.3){\small $\frac{1}{\alpha_n}$}

\psline[linewidth=1.5pt,linecolor=gray]{->}(7.9,2)(1,2)
\rput(4,2.2){\small $u_n$}

\psline[linewidth=1.5pt,linecolor=gray]{->}(11.9,2.5)(5.1,2.5)
\rput(8.5,2.7){\small $v_n^\pm$}

\psline[linewidth=1.5pt,linecolor=gray]{->}(12.9,3)(6.1,3)
\rput(9.5,3.2){\small $w_n^+$}

\psline[linewidth=1.5pt,linecolor=gray,origin={0,6}]{->}(1.1,2.7)(2,2.7)
\rput(1.5,8.9){\small $+1$}

\psline[linewidth=1.5pt,linecolor=gray,origin={0,6}]{->}(1.05,3.9)(4.45,3.9)
\rput(3,10.2){\small $+a_{n-1}+\epsilon_n$}

\rput(8,5.7){\small $0$}
\rput(11,5.7){\small $\frac{1}{\alpha_{n-1}}-1$}
\rput(12,5.7){\small $\frac{1}{\alpha_{n-1}}$}

\psline[linewidth=1.5pt,linecolor=gray,origin={0,6}]{->}(7.9,2)(1,2)
\rput(3,8.2){\small $u_{n-1}$}

\psline[linewidth=1.5pt,linecolor=gray,origin={0,6}]{->}(10.9,2.5)(4.2,2.5)
\rput(6,8.7){\small $v_{n-1}^\pm$}

\psline[linewidth=1.5pt,linecolor=gray,origin={0,6}]{->}(11.9,3)(5.2,3)
\rput(8.5,9.2){\small $w_{n-1}^\pm$}

\psline[linewidth=1.5pt,linecolor=gray]{->}(1,4.1)(1,5.9)
\rput(1.3,5){\small $\Upsilon_n$}

\psline[linewidth=1.5pt,linecolor=gray]{->}(6,4.1)(2,5.9)
\rput(4,4.7){\small $\Upsilon_n$}

\psline[linewidth=1.5pt,linecolor=gray]{->}(8,4.1)(8,5.4)
\rput(8.3,5){\small $Y_n$}

\psline[linewidth=1.5pt,linecolor=gray]{->}(5,4.1)(5.2,6.2)
\rput(5.6,5.3){\small $\Upsilon_n +a_{n-1}+\epsilon_n$}
\end{pspicture}
\caption{Illustration of the markings and commutative relations, when $\epsilon_n=-1$.}
\label{F:lifting-partial-uniformisations}
\end{center}
\end{figure}

\begin{proof}
Fix an arbitrary $n\geq 1$ and $j\geq 0$. Let us first assume that $\gep_n=-1$. 
Recall from \refE{E:M_n^j--1} and \eqref{E:CM_n^j--1} that 
\[M_{n-1}^{j+1} 
=\bigcup_{l=0}^{a_{n-1}-2} \big( Y_n(M_n^j)+ l \big) \bigcup \big(Y_n(K_n^j)+ a_{n-1}-1\big).\]
and 
\[\C{M}_{n-1}^{j+1} 
=\bigcup_{l=0}^{a_{n-1}-2} \big(\gU_n(\C{M}_n^j)+ l \big) \bigcup \big(\gU_n(\C{K}_n^j)+a_{n-1}-1\big).\]
For each $0 \leq l \leq a_{n-1}-2$, define $\gO_{n-1}^{j+1}: \gU_n(\C{M}_n^j)+l \to Y_n(M_n^j)+ l$ as 
\[\gO_{n-1}^{j+1}(z)= Y_n\circ \gO_n^j \circ \gU_n^{-1}(z-l)+l.\]  
Since $Y_n$, $\gO_n^j$ and $\gU_n$ are continuous, the above map is continuous. As $\gO_n^j$ is surjective, 
$\gO_{n-1}^{j+1}$ covers $Y_n(M_n^j)+l$. 
Similarly, we define $\gO_{n-1}^{j+1}: \gU_n(\C{K}_n^j)+a_{n-1}-1 \to Y_n(K_n^j)+ a_{n-1}-1$ as 
\[\gO_{n-1}^{j+1}(z)= Y_n \circ \gO_n^j \circ \gU_n^{-1}(z-a_{n-1}+1)+ a_{n-1}-1.\]  
Since $\gO_n^j :\C{M}_n^j \to M_n^j$ matches $u_n$ and $v_n^+$, it follows that $\gO_n^j: \C{K}_n^j \to K_n^j$ is continuous and surjective.

We need to show that $\gO_{n-1}^{j+1}$ is well-defined on the common 
boundaries of $\gU_n(\C{M}_n^j)+l$ and $\gU_n(\C{M}_n^j)+l+1$, for integers $l$ with $0 \leq l \leq a_{n-1}-2$. 
Fix an arbitrary $z$ on the common boundary. There is $t' \geq \Im Y_n(-i)$ such that $z=u_{n-1}(it')+l+1$. 
Let $it'=Y_n(it)$, for some $t\geq -1$, (see \refE{E:invariant-imaginary-line}). 
The map induced from the left-hand component provides 
\begin{align*}
\gO_{n-1}^{j+1}(z)&=Y_n \circ \gO_n^j \circ \gU_n^{-1}(u_{n-1}(it') + l+1-l)+l &&  \\ 
&= Y_n \circ  \gO_n^j \circ w_n^+(1/\ga_n+ Y_n^{-1}(i t')) + l && (\text{by (E2)})  \\
&= Y_n (1/\ga_n+Y_n^{-1}(i t'))+l && (\text{by (M2)}) \\
&=i t'+1+l. && (\text{\refE{E:Y_n-comm-1}})
\end{align*}
The map induced from the right-hand component provides 
\begin{align*}
\gO_{n-1}^{j+1}(z)&=Y_n \circ  \gO_n^j \circ \gU_n^{-1}(u_{n-1}(it')+l+1 - l-1)+l+1  &&\\ 
&=Y_n \circ \gO_n^j \circ u_n(Y_n^{-1}(i t'))+l+1 && (\text{by (E1)}) \\ 
&=Y_n(Y_n^{-1}(i t')) +l +1 && (\text{by (M1)})  \\
&=i t'+l+1. && 
\end{align*}
Thus, the two induced maps are identical on the common boundaries.  

By the definition of $\gO_{n-1}^{j+1}$, the functional equations in the proposition hold. 
There remains to show that $\gO_{n-1}^{j+1}$ matches $(u_{n-1}, v_{n-1}^{\pm}, w_{n-1}^+)$. 

$\bullet$ $\gO_{n-1}^{j+1}$ matches $u_{n-1}$: Let $z=u_{n-1}(it')= Y_n(it)$ for some $t' \geq \Im Y_n(-i)$ and $t\geq -1$. 
By the definition of $\gO_{n-1}^{j+1}$ (use $l=0$), we have 
\begin{align*}
\gO_{n-1}^{j+1}(u_{n-1}(it')) =Y_n \circ \gO_n^j \circ \gU_n^{-1}(u_{n-1}(it')) 
= Y_n \circ \gO_n^j \circ u_n(Y_n^{-1}(it')) = Y_n(Y_n^{-1}(it'))= it'.
\end{align*}
In the above equation we have employed (E1) and  (M1). 

$\bullet$ $\gO_{n-1}^{j+1}$ matches $w_{n-1}^+$: Let $z=w_{n-1}^+(it'+1/\ga_{n-1})$ for some $t' \geq -1$.
As $w_{n-1}^+(it'+1/\ga_{n-1})$ belongs to $\gU_n(\C{K}_n^j)+a_{n-1}-1$, by the definition of $\gO_{n-1}^{j+1}$, 
\begin{align*}
\gO_{n-1}^{j+1}(z) 
&=Y_n \circ \gO_n^j \circ \gU_n^{-1}(w_{n-1}^+(it'+ 1/\ga_{n-1}) - (a_{n-1}-1)) + a_{n-1}-1 && \\
&= Y_n \circ \gO_n^j \circ v_n^+ (1/\ga_n -1+Y_n^{-1}(i t')) + a_{n-1}-1 && (\text{by (E3)}) \\
&=Y_n (1/\ga_n-1+Y_n^{-1}(it')) + a_{n-1}-1 && (\text{by (M3)})  \\
&= Y_n(Y_n^{-1}(i t')) + 1-\ga_n + a_{n-1}-1 && (\text{\refE{E:Y_n-comm-2}})  \\
&=it'+1/\ga_{n-1}.  && (\eqref{E:rotations-relations})
\end{align*}

$\bullet$ $\gO_{n-1}^{j+1}$ matches $v_{n-1}^+$: Let $z=v_{n-1}^+(it'+1/\ga_{n-1}-1)$ for some $t' \geq -1$. 
Using $\gO_{n-1}^{j+1}(z+1)=\gO_{n-1}^{j+1}(z)+1$, we have 
\begin{align*}
\gO_{n-1}^{j+1} (v_{n-1}^+(it'+1/\ga_{n-1}-1)) &= \gO_{n-1}^{j+1} (v_{n-1}^+(it'+1/\ga_{n-1}-1)+1)-1  && \\
&= \gO_{n-1}^{j+1} (w_{n-1}^+(it'+1/\ga_{n-1}))-1  && \\
&= it'+1/\ga_{n-1}-1.    &&  (\text{by (M2)})
\end{align*}

$\bullet$ $\gO_{n-1}^{j+1}$ matches $v_{n-1}^-$: Let $z=v_{n-1}^-(it'+1/\ga_{n-1}-1)$ and $i t'= Y_n(it)$ for 
some $t' \geq \Im Y_n(-i)$ and $t\geq -1$. 
Since $\gep_n=-1$, $1/\ga_{n-1}= a_{n-1}-\ga_n$, and hence $1/\ga_{n-1}-1$ lies strictly between $a_{n-1}-2$ and $a_{n-1}-1$. 
By the definition of $\gO_{n-1}^{j+1}$ (here use $l=a_{n-1}-2$), we have 
\begin{align*}
\gO_{n-1}^{j+1}(v_{n-1}^-(& it'+1/\ga_{n-1}-1)) && \\
&=Y_n \circ \gO_n^j \circ \gU_n^{-1}(v_{n-1}^-(it'+ 1/\ga_{n-1}-1) - (a_{n-1}-2)) + a_{n-1}-2 && \\
&=Y_n \circ \gO_n^j \circ \gU_n^{-1}(w_{n-1}^-(it'+ 1/\ga_{n-1}) - (a_{n-1}-1)) + a_{n-1}-2 && \\
&=Y_n \circ \gO_n^j \circ v_n^-(Y_n^{-1}(it') + 1/\ga_n-1)) + a_{n-1}-2 
&& (\text{by (E3)})  \\
&=Y_n (Y_n^{-1}(it') + 1/\ga_n-1)) + a_{n-1}-2 && (\text{by (M4)})  \\
&=Y_n(Y_n^{-1}(i t') + 1-\ga_n +  a_{n-1}-2   && (\text{\refE{E:Y_n-comm-2}})   \\
&= it' +1/\ga_{n-1} -1. && (\eqref{E:rotations-relations})
\end{align*}
This completes the proof of the proposition when $\gep_n=-1$. 
The proof for the case $\gep_n=+1$ is similar, so we briefly explain the two parts which require attention. 

$\bullet$ $\gO_{n-1}^{j+1}$ matches $w_{n-1}^+$: Let $z=w_{n-1}^+(it'+1/\ga_{n-1})$ for some $t' \geq -1$.
Note that $w_{n-1}^+(it'+1/\ga_{n-1}) \in \gU_n(\C{J}_n^j)+a_{n-1}+1$. Then, by the definition of $\gO_{n-1}^{j+1}$, 
\begin{align*}
\gO_{n-1}^{j+1}&(w_{n-1}^+(it'+1/\ga_{n-1})) && \\
&=Y_n \circ \gO_n^j \circ \gU_n^{-1}(w_{n-1}^+(it'+ 1/\ga_{n-1}) - (a_{n-1}+1)) + a_{n-1}+1 && \\
&= Y_n \circ \gO_n^j \circ v_n^- (Y_n^{-1}(i t')+ 1/\ga_n -1) + a_{n-1}+1 && (\text{by (E3)}) \\
&=Y_n (Y_n^{-1}(it')+ 1/\ga_n-1) + a_{n-1}+1 &&( \text{by (M4)})  \\
&= Y_n(Y_n^{-1}(i t') + \ga_n -1 + a_{n-1}+1 && (\text{\refE{E:Y_n-comm-2}})  \\
&=it'+1/\ga_{n-1}.  && (\eqref{E:rotations-relations})
\end{align*}

$\bullet$ $\gO_{n-1}^{j+1}$ matches $v_{n-1}^-$: Let $z=v_{n-1}^-(it'+1/\ga_{n-1}-1)$ for some $t' \geq \Im Y_n(-i)$. 
Let $i t'= Y_n(it)$. 
Since $\gep_n=+1$, $1/\ga_{n-1}=a_{n-1}+\ga_n$. 
Thus, $1/\ga_{n-1}-1$ lies strictly between $a_{n-1}-1$ and $a_{n-1}$, and hence 
\begin{align*}
\gO_{n-1}^{j+1}&(v_{n-1}^-(it'+1/\ga_{n-1}-1)) && \\
&=Y_n \circ \gO_n^j \circ \gU_n^{-1}(v_{n-1}^-(it'+ 1/\ga_{n-1}-1) - a_{n-1}) + a_{n-1} && \\
&=Y_n \circ \gO_n^j \circ \gU_n^{-1}(w_{n-1}^-(it'+ 1/\ga_{n-1}) - (a_{n-1}+1)) + a_{n-1} && \\
&=Y_n \circ \gO_n^j \circ v_n^+(Y_n^{-1}(it') + 1/\ga_n-1)) + a_{n-1}
&& (\text{by (E3)})  \\
&=Y_n (Y_n^{-1}(it') + 1/\ga_n-1) + a_{n-1} && \text{(by (M3))})  \\
&=Y_n(Y_n^{-1}(i t') + \ga_n -1 +  a_{n-1} && (\text{\refE{E:Y_n-comm-2}})   \\
&= it' +1/\ga_{n-1} -1. && (\eqref{E:rotations-relations})   \qedhere     
\end{align*}
\end{proof}

\subsection{Convergence of the partial uniformisations}\label{SS:convergence-partial-uniformisation}
Fix an arbitrary $n\geq 0$. By \refP{P:U_n^0}, for each $j\geq 0$, there is a continuous and surjective map 
$\gO_{n+j}^0: \C{M}_{n+j}^0 \to M_{n+j}^0$ which matches $(u_{n+j}, v_{n+j}^{\pm}, w_{n+j}^+)$. 
Inductively applying \refP{P:lifting-interpolations}, we obtain a continuous and surjective map
\[\gO_n^j: \C{M}_n^j \to M_n^{j}.\]

\begin{propo}\label{P:U_n^j-Cauchy}
There is $C_9 \in \mathbb{R}$ such that for every $n\geq 0$ and every $j\geq 0$, on $\C{M}_n^{j+1}$, 
\[\big |\gO_n^{j+1} - \gO_n^{j} \big | \leq C_9 (0.9)^j.\]
In particular, for each $n\geq 0$, as $j \to \infty$, the sequence $\gO_n^j$ converges to a continuous map 
\[\gO_n: \C{M}_n \to M_n.\]
\end{propo}

\begin{proof}
First we show that there is $C_9 \in \mathbb{R}$ such that $| \gO_{n}^1- \gO_n^0| \leq C_9$ on $\C{M}_n^1$, 
for all $n \geq 0$. 
Fix an arbitrary $n\geq 0$. Let $z_n \in \C{M}_n^{1}$. There is an integer $l_n$ such that 
$z_n- l_n \in \gU_{n+1}(\C{M}_{n+1}^0)$ and $z_{n+1}= \gU_{n+1}^{-1}(z_n-l_n)$ is defined. 
Note that $\gO_n^0: \C{M}_n^0 \to M_n^0$ and $\gO_{n}^1: \C{M}_{n}^1 \to M_{n}^1$. 
As $\C{M}_{n}^1 \subset \C{M}_{n}^0$, both maps are defined on $\C{M}_{n}^1$. 
By \refT{T:Ino-Shi2}, there is a uniform constant $C\geq 0$, independent of $n$, such that for all 
$z \in \C{M}_n^0$, $\Im z +C \geq -1$. 
Employing \refP{P:gc-vs-Y-value} with $w_1= \gU_{n+1}^{-1}(z_{n}-l_{n})+C$ and $w_2=\gU_{n+1}^{-1}(z_{n}-l_{n})$, 
and then using \refP{P:coordinate-properties}-(iv) and \refP{P:U_n^0}, we obtain 
\begin{align*}
|\gO_{n}^1(z_{n})  - z_{n} | 
&=\left |\left(Y_{n+1} \circ \gO_{n+1}^0 \circ \gU_{n+1}^{-1}(z_{n}-l_{n}) + l_{n} \right)  - z_{n} \right |   \\
&\leq \left | \left(Y_{n+1} \circ \gO_{n+1}^0 (w_2) + l_{n} \right) - \left(Y_{n+1} (w_1) + l_{n} \right)\right | \\
&  \qquad + \left| \left(Y_{n+1} (w_1) + l_{n} \right) - ( \gU_{n+1}(w_2)+l_{n})\right| \\
&\leq 0.9 \cdot \left | \gO_{n+1}^0(w_2) -w_1 \right |  + C_3 \max\{C, 1\}. \\
&\leq \left | \gO_{n+1}^0(w_2) - w_2| + | w_2-w_1 \right |  + C_3 (C+1). \\
&\leq C_8+ C + C_3( C+1). 
\end{align*}
Therefore,
\begin{align*}
|\gO_{n}^1(z_{n}) - \gO_{n}^0(z_{n})| 
\leq |\gO_{n}^1(z_{n}) - z_{n} | + | z_{n}-\gO_{n}^0(z_{n})| \leq C_8+ C + C_3( C+1) + C_8.
\end{align*}
Let us introduce $C_9= 2 C_8 + C+ C_3( C+1)$. 

Now assume that the inequality holds on $\C{M}_{n}^{j}$, for some $j-1 \geq 0$ and all $n\geq 0$. 
For  $z_n \in \C{M}_n^{j+1}$, $\gU_{n+1}^{-1}(z_n-l_n) \in \C{M}_{n+1}^{j}$. 
Using \refP{P:coordinate-properties}-(iv), and the induction hypothesis, for all $n\geq 0$, 
\begin{align*}
\big| \gO_n^{j+1}(z_n) - \gO_n^j(z_n)\big| & = 
\big| Y_{n+1} \circ \gO_{n+1}^{j} \circ \gU_{n+1}^{-1}(z_n-l_n)-Y_{n+1}\circ \gO_{n+1}^{j-1}\circ \gU_{n+1}^{-1}(z_n-l_n)\big| \\
& \leq 0.9 \cdot C_9 (0.9)^{j-1}= C_9 (0.9)^j.
\end{align*}

The sequence $(\gO_n^j)_{j\geq 0}$ is uniformly Cauchy on $\C{M}_n \subset \C{M}_n^j$. It converges to a continuous map.
\end{proof}

\begin{cor}\label{C:commu-U_n}
For every $n\geq 1$ and all integers $l$ satisfying $(\gep_{n}+1)/2 \leq l \leq a_{n-1} +(\gep_n-1)/2$, 
\[\gO_{n-1} \circ (\gU_{n}+l) = Y_n \circ \gO_n+l,\]
whenever both sides of the equation are defined.
\end{cor}

\begin{proof}
This follows from \refP{P:U_n^j-Cauchy} and the functional relation in \refP{P:lifting-interpolations}. 
\end{proof}

We may define $\gO_{-1}: \C{M}_{-1} \to M_{-1}$, as 
\[\gO_{-1}(z)= Y_0 \circ \gO_0 \circ \gU_0^{-1} (z - (\gep_0+1)/2) + (\gep_0+1)/2.\]
By \refP{P:w_n-marking}, and since $\gO_0$ satisfies (M2), for $t \geq 0$, 
\[\gO_{-1}(it+1)= \gO_{-1}(it)+1.\]

\begin{propo}\label{P:near-identity-U_n}
There is a constant $C_{10}$ such that for every $n\geq -1$ and every $z \in \C{M}_n$, 
\[|\gO_n(z)-z| \leq C_{10}.\] 
\end{propo}

\begin{proof}
By Propositions~\ref{P:U_n^0} and \ref{P:U_n^j-Cauchy}, for all $n\geq -1$, all $j\geq 0$, and all $z \in \C{M}_n$ 
we have 
\begin{equation*}
\big|\gO_n^j(z) - z \big| 
\leq \textstyle{\sum_{l=0}^{j-1}} \big| \gO_n^{l+1}(z)- \gO_n^l(z) \big| +  \big|  \gO_n^0(z) - z \big |
\leq \textstyle{\sum_{l=0}^{+\infty}} C_9 (0.9)^l + C_8.     \qedhere  
\end{equation*} 
\end{proof}

\begin{propo}\label{P:injective-U_n}
For every $n\geq -1$, $\gO_n: \C{M}_n \to M_n$ is injective. 
\end{propo}

\begin{proof}
Fix an arbitrary $n\geq -1$. Let $z_n \neq z'_n$ be arbitrary points in $\C{M}_n$. 
Since $\C{M}_n= \cap_{j\geq 0} \C{M}_n^j$, we may inductively identify integers $l_i$ and $l'_i$ such that 
$z_i- l_i$ and $z'_i- l'_i$ belong to $\gU_{i+1}(\C{M}_{i+1}) \setminus u_{i+1}(i[-1, +\infty))$, and 
$z_{i+1}=\gU_{i+1}^{-1}(z_i-l_i)$ and $z'_{i+1}= \gU_{i+1}^{-1}(z'_i-l'_i)$ are defined. 

If there is a smallest $i \geq n$ such that $l_i \neq l_i'$, then $\gO_i(z_i) \neq \gO_i(z_i')$. 
Using the commutative relations in Corollary~\ref{C:commu-U_n}, one concludes that 
$\gO_n(z_n) \neq \gO_n(z_n')$. 

Now assume that for all $i\geq n$, $l_i=l_i'$. 
Recall the sets $\tilde{\C{M}}_n^0$ and the hyperbolic metrics $\varrho_n$ defined in \refS{SS:uniform-contraction}. 
By the uniform contraction of the maps $\gU_i+l_{i-1}$ in \refP{P:uniform-contraction-gU_n}, 
we must have $d_{\varrho_i}(z_i, z_i') \to +\infty$ as $i \to +\infty$. 
By \refL{L:extension-gU_n}, $\C{M}_{n+j}^1$ is well-contained in $\tilde{\C{M}}_{n+1}^0$. This implies that 
$|z_i - z_i'| \to +\infty$, as $i \to +\infty$. 
In particular, there is $i\geq n$, such that $|z_i - z_i'| > 2 C_{10}$. 
By virtue of \refP{P:near-identity-U_n}, we must have $\gO_i(z_i) \neq \gO_i(z_i')$. 
Using the functional relations in Corollary~\ref{C:commu-U_n} as well as the injectivity of $Y_l$ on $M_l^0$ 
and $\gU_{l}$ on $\C{M}_l^0$, we must have $\gO_n(z_n) \neq \gO_n(z_n')$.
\end{proof}

\begin{propo}\label{P:U_n-onto}
For every $n \geq -1$, $\gO_n: \C{M}_n \to M_n$ is surjective.
\end{propo}

\begin{proof}
Fix an arbitrary $n\geq -1$, and $z \in M_n$. As $M_n= \cap_{j\geq 0} M_n^j$ and for each $j\geq 0$, 
$\gO_n^j: \C{M}_n^j \to M_n^j$ is surjective, there is $z_j \in \C{M}_n^j$ with $\gO_n^j(z_j)= z$. 
By \refP{P:near-identity-U_n}, $z_j$ are contained in a bounded subset of $\C{M}_n^0$. 
Thus, there is a subsequence, say $z_{j_k}$, for $k \geq 1$, which converges to some $z' \in \C{M}_n^0$. 
However, because $\cap_{j\geq 0} \C{M}_n^j$ is a nest of closed sets, we must have $z' \in \C{M}_n$. 
Then, the uniform convergence of $\gO_n^j$ to $\gO_n$ implies that $\gO_n(z')=z$. 
\end{proof}

\subsection{Proofs of the main theorems, and corollaries}\label{SS:harvest}
Recall the straight topological model $A_\ga = \partial \hat{A}_\ga$ and the model map 
$T_\ga: \hat{A}_\ga \to \hat{A}_\ga$ from \refS{S:arithmetic-model}, 
and the closed set $\C{A}_f =\partial \hat{\C{A}}_f$ from \refS{SS:capturing-pc}.
In this section we show that $f: \gL(c_f) \to \gL(c_f)$ is topologically conjugate to $T_\ga: A_\ga \to A_\ga$, 
and transfer the features of the latter system to the former one. 

\begin{thm}\label{T:pc-model-related}
There is $N \in \mathbb{N}$ such that for every $\ga \in \irr$ and every $f\in \QIS_\ga$, there is a 
homeomorphism $\gY_f: \hat{\C{A}}_f  \to \hat{A}_\ga$ which satisfies $T_\ga \circ \gY_f= \gY_f \circ f$ on 
$\hat{\C{A}}_f$.
Moreover, $\gY_f (\gL(c_f))= A_\ga$.
\end{thm}

\begin{proof}
Let $N$ be the integer in \refP{P:sequence-renormalisations}. 
For $\ga \in \irr$ and $f\in \QIS_\ga$, by following the constructions in Sections \ref{S:parametrised curve}, 
\ref{S:modified nest} and \ref{S:uniformisation}, we obtain the homeomorphism $\gO_{-1}:\C{M}_{-1} \to M_{-1}$, 
which satisfies $\gO_{-1}^{-1}(it+1)=\gO_{-1}^{-1}(it)+1$, for $t\geq 0$. 
Then, $U_{-1}$ induces a homeomorphism $\gY_f: \hat{\C{A}}_f \to \hat{A}_\ga$.

By \refP{P:pc-covered}, $f(\hat{\C{A}}_f) \subseteq \hat{\C{A}}_f$. 
By \refP{P:iterates-lifts-one-step}, the definition of $T_\ga$, and \refC{C:commu-U_n}, 
$T_\ga \circ \gY_f= \gY_f \circ f$ on $\hat{\C{A}}_f$.
Indeed, one only needs to verify the conjugacy on the set of points which satisfy item (i) in the definition of $T_\ga$, 
because the set of such points is dense in $\hat{A}_\ga$. 

By \refP{P:pc-covered}, $\gL(c_f) \subseteq \hat{\C{A}}_f$. 
Since $U_{-1}(+1)=0$, $\gY_f$ maps the critical value of $f$ to $+1$ in $\hat{A}_\ga$. 
Then, $\gY_f$ maps $\gL(c_f)$ to the closure of the orbit of $+1$ in $\hat{A}_\ga$. 
On the other hand, by \refT{T:model-dynamics}, the orbit of $+1$ by $T_\ga$ is, contained in and, dense in $A_\ga$. 
Thus, $\gY(\gL(c_f))= A_\ga$. 
\end{proof}

\begin{proof}[Proofs of Theorems \ref{T:trichotomy-main-thm}, \ref{T:degenerations-invariant-curves} 
and \ref{T:dynamics-on-pc}]
Recall from \refS{S:NP-renormalisation-scheme} that the class of maps $\cup_{\ga \in \irr} \QIS_\ga$ is the class 
$\C{F}$ in the introduction. 
\refT{T:trichotomy-main-thm} follows from \refT{T:pc-model-related} and the trichotomy of $A_\ga$ in 
\refT{T:trichotomy-model}. 
Theorems \ref{T:degenerations-invariant-curves} and \ref{T:dynamics-on-pc} follow from \refT{T:pc-model-related} and 
\refT{T:model-dynamics}.
\end{proof}

We make the following connection to the set $\C{A}_f$ for future purposes. 

\begin{cor}
For every $\ga \in \irr$ and every $f\in \QIS_\ga$, $\gL(c_f)=\C{A}_f$.   
\end{cor}

\begin{proof}
Because $\hat{\C{A}}_f$ is closed, $\C{A}_f \subseteq \hat{\C{A}}_f$.
By \refT{T:pc-model-related} and the invariance of domain theorem, 
\[\gY_f(\C{A}_f) = \gY_f(\partial \hat{\C{A}}_f) = \partial \hat{A}_\ga= A_\ga = \gY_f(\gL(c_f)).\] 
Therefore, $\C{A}_f = \gL(c_f)$.
\end{proof}

Combining Theorems \ref{T:model-dynamics} and \ref{T:pc-model-related}, we obtain a proof of the following 
result in \cite[Thm 4.6]{AC18}.

\begin{cor}\label{C:injective-on-pc}
For every $\ga \in \irr$ and every $f \in \QIS_\ga$, $f: \gL(c_f) \to \gL(c_f)$ is injective.  
\end{cor}

Given a connected set $X \subset \mathbb{C}$, let us say that $x \in X$ is an \textbf{end point} 
of $X$ if $X \setminus \{x\}$ is connected. 

\begin{thm}\label{T:cv-endpoint}
For every $\ga\in \irr \setminus \E{H}$, every $f\in \QIS_\ga$, and every integer $k\geq 0$, 
$f\co{k}(c_f)$ is an end point of $\gL(c_f)$. 
\end{thm}

\begin{proof}
Recall that $\gY_f$ maps the critical value of $f$ to $1$ in $A_\ga$. 
By \refE{E:b_n(0)=0}, $1$ is an end point of $A_\ga$ when $\ga$ is an irrational number outside $\E{H}$. 
Then, by \refT{T:pc-model-related}, the critical value of $f$ is an end point of $\gL(c_f)$.
Then, by \refC{C:injective-on-pc}, every $f\co{k}(c_f)$ must be an end point of $\gL(c_f)$.
\end{proof}

As an immediate corollary of \refT{T:pc-model-related} we obtain the following. 

\begin{cor}\label{C:pc-topological-universality}
For every $\ga\in \irr$ and every $f$ and $g$ in $\QIS_\ga$ with $f'(0)=g'(0)=e^{2\pi i \ga}$, 
$\gY_g^{-1} \circ \gY_f: \gL(c_f) \to \gL(c_g)$ topologically conjugates $f$ with $g$, and satisfies 
$\gY_g^{-1} \circ \gY_f(c_f)=c_g$. 
\end{cor}

\begin{rem}
One may give a different proof of the above corollary using the holomorphic motion of the critical orbits of the maps, 
parametrised over the infinite dimensional complex manifold $\IS_\ga$. 
The proof of \refT{T:pc-model-related} only uses the compactness of $\IS_\ga$, 
and does not require any complex structure on $\IS_\ga$. 
\end{rem}

By  a general result in dynamics of polynomials \cite[thm 18.5]{Mi06}, having no critical point on the 
boundary of the Siegel disk makes the Julia set non locally connected; see also \cite{Ki00,KLN15}. 

\begin{cor}
For every $\ga \in \irr \cap (\E{B} \setminus \E{H})$ and every polynomial $P \in \QIS_\ga$, the Julia set of $P$ 
is not locally connected.
\end{cor}

In light of this, we make the following conjecture.  
\begin{conjecture}
For every $\ga \in \mathbb{R} \setminus \mathbb{Q}$, the Julia set of $Q_\ga$ is locally connected iff 
$\ga \in \E{H}$.
\end{conjecture}

In \cite{Blo10}, some progress is made in describing the topology of the Julia set of $Q_\ga$ when it is not 
locally connected. 

\begin{rem}
In \cite{PM97}, Perez-Marco introduced the notion of hedgehogs, or Siegel-compacta, for a general holomorphic map 
with an irrationally indifferent fixed point. That is a non-trivial local invariant compact set containing the fixed point. 
In general, a Siegel compacta may have a complicated topology, see for instance \cite{Che11}. 
But, if the holomorphic map is a restriction of an element of $\QIS_\ga$ to a neighbourhood of the fixed point, 
then the hedgehog may not have a complicated topology, provided $\ga \in \irr$. 
Indeed, using a general result of Ma{\~n}{\'e}, and the lack of expansion along orbits in a Siegel compacta, any Siegel 
compacta of a rational function must be contained in the post-critical set. 
This is true for any element of $\QIS_\ga$, \cite[Section 4.3]{AC18}. 
Thus, such Siegel compacta and hedgehogs must be one of the invariant sets presented in \refT{T:dynamics-on-pc}. 
Our work suggests that for rational functions of the Riemann sphere the hedgehogs and Siegel compacta have 
tame topologies. 
In contrast, for an arbitrary holomorphic germ of diffeomorphism of $(\D{C}, 0)$, this is far from true as it is shown in 
\cite{Bis16}. 
\end{rem}

\begin{rem}
Cantor bouquets also appear as the closure of the set of escaping points of the maps $\gl e^z$, for $0< \gl< 1/e$, 
see \cite{Dev84,AaOv93,Rem06}. 
The analysis of the renormalisation in this paper, among others, presents the similarity between these dynamical systems.
Naively speaking, that is due to the behaviour of the changes of coordinates in the renormalisation 
tower of a given map in $\QIS_\ga$. 
Each change of coordinate behaves like the $\log$ function below a certain horizontal line (while behaving like a 
linear map above that line). 
\end{rem}

Evidently, the analysis of the renormalisation scheme presented in this paper provides some geometric features of the post-critical 
set as well. For example, we have the following result. 

\begin{cor}\label{C:landing-angle}
For every non-Brjuno number $\ga \in \irr$ and every $f \in \QIS_\ga$, every connected component of 
$\Lambda(f) \setminus \{0\}$ lands at $0$ at a well-defined angle. 
\end{cor} 

\begin{proof}
By \refP{P:u_n-asymptotics}, for every $n\geq -1$, $\lim_{t\to +\infty} \Re u_n(it)$ exists and is finite. This implies that 
the connected component of $\Lambda(f) \setminus \{0\}$ which contains the critical value of $f$ lands at $0$ at a well-defined angle. 
Any iterate of this curve by $f$ lands at $0$ at a well-defined angle. 
Since the set of the angles of all those rays is dense on $\mathbb{R}/\mathbb{Z}$, any other component of 
$\Lambda(f) \setminus \{0\}$ must also land at $0$ at a well-defined angle. 
\end{proof}

\medskip

\textbf{Acknowledgement:}
The author acknowledges financial support from EPSRC of the UK - grant no EP/M01746X/1 - 
while working on this project. 
\bibliographystyle{smfalpha}
\bibliography{Data}
\end{document}